\documentclass[12pt, reqno]{amsart}
\usepackage{amssymb,xy,color, doi}
\usepackage[dvipsnames]{xcolor}
\xyoption{all}
\usepackage{ytableau,hyperref}
\usepackage[capitalise]{cleveref}
\hypersetup{
	colorlinks = true,
	linkcolor = blue,
	anchorcolor = blue,
	citecolor = blue,
	filecolor = blue,
	urlcolor = blue
}

\usepackage{caption}
\usepackage{amsmath}
\usepackage{MnSymbol}
\usepackage{tikz}
\usepackage{tkz-graph}
\tikzstyle{vertex}=[circle, draw, inner sep=0pt, minimum size=6pt]
\usetikzlibrary{shapes, cd}
\tikzcdset{every label/.append style = {font = \small}}

\usepackage{array, multirow, float}
\usepackage{relsize}

\usepackage[margin=1in]{geometry}                
\usepackage{diagbox}
\usetikzlibrary{matrix}
\usetikzlibrary{arrows.meta}
\usepackage{booktabs}
\usepackage{enumitem}   
\usepackage[normalem]{ulem}

\usepackage{comment}                        
\newcommand{\commentAlt}[1]{\ignorespaces}     
\newcommand{\commentLongAlt}[1]{\ignorespaces}  

\newcommand{\ZZ}{\mathbb Z}

\newcommand{\U}{\mathcal{U}} 
\newcommand{\LAT}{\mathsf{LAT}} 
\newcommand{\SVT}{\mathsf{SVT}} 
\newcommand{\SSYT}{\mathsf{SSYT}}
\newcommand{\GT}{\mathsf{GT}} 
\newcommand{\MGT}{\mathsf{MGT}} 

\newcommand{\states}{\mathfrak{S}}  
\newcommand{\decoratedstate}{\mathfrak{D}}
\newcommand{\G}{\mathfrak{G}}  
\newcommand{\RSK}{\mathrm{RSK}}
\DeclareMathOperator{\rd}{rd} 

\newcommand{\leftbrack}{\textcolor{blue}{[}}
\newcommand{\rightbrack}{\textcolor{red}{]}}

\newcommand{\wt}{\mathsf{wt}} 

\newcommand{\inner}[2]{\left\langle #1, #2 \right\rangle}

\definecolor{darkred}{rgb}{0.7,0,0} 
\newcommand{\defncolor}{\color{darkred}}
\newcommand{\defn}[1]{{\defncolor\emph{#1}}} 

\makeatletter
\def\revddots{\mathinner{\mkern1mu\raise\p@
\vbox{\kern7\p@\hbox{.}}\mkern2mu
\raise4\p@\hbox{.}\mkern2mu\raise7\p@\hbox{.}\mkern1mu}}
\makeatother

\usepackage[colorinlistoftodos]{todonotes}

\newtheorem{theorem}{Theorem}[section]
\newtheorem{lemma}[theorem]{Lemma}
\newtheorem{corollary}[theorem]{Corollary}

\newtheorem{proposition}[theorem]{Proposition}

\theoremstyle{definition}
\newtheorem{definition}[theorem]{Definition}

\newtheorem{example}[theorem]{Example}

\newtheorem{remark}[theorem]{Remark}

\numberwithin{equation}{section}

\title[Uncrowding the 5-vertex model: RSK and crystal structures]{Uncrowding the 5-vertex model:\\ RSK and crystal structures}

\author[Johnston]{Lisa Johnston}
\address[L. Johnston]{Department of Mathematics, University of California, One Shields
        Avenue, Davis, CA 95616-8633, U.S.A.}
\email{lisjohnston@ucdavis.edu}

\author[Nguyen]{Evuilynn Nguyen}
\address[E. Nguyen]{Department of Mathematics, University of California, One Shields
        Avenue, Davis, CA 95616-8633, U.S.A.}
\email{evtnguyen@ucdavis.edu}

\author[Schilling]{Anne Schilling}
\address[A. Schilling]{Department of Mathematics, University of California, One Shields
        Avenue, Davis, CA 95616-8633, U.S.A.}
\email{aschilling@ucdavis.edu}
\urladdr{\href{http://www.math.ucdavis.edu/~anne}{http://www.math.ucdavis.edu/~anne}}

\keywords{symmetric Grothendieck polynomials, set-valued tableaux, 5-vertex model, crystal bases, uncrowding algorithm}

\makeatletter

\begin{document}

\begin{abstract}
        While the uncrowding algorithm on set-valued tableaux has long been instrumental in proving the Schur positivity of symmetric Grothendieck polynomials, lattice models have emerged as a modern framework for investigating symmetric 
        functions, in particular symmetric Grothendieck polynomials. In this work, we synthesize these combinatorial and lattice-theoretic 
        approaches by defining both the Robinson--Schensted--Knuth (RSK) correspondence and the uncrowding operation directly 
        on a 5-vertex model of Motegi and Sakai. Our 
        lattice-based RSK formulation yields a powerful new result: the direct construction of the associated crystal structure on the 
        states of the 5-vertex model.
\end{abstract}

\maketitle

\tableofcontents

\section{Introduction}

Symmetric Grothendieck polynomials serve as a $K$-theoretic analogue to classical Schur functions, which originate from
the geometric study of Grassmannians. While standard Schur functions represent the cohomology classes of Schubert varieties, symmetric Grothendieck polynomials
represent the classes of their structure sheaves in $K$-theory. 
Grothendieck polynomials were first introduced by Lascoux and Sch\"utzenberger~\cite{LS.1982} and their stable limits (which includes the symmetric Grothendieck polynomials) were studied by Fomin and Kirillov~\cite{FK.1994}.
Combinatorially, just as Schur functions are defined as the generating functions
of semistandard Young tableaux, symmetric Grothendieck polynomials are famously realized as the generating functions of set-valued 
tableaux as established by Buch~\cite{Buch_2002}. Specifically, the symmetric Grothendieck polynomials in $n$ variables
$\mathbf{z} = (z_1,\ldots,z_n)$ are indexed by a partition $\lambda$ and can be defined as 
\begin{equation*}
    \G_\lambda(\mathbf{z};\beta):= \sum_{T \in \SVT^n(\lambda)} \beta^{\mathsf{ex}(T)} \mathbf{z}^{\wt(T)},
\end{equation*}
where $\SVT^n(\lambda)$ is the set of semistandard set-valued tableaux of shape $\lambda$ in the alphabet $\{1,2,\ldots,n\}$, 
$\mathsf{ex}(T)$ denotes the number of entries in $T$ 
minus the size of $\lambda,$ and $\wt(T)$ is the vector with the $i$-th entry being the number of $i$'s in $T$. 

Due to their rich geometric origins and complex combinatorial properties, these polynomials have become a central object of study in modern 
algebraic combinatorics. A powerful combinatorial tool associated to symmetric Grothendieck polynomials is the uncrowding algorithm introduced by
Buch~\cite{Buch_2002}. The uncrowding algorithm is a bijection between set-valued tableaux and pairs of semistandard Young tableaux and 
flagged increasing tableaux. Multiple integers can occupy a single cell in the set-valued tableau. The uncrowding algorithm resolves this by iteratively 
``moving'' the extra entries out of crowded cells and into new rows via the Schensted bumping algorithm, eventually yielding a semistandard tableau. 
This process is recorded via a flagged increasing tableau. Buch~\cite{Buch_2002} used the uncrowding algorithm to provide a bijective
proof of the Schur expansion of the symmetric Grothendieck polynomials previously given by Lenart~\cite{Len_00}.
The uncrowding algorithm and its generalizations have been studied further in~\cite{Bandlow_2012, Patrias_2016, Reiner_2018, Chan_2021,
pan_uncrowding_2022, JKPPS.2026}.

Solvable lattice models have emerged as a powerful, modern framework for investigating symmetric functions. Originally developed to study physical 
systems like ice, solvable lattice models have become a versatile tool in mathematics with applications in algebraic geometry, integrable probability, 
and algebraic combinatorics. Demazure atoms and their associated crystals have been interpreted using an open colored 5-vertex model in~\cite{BBBG.2021}.
A closed colored 5-vertex model for Demazure characters and their crystals was studied in~\cite{Yang.2025}.
Symmetric Grothendieck polynomials admit an interpretation as partition functions of certain 5-vertex lattice models 
as described by Motegi and Sakai~\cite{Motegi_2013, MS.2014}. Gorbounov and Korff~\cite{GK.2017} gave a $T$-equivariant $K$-theory version, 
and Buciumas, Scrimshaw, and Weber~\cite{Buciumas_2020} gave a colored version of this model.
In the Motegi--Sakai model, admissible states of the lattice model are in weight-preserving bijection with Gelfand--Tsetlin patterns. Expanding the lattice model 
weights refines these states in a way that recovers the set-valued tableaux indexing symmetric Grothendieck polynomials. This realization 
places symmetric Grothendieck polynomials within the broader context of integrable systems and statistical mechanics, providing 
new structural insight and diagrammatic tools.

While the use of solvable lattice models to study symmetric functions is a relatively recent development, we aim to further connect these 
two combinatorial frameworks. In this paper we give a lattice model interpretation of the Robinson--Schensted--Knuth (RSK) insertion algorithm,
which we then generalize to a lattice model analogue of uncrowding. Although the Schur expansion of symmetric Grothendieck polynomials is 
not immediately apparent from their lattice model description, our algorithm provides a method to recover the Schur expansion directly within the lattice model. 
Furthermore, our lattice model interpretation of RSK yields another new result: the direct construction of a crystal
structure on the states of the 5-vertex lattice model. A different interpretation of RSK in a lattice model appeared in~\cite{MP.2022}.

The paper is organized as follows. In Section~\ref{section.background}, we review the definition of the symmetric Grothendieck
polynomials, their combinatorics, their lattice model interpretation, and background on crystals. In Section~\ref{section.RSK}, we provide the new description
of RSK directly on the 5-vertex model and use it in Section~\ref{sec:uncrowd-lattice-algo} to define and analyze the uncrowding
algorithm in the lattice model setting. In Section~\ref{section.crystal}, we construct crystal operators on the states
of the 5-vertex model and show that they intertwine with the uncrowding algorithm. We conclude in Section~\ref{sec:future_work} with
a discussion of possible generalizations of this work.

\subsection*{Acknowledgements}

We thank Dan Bump and Travis Scrimshaw for useful discussions.

We thank ICERM for the stimulating research atmosphere during the semester program ``Categorification and Computation in Algebraic Combinatorics''
in Fall 2025 and related workshops. This material is based on work supported by the National Science Foundation under Grant No. DMS-1929284, 
while the authors were in residence at ICERM. L.\ Johnston has been supported as a GAANN Fellow through the Mathematics at UC Davis GAANN 
grant, funded by the Department of Education under grant P200A240025.
A.\ Schilling was partially supported by Simons Foundation grant MPS-TSM-00007191.

\section{Background}
\label{section.background}

In this section, we give background on symmetric Grothendieck polynomials, their combinatorics, and their lattice model
interpretation.

\subsection{Tableau combinatorics}
We begin by reviewing tableau combinatorics associated to symmetric Grothendieck polynomials.
We use English notation for partitions and tableaux. For a partition $\lambda=(\lambda_1,\lambda_2,\ldots)$, let $|\lambda|=\sum_i \lambda_i$ denote its size and let $\ell(\lambda)$ denote its length, the number of nonzero parts. We freely append trailing zeroes to partitions when needed. For a partition $\lambda$, we denote by \defn{$\SSYT(\lambda)$} (resp. \defn{$\SSYT^n(\lambda)$}) 
the set of semistandard Young tableaux of shape $\lambda$ (resp. with entries at most $n$). 
Set-valued tableaux were first introduced in~\cite{Buch_2002}. For a partition $\lambda$, a \defn{semistandard set-valued tableau} 
(or set-valued tableau for short) of shape $\lambda$ is a filling $T$ of the Young diagram for $\lambda$ with non-empty finite sets of 
positive integers, such that
\begin{enumerate}[label=(\roman*)]
     \item max($A$) $\leqslant$ min($B$) whenever the cell $A$ is in the same row and to the left of $B$;
     \item max($A$) $<$ min($C$) whenever the cell $A$ is in the same column and above $C$.
\end{enumerate}
Equivalently, $T$ is a set-valued tableau if any choice of a single element from each cell always gives a semistandard tableau. 
A cell in a set-valued tableau with more than one number is a \defn{multicell}. We call the non-minimal numbers in a multicell the 
\defn{crowded} entries. The set of all set-valued tableaux of shape $\lambda$ (resp. with entries at most $n$) is denoted by \defn{$\SVT(\lambda)$} 
(resp. \defn{$\SVT^n(\lambda)$}).
The \defn{weight} of a tableau $T\in \SVT^n(\lambda)$ is the tuple $\wt(T)=(\alpha_1,\ldots,\alpha_n)$, where $\alpha_i$ is the number of
$i$'s in $T$. The \defn{excess} of $T\in \SVT^n(\lambda)$ with weight $\mu$ is defined as $\mathsf{ex}(T) = |\mu|-|\lambda|$. 

A \defn{flagged increasing tableau} (also known as a strict elegant filling) is a semistandard Young tableau of skew shape that is strictly increasing 
across rows with the property that entries in the $i$-th row (counted from the top row of $\lambda$ and $\mu$) are restricted to $1,2,\dots, i-1$ 
for all $1 \leqslant i \leqslant \ell(\mu)$. Let $\mathcal{F}(\mu/\lambda)$ denote the set of flagged increasing tableaux of shape $\mu/\lambda$.

Let $T\in \SSYT^n(\lambda)$ and let $i$ be a positive integer. The \defn{row-insertion} of $i$ into a row $R$ of $T$, denoted $R \leftarrow i$, 
is defined as follows. Compare $i$ with the entries of $R$ from left to right. If no entry of $R$ is strictly larger than $i$, append $i$ to the end of the 
row. Otherwise, let $j$ be the leftmost entry in $R$ strictly larger than $i$. Replace $j$ by $i$, thereby ``bumping'' $j$, and then insert $j$ into the next 
row by repeating the procedure. 

For a word $w= w_1 w_2 \ldots w_\ell$, let $P(w)$ denote the tableau obtained by successively inserting the letters $w_1, w_2, \dots , w_\ell$ 
into the first row, where we start with the empty tableau. 

Next recall the row-insertion formulation of the Robinson--Schensted--Knuth correspondence. A generalized permutation 
\[ w = \left(
\begin{array}{cccc}
    a_1 & a_2 & \cdots & a_\ell \\
    b_1 & b_2 & \cdots & b_\ell 
\end{array}\right),
\]
is a two-row array such that $a_i \leqslant a_{i+1}$ and $b_i\leqslant b_{i+1}$ whenever $a_i=a_{i+1}$. Given a generalized permutation $w$,
the \defn{Robinson--Schensted--Knuth (RSK) correspondence} assigns a pair of semistandard tableaux
\[
\RSK(w)= (P,Q)
\]
where $P=P(b_1b_2\ldots b_\ell)$ and $Q$ records the entry $a_i$ in the new box created after inserting $b_i$. 
Note that $P$ and $Q$ by construction have the same shape.

We now state definitions and bijections from~\cite[Section 3]{MS.2014}, ~\cite[Section 4.2]{Monical_2020}, which we will need later.

\begin{definition}
\label{definition.GT}
    A \defn{Gelfand--Tsetlin (GT) pattern} is a triangular array 
    \[
    \begin{matrix}
   \lambda^{(n)}_{1}&   \lambda^{(n)}_{2}&   \cdots&  \cdots&   \lambda^{(n)}_{n-1}&&\hspace{-1em} \lambda^{(n)}_{n} \\[9pt]
   & \hspace{-1em}\lambda^{(n-1)}_{1}& \hspace{-0.5em}\lambda^{(n-1)}_{2}& \cdots&& \hspace{-1.5em}\lambda^{(n-1)}_{n-1}\\[9pt]
   & \ddots &&\hspace{-4em}\ddots &\hspace{-1.5em}\revddots \\[9pt]
   &&&\hspace{-2em}\lambda^{(1)}_{1}
  \end{matrix}
  \]
  where each row is weakly decreasing and hence can be viewed as a partition (with possibly trailing zeroes). 
  Setting $\lambda^{(0)}=\emptyset$ and denoting the $j$-th row by the partition $\lambda^{(j)}=(\lambda_1^{(j)},\ldots,\lambda_j^{(j)})$, 
  it is further required that the skew shape $\lambda^{(j)}/\lambda^{(j-1)}$ is a horizontal strip. We succinctly denote a GT pattern
  by its sequence of partitions $\Lambda= (\lambda^{(j)})_{j=0}^n$.
  
  A \defn{marked Gelfand--Tsetlin (MGT) pattern} is a GT pattern together with a distinguished set of marked positions $M$. 
    An entry $(i,j)$, with $1 \leqslant i <j$ and $2 \leqslant j \leqslant n$, may (but not necessarily) be marked if 
    $\lambda^{(j)}_{i+1}<\lambda^{(j-1)}_{i}$. Let \defn{$\MGT^n(\lambda)$} denote the set of all MGT patterns with top row $\lambda^{(n)}=\lambda$
    and a sequence of $n+1$ partitions.
\end{definition}

\begin{example}
Below is a marked GT pattern with top row $\lambda = (3,2,0,0,0)$ whose marked entries are boxed:
\begin{equation}
\label{equation:GTpattern}
\begin{array}{ccccccccc}
3 && \boxed{2} && 0 && 0 && 0 \\[4pt]
& \boxed{3} && 2 && 0 && 0 \\[4pt]
&& 3 && \boxed{1} && 0 \\[4pt]
&&& \boxed{2} && 1 \\[4pt]
&&&& 2
\end{array}
\end{equation}
\end{example}

Consider the recursive construction of the bijection between marked Gelfand--Tsetlin patterns with top row $\lambda$ and set-valued tableaux 
of shape $\lambda$:
\begin{equation*} 
\phi \colon \MGT^n(\lambda) \longrightarrow \SVT^n(\lambda).
\end{equation*}
Let $(\Lambda, M) \in \MGT^n(\lambda)$. 
\begin{enumerate}[label=(\roman*)]
    \item Begin with an empty tableau $T_0 = \emptyset$. 
    \item Assume that at stage $1\leqslant j \leqslant n$ we have already formed a set-valued tableau $T_{j-1}$ whose entries lie in $\{1, \ldots, j-1\}$.
    \item For each marked position $(i,j)\in M$, insert the value $j$ into the rightmost cell of the $i$-th row of $T_{j-1}$. Let the tableau obtained after 
    performing all such insertions be $\tilde{T}_j$. 
    \item Next, take the horizontal strip $\lambda^{(j)}/\lambda^{(j-1)}$ and fill each of its cells with $j$. Attach this strip to $\tilde{T}_j$ to obtain a tableau $T_j$ 
    of shape $\lambda^{(j)}$. 
    \item Continue this process for every row of $\Lambda$. The final tableau is $\phi(\Lambda, M)$. 
\end{enumerate}
The weight of a MGT pattern is then $\wt(\Lambda, M):=\wt(\phi(\Lambda, M))$. 

\begin{example}
Recall the MGT pattern $(\Lambda, M)$ from~\eqref{equation:GTpattern}. Below we show the image under $\phi$. Here, a colored marking $(i,j)\in M$
represents the colored number $j$ in $\phi(\Lambda, M)$:
\[
(\Lambda, M) = \begin{array}{ccccccccc}
3 && \textcolor{violet}{\boxed{2}} && 0 && 0 && 0 \\[4pt]
& \textcolor{ForestGreen}{\boxed{3}} && 2 && 0 && 0 \\[4pt]
&& 3 && \textcolor{blue}{\boxed{1}} && 0 \\[4pt]
&&& \textcolor{red}{\boxed{2}} && 1 \\[4pt]
&&&& 2
\end{array}, \qquad \qquad 
\ytableausetup{boxsize=2.2em}
\phi(\Lambda, M) = 
\ytableaushort{1{1,\!\textcolor{red}{2}}{3,\!\textcolor{ForestGreen}{4}},{2,\! \textcolor{blue}{3}}{4,\!\textcolor{violet}{5}}}\,
\]
\end{example}

Next we recall the uncrowding operator for set-valued tableaux introduced by Buch~\cite{Buch_2002}. This operator maps a given set-valued tableau 
to a pair consisting of a semistandard Young tableau together with a flagged increasing tableau.

\begin{definition}
\label{def:uncrowding-SVT}
    Let $T\in \SVT^n(\lambda)$. The \defn{uncrowding operation} on $T$ proceeds as follows.
    Find the bottommost row $r$ that contains a multicell, and within that row, find the largest entry $x$ appearing in a multicell. 
    Remove $x$ from its cell and row-insert it into the row below. The result is a tableau whose shape is obtained from $\lambda$ by adding a single 
    box.
    
    The \defn{uncrowding map}
    \begin{equation}
    \label{eq:uncrowding_SVT_to_pair}
        \U_\SVT \colon \SVT^n(\lambda) \longrightarrow \bigsqcup_{\mu \supseteq \lambda} \SSYT^n(\mu) \times \mathcal{F}(\mu/\lambda)
    \end{equation}
    is defined as follows:
    \begin{enumerate}[label=(\roman*)]
        \item Set $U_0=T$, and let $F_0$ be the unique flagged increasing tableau of shape $\lambda/\lambda$. 
        \item For each $1\leqslant i \leqslant \mathsf{ex}(T)$, obtain $U_i$ by applying the uncrowding operation to $U_{i-1}$. If the new cell $C$ 
        added in $U_i$ lies in row $r'$, then form $F_i$ from $F_{i-1}$ by placing the entry $r'-r$ in the corresponding cell. 
        \item Define \[\U_\SVT(T)=\bigl(P_\SVT(T), F_\SVT(T)\bigr):=\bigl( U_{\mathsf{ex}(T)}, F_{\mathsf{ex}(T)}\bigr) .\]
    \end{enumerate}
\end{definition}
Buch~\cite{Buch_2002} showed that \eqref{eq:uncrowding_SVT_to_pair} is a bijection. 

\subsection{Symmetric Grothendieck polynomials}
For a positive integer $n$, let $\mathbf{z}=(z_1,z_2,\dots, z_n)$ be a finite list of indeterminates. 

\begin{definition}[\cite{Buch_2002}] \label{def:groth-SVT}
Let $\lambda$ be a partition. Define the \defn{symmetric Grothendieck polynomial} as 
\begin{equation*}
    \G_\lambda(\mathbf{z};\beta):= \sum_{T \in \SVT^n(\lambda)} \beta^{\mathsf{ex}(T)} \mathbf{z}^{\wt(T)}.
     \end{equation*}
\end{definition}
We denote the symmetric Grothendieck polynomials as $\G_\lambda$ when the context is clear. 

The Schur expansion for the symmetric Grothendieck polynomials was first given by Lenart~\cite{Len_00}.

\begin{theorem}[\cite{Len_00}]
\label{thm:G_expansion_lenart}
Let $\hat{\lambda}$ be the unique maximal partition with $n$ rows (in the Young lattice) obtained from $\lambda$ by adding at most $i-1$ boxes to 
its $i$-th row for $2 \leqslant i \leqslant n$. Then
\[
	\G_\lambda(\mathbf{z};\beta)= \sum_{\lambda \subseteq \mu \subseteq \hat{\lambda}} 
	\beta^{|\mu|-|\lambda|} M^\mu_{\lambda} s_\mu(\mathbf{z}),
\]    
where $ M^\mu_{\lambda}=|\mathcal{F}(\mu/\lambda)|$.
\end{theorem}

The uncrowding map given in \eqref{eq:uncrowding_SVT_to_pair} was used in~\cite{Buch_2002} to give a bijective proof of \Cref{thm:G_expansion_lenart}.

\begin{definition}
\label{def:SSYT-reading} 
Let $T \in \SSYT^n(\lambda).$ The \defn{reading word} of $T$ is the word obtained by concatenating the entries of $T$ starting from the bottom 
row to the top and reading the entries in each row left to right.
\end{definition}

\begin{definition}[\cite{Monical_2020}]
\label{def:SVT-reading} 
Let $T \in \SVT^n(\lambda).$ The \defn{reading word} of $T$ is the word created by concatenating the entries of $T$ starting from the bottom row to 
the top and reading the entries in each row left to right. Boxes with multiple entries are read in decreasing order.
\end{definition}

\begin{definition}
\label{def:Yamanouchi} 
A word $w = w_1w_2 \ldots w_d$ on the alphabet $[n]$ is \defn{Yamanouchi} if every final subword $w_k w_{k+1} \ldots w_d$ for 
$1 \leqslant k\leqslant d$ contains at least as many $i$'s as $i+1$'s for all $i \in [n-1].$ 
\end{definition}

\begin{example}\label{ex:tab-reading-word}
The reading words of two set-valued tableaux are given below, where the word on the left is Yamanouchi while the word on the right is not:
\ytableausetup{boxsize=2.5em}
\[
    \begin{array}{lc@{\qquad\qquad\qquad}c}
        & \qquad \ytableaushort{1 {1,2},3} & \qquad \ytableaushort{1 {1,2,3},3} \\[3em]
        \text{reading word:} & 3121 & 31321
    \end{array}
\]
\end{example}

We say a set-valued tableau is Yamanouchi if it has a Yamanouchi reading word. The following theorem gives an equivalent formulation for the Schur expansion using Yamanouchi set-valued tableaux.
\begin{theorem}[\cite{Monical_2020}] 
\label{thm:yamanouchi-expand}
We have 
\[
	\G_\lambda(\mathbf{z};\beta)= \sum_{\lambda \subseteq \mu} \beta^{|\mu|-|\lambda|} M^\mu_{\lambda} s_\mu(\mathbf{z}),
\] 
where $M^\mu_{\lambda}$ counts the number of Yamanouchi set-valued tableaux of shape $\lambda$ and weight $\mu$.
\end{theorem}

\subsection{Lattice model background}
The 5-vertex lattice model we use throughout this paper was originally introduced by Motegi and Sakai~\cite{Motegi_2013, MS.2014}, and later by Buciumas, Scrimshaw, and Weber~\cite{Buciumas_2020} through a colored version. Our formulation here differs slightly from these two references, as
outlined in Remark~\ref{remark.comparison} below.

Fix positive integers $n$ and $m$. Our model lives on a rectangular lattice with $n$ horizontal lines and $m$ vertical lines, where each intersection 
is a vertex and each edge may carry an arrow pointing right, up, or be empty. 

We fix the boundary conditions as follows. Every left boundary edge carries a right arrow. All right and bottom edges are empty. The top boundary is 
specified by a binary string of length $m$, where a 1 encodes an up arrow and a 0 encodes an empty edge.
\begin{figure}[hbtp!]
\centering
\renewcommand{\arraystretch}{1.2}
\[\begin{array}{|c|c|c|c|c|}
\hline
\tt a_1 & \tt a_2 & \tt b_1 & \tt b_2 & \tt c_1 \\
\hline
\begin{tikzpicture}
\coordinate (a) at (-.75, 0);
\coordinate (b) at (0, .75);
\coordinate (c) at (.75, 0);
\coordinate (d) at (0, -.75);
\coordinate (e) at (0,0);
\draw (a)--(c);
\draw (b)--(d);
\end{tikzpicture} &
\begin{tikzpicture} 
\draw (a)--(e);
\draw (e)--(b); 
\draw[{-{Stealth[length=3mm,width=2mm]},red,line width=1.2pt}] (e)--(c);
\draw[{-{Stealth[length=3mm,width=2mm]},red,line width=1.2pt}] (d)--(e);
\end{tikzpicture} &
\begin{tikzpicture}
\draw[{-{Stealth[length=3mm,width=2mm]},red,line width=1.2pt}] (e)--(c);
\draw[{-{Stealth[length=3mm,width=2mm]},red,line width=1.2pt}] (d)--(e);
\draw[{-{Stealth[length=3mm,width=2mm]},red,line width=1.2pt}] (a)--(e);
\draw[{-{Stealth[length=3mm,width=2mm]},red,line width=1.2pt}] (e)--(b); 
\end{tikzpicture} &
\begin{tikzpicture}
\draw (d)--(e);
\draw (e)--(b); 
\draw[{-{Stealth[length=3mm,width=2mm]},red,line width=1.2pt}] (e)--(c);
\draw[{-{Stealth[length=3mm,width=2mm]},red,line width=1.2pt}] (a)--(e);
\end{tikzpicture} &
\begin{tikzpicture}
\draw (d)--(e);
\draw (c)--(e);
\draw[{-{Stealth[length=3mm,width=2mm]},red,line width=1.2pt}] (a)--(e);
\draw[{-{Stealth[length=3mm,width=2mm]},red,line width=1.2pt}] (e)--(b);
\end{tikzpicture}\\
\hline

1 & 1+\beta z_i & 1 & z_i & 1\\
\hline
\end{array}\]
\caption{Boltzmann weights $\wt(v)$ for a vertex $v$ in the $i$-th row of the uncolored model.}
\label{fig:uncolored_wts}
\commentAlt{Figure~\ref{fig:uncolored_wts}. A table with five columns and three rows. See long description.} 
\commentLongAlt{Figure~\ref{fig:uncolored_wts}. The first row is labeled with a1, a2, b1, b2, and c1. Each column in the middle row contains two perpendicular axes crossing at the center. There are red arrows along these axes in different arrangements for each column. The last row is labeled with the corresponding weight values $1, 1+\beta z_i, 1, z_i$, and 1.}
\end{figure}

With these boundary conditions, a \defn{state} of the model is an assignment of up, right, or empty arrows to all remaining edges. The 5-vertex types 
in \Cref{fig:uncolored_wts} are the only locally allowed configurations, and a state is \defn{admissible} if every vertex is one of these five types. We 
single out vertices of type $\tt a_2$ and call them \defn{bumps}. A vertex with $\tt a_2$ configuration is called a \defn{bump vertex}.
Each of the 5-vertex types is assigned a \defn{Boltzmann weight}. 

The top boundary condition is determined by a partition $\lambda$ as follows. Set $m=\lambda_1+n$, and inscribe the Young diagram of 
$\lambda$ in the upper-left corner of an $n \times (m-n)$ grid. Trace the border of $\lambda$ from its lower-left corner to its upper-right corner 
and write a 1 for each northward step and 0 for each eastward step. We call the resulting binary string of length $m$ the \defn{$\{0,1\}$-sequence} of 
$\lambda$. In this encoding, the $1$'s appear exactly in positions $\lambda_i+(n-i+1)$ for $1 \leqslant i \leqslant n$. 

\begin{example}
For $\lambda=(3,2,0,0,0)$ with $n=5$ and $m=8$, we obtain the $\{0,1\}$-sequence $11100101$. 
\end{example}

Let $\states_\lambda^n$ denote the set of all admissible states with boundary condition $\lambda$ for a given $n$. For an admissible state 
$S \in \states_\lambda^n$, its \defn{weight} $\wt(S)$ is computed by multiplying together the Boltzmann weight at each vertex, where the variable 
at a vertex in lattice row $i$ (with rows counted from bottom to top) is taken to be $z_i$, so that $z_n$ appears in the top row and $z_1$ in the bottom row. 
States that are not admissible are assigned weight 0. The \defn{partition function} of a set of admissible states $\mathcal{M}$ is then 
\[
	Z(\mathcal{M}; \mathbf{z}; \beta) := \sum_{S \in \mathcal{M}} \wt(S),
\]
that is, the sum of the Boltzmann weights over all admissible states in $\mathcal{M}$.

We now associate a family of labeled directed paths to each admissible state. First, we assign the label $i$ to each arrow entering from the left boundary in row $i$. We extend these labeled paths through the lattice until they reach an arrow in the top row, called an \defn{exit arrow}. Excluding the vertices of type $\tt b_1$, each labeled path follows the unique directed path determined by the arrows. At a $\tt b_1$ vertex, two distinct paths meet at a vertex but do not cross. That is, the path entering from the left exits through the top, while the path entering from the bottom exits through the right. Thus, distinct paths may share vertices but never intersect. For each $i \in \{1, \dots, n \}$, let $L_i$ be the directed path carrying the label $i$. We call $\{L_1, \dots, L_n\}$ the \defn{natural paths}. We index the exit arrows $1, \dots, n$ from right to left. Therefore, each natural path $L_i$ enters through the $i$-th row and ends at the $i$-th exit arrow by following the southeastern-most directed path. For clarity, we differentiate the natural paths with various colors in subsequent pictures.

We will show in Lemma~\ref{lem:paths-to-tab} that under the bijection with set-valued 
tableaux, $L_i$ corresponds to row $i$ of the tableau. 

\begin{remark}
\label{remark.comparison}
    The lattice model described above is equivalent to the lattice models introduced in~\cite{Motegi_2013,MS.2014,Buciumas_2020}. Our formulation 
    replaces their $\{0,1\}$-labels on half-edges with an equivalent description using up, right, and empty arrows. 
    Furthermore, our row indexing differs from the convention in~\cite{Buciumas_2020}. Throughout this paper, rows, weights, and natural paths in the lattice are indexed from bottom to top, instead of top to bottom.
\end{remark}

There is a well-known bijection between admissible states and GT patterns with top row $\lambda$,  
\begin{equation*}
	\mathfrak{P} \colon \states_\lambda^n \longrightarrow \GT^n(\lambda).
\end{equation*}
Concretely, for each $0 \leqslant j \leqslant n$, read off the $\{0,1\}$-sequence on the 
$j$-th row of vertical edges (treating the bottom boundary as row 0) and 
let $\lambda^{(j)}$ be the partition it encodes. The sequence $(\lambda^{(j)})_{j=0}^n$ is then the GT pattern $\mathfrak{P}(S)$.

Since each bump vertex carries a factor of $1+\beta z_i$ in $\wt(S)$, expanding this product introduces a choice at every bump: include the 
$\beta z_i$ term or the constant 1. Each such choice corresponds to selecting a marking on the GT pattern $\mathfrak{P}(S)$, and so 
\begin{equation*}
    \wt(S)=\sum_{M} \beta^{|M|}\mathbf{z}^{\wt(\mathfrak{P}(S),M)},
\end{equation*}
where the sum ranges over all markings $M$ of $\mathfrak{P}(S)$.

\begin{theorem}[\cite{Motegi_2013}]
\label{thm:groth-equals-partition-function}
The symmetric Grothendieck polynomial is equal to the partition function of the model $\states_\lambda^n$:
\[
	\G_\lambda(\mathbf{z}; \beta) = Z(\states_\lambda^n; \mathbf{z}; \beta).
\]
\end{theorem}

Recall that each $\tt{a}_2$ (bump) vertex in a state $S\in \states_\lambda^n$ contributes a factor of $1+\beta z_i$ to $\wt(S)$. Choosing 
either 1 or $\beta z_i$ from each bump determines a term in the expansion of $\wt(S)$. Furthermore, by Definitions~\ref{definition.GT} 
and~\ref{def:groth-SVT} and \Cref{thm:groth-equals-partition-function}, each term corresponds to a set-valued tableau. We therefore say that a bump 
is \defn{trivial} if the contribution $1$ is chosen, and \defn{non-trivial} if the contribution $\beta z_i$ is chosen. This motivates the following definition. 
    
\begin{definition}
For a fixed $n$, let $\decoratedstate_\lambda$ denote the set of pairs $S^{\text{dec}}=(S,N)$, where $S \in \states_\lambda^n$ and $N$ is a subset of the 
bump vertices in $S$ specifying the non-trivial bump vertices. We refer to these pairs as \defn{decorated states}. Furthermore, let 
$\decoratedstate^0_\lambda \subseteq \decoratedstate_\lambda$ denote the decorated states where all bumps, if any, are trivial; this 
corresponds to $N=\emptyset$. We refer to this set as the \defn{trivial decorated states}. 
\end{definition}

For a decorated state $S^{\text{dec}}\in \decoratedstate_\lambda$, we call a vertex \defn{contributing} if its local weight is not equal to 1. Thus, the contributing 
vertices in $S^\text{dec}$ are precisely $\tt b_2$ vertices and the non-trivial bumps in $N$.

Under the bijection ${\mathfrak P} \colon \states_\lambda^n \longrightarrow \GT^n(\lambda)$, the non-trivial bumps correspond exactly to markings 
of the associated GT pattern. Hence, $\mathfrak P$ naturally extends to a bijection 
\begin{equation}
\widetilde{\mathfrak P} \colon \decoratedstate_\lambda \longrightarrow \MGT^n(\lambda).
\end{equation}
By composing with $\phi \colon \MGT^n(\lambda)\longrightarrow \SVT^n(\lambda)$, we obtain the bijection 
\begin{equation}
\label{eq:bij-decoratedstate-SVT}
    \psi \colon \decoratedstate_\lambda \longrightarrow \SVT^n(\lambda).
\end{equation}
For a decorated state, we define its weight by assigning weight $1$ to trivial bumps, weight $\beta z_i$ to non-trivial bumps in row $i$, and weight $z_i$ to vertices of type $\tt b_2$ in row $i$, with all other vertices having weight $1$.
Hence, 
\[
	\G_\lambda(\mathbf{z}; \beta) = Z(\decoratedstate_\lambda; \mathbf{z}; \beta).
\]
In the lattice model, we represent trivial bumps by crossed vertices, non-trivial bumps by circled vertices, and $\tt b_2$ vertices by dotted marks. 
In doing so, every contributing vertex can be visually distinguished.

\begin{example}
Let $\lambda=(2,1)$ and $n=2$. Then 
\[
	\G_\lambda(z_1, z_2; \beta)=s_{(2,1)}+\beta s_{(2,2)}= z_1z_2^2+z_1^2z_2(1+\beta z_2).
\]
The set $\SVT^2(\lambda)$ consists of the following three tableaux: 
\[
	\ytableausetup{boxsize=2.2em}
	\ytableaushort{1 2,2} \qquad \ytableaushort{1 1,2} \qquad  \ytableaushort{1 {1,2},2}\;.
\]
There are two admissible states in $\states_\lambda^2$. After decorating the bump in the second state, we obtain three decorated states shown below 
with their weights. The state $S_1$ has no bumps, so it determines a unique decorated state $S_1^{\text{triv}}=(S_1,\emptyset)$. The second state 
$S_2$ has one bump vertex $v$, so it gives rise to two decorated states $S_2^{\text{triv}}=(S_2,\emptyset)$, $S_2^{\text{nontriv}}=(S_2,\{v\})$ in 
$\decoratedstate_\lambda$ depending on whether the bump is trivial or non-trivial. Hence, $\decoratedstate^0_\lambda=\{S_1^{\text{triv}}, S_2^{\text{triv}}\}$. 
The natural paths $L_1$ and $L_2$ are colored red and blue, respectively. Since $n=2$, $L_1$ corresponds to row $1$ of the tableau and $L_2$ 
corresponds to row $2$. The state $S_1^{\text{triv}}$ corresponds to the left tableau, $S_2^{\text{triv}}$ corresponds to the middle tableau, 
and $S_2^{\text{nontriv}}$ corresponds to the right tableau which contains one crowded entry.
\[
\begin{array}{ccc}

\scalebox{0.9}{
 \begin{tikzpicture}

    \draw (1,0)--(1,3);
    \draw (2,0)--(2,3);
    \draw (3,0)--(3,3);
    \draw (4,0)--(4,3);
    
    \draw (0,1)--(5,1);
    \draw (0,2)--(5,2);
    
    \draw[-{Stealth[length=3mm,width=2mm]},red ,line width=1.2pt] (0,1)--(1,1);
    \draw[-{Stealth[length=3mm,width=2mm]},red,line width=1.2pt] (1,1)--(2,1);
    \draw[-{Stealth[length=3mm,width=2mm]},red,line width=1.2pt] (2,1)--(2,2);    
    \draw[-{Stealth[length=3mm,width=2mm]},red,line width=1.2pt] (2,2)--(3,2);  
    \draw[-{Stealth[length=3mm,width=2mm]},red,line width=1.2pt] (3,2)--(4,2); 
    \draw[-{Stealth[length=3mm,width=2mm]},red,line width=1.2pt] (4,2)--(4,3);  

    \draw[-{Stealth[length=3mm,width=2mm]},blue,line width=1.2pt] (0,2)--(1,2);
    \draw[-{Stealth[length=3mm,width=2mm]},blue,line width=1.2pt] (1,2)--(2,2);
    \draw[-{Stealth[length=3mm,width=2mm]},blue,line width=1.2pt] (2,2)--(2,3);
    
    \node at (-.25,1) {$ z_1$};
    \node at (-.25,2) {$ z_2$};

    \node[circle, fill=violet, draw = violet, opacity = 0.4, ultra thick, minimum size=7pt, inner sep=2pt] (c) at (1,1){};
    \node[circle, fill=violet, draw = violet, opacity = 0.4, ultra thick, minimum size=7pt, inner sep=2pt] (c) at (1,2){};
    \node[circle, fill=violet, draw = violet, opacity = 0.4, ultra thick, minimum size=7pt, inner sep=2pt] (c) at (3,2){};
  \end{tikzpicture}
  }
  &
  \scalebox{0.9}{
  \begin{tikzpicture}

     \draw (1,0)--(1,3);
    \draw (2,0)--(2,3);
    \draw (3,0)--(3,3);
    \draw (4,0)--(4,3);
    
    \draw (0,1)--(5,1);
    \draw (0,2)--(5,2);
    
    \draw[-{Stealth[length=3mm,width=2mm]},red ,line width=1.2pt] (0,1)--(1,1);
    \draw[-{Stealth[length=3mm,width=2mm]},red,line width=1.2pt] (1,1)--(2,1);
    \draw[-{Stealth[length=3mm,width=2mm]},red,line width=1.2pt] (2,1)--(3,1);    
    \draw[-{Stealth[length=3mm,width=2mm]},red,line width=1.2pt] (3,1)--(3,2);  
    \draw[-{Stealth[length=3mm,width=2mm]},red,line width=1.2pt] (3,2)--(4,2); 
    \draw[-{Stealth[length=3mm,width=2mm]},red,line width=1.2pt] (4,2)--(4,3);  

    \draw[-{Stealth[length=3mm,width=2mm]},blue,line width=1.2pt] (0,2)--(1,2);
    \draw[-{Stealth[length=3mm,width=2mm]},blue,line width=1.2pt] (1,2)--(2,2);
    \draw[-{Stealth[length=3mm,width=2mm]},blue,line width=1.2pt] (2,2)--(2,3);
    
    \node at (-.25,1) {$ z_1$};
    \node at (-.25,2) {$ z_2$};
    
    \node[circle, fill=violet, draw = violet, opacity = 0.4, ultra thick, minimum size=7pt, inner sep=2pt] (c) at (1,1){};
    \node[circle, fill=violet, draw = violet, opacity = 0.4, ultra thick, minimum size=7pt, inner sep=2pt] (c) at (1,2){};
    \node[circle, fill=violet, draw = violet, opacity = 0.4,ultra thick, minimum size=7pt, inner sep=2pt] (c) at (2,1){};    
    \node[cross out, fill=violet, draw = violet, ultra thick, minimum size=10pt, inner sep=2pt] (c) at (3,2){};   
  \end{tikzpicture}
  }
  &
  \scalebox{0.9}{
  \begin{tikzpicture}

     \draw (1,0)--(1,3);
    \draw (2,0)--(2,3);
    \draw (3,0)--(3,3);
    \draw (4,0)--(4,3);
    
    \draw (0,1)--(5,1);
    \draw (0,2)--(5,2);
    
    \draw[-{Stealth[length=3mm,width=2mm]},red ,line width=1.2pt] (0,1)--(1,1);
    \draw[-{Stealth[length=3mm,width=2mm]},red,line width=1.2pt] (1,1)--(2,1);
    \draw[-{Stealth[length=3mm,width=2mm]},red,line width=1.2pt] (2,1)--(3,1);    
    \draw[-{Stealth[length=3mm,width=2mm]},red,line width=1.2pt] (3,1)--(3,2);  
    \draw[-{Stealth[length=3mm,width=2mm]},red,line width=1.2pt] (3,2)--(4,2); 
    \draw[-{Stealth[length=3mm,width=2mm]},red,line width=1.2pt] (4,2)--(4,3);  

    \draw[-{Stealth[length=3mm,width=2mm]},blue,line width=1.2pt] (0,2)--(1,2);
    \draw[-{Stealth[length=3mm,width=2mm]},blue,line width=1.2pt] (1,2)--(2,2);
    \draw[-{Stealth[length=3mm,width=2mm]},blue,line width=1.2pt] (2,2)--(2,3);
    
    \node at (-.25,1) {$ z_1$};
    \node at (-.25,2) {$ z_2$};
    
    \node[circle, fill=violet, draw = violet, opacity = 0.4,ultra thick, minimum size=7pt, inner sep=2pt] (c) at (1,1){};
    \node[circle, fill=violet, draw = violet, opacity = 0.4,ultra thick, minimum size=7pt, inner sep=2pt] (c) at (1,2){};
    \node[circle, fill=violet, draw = violet, opacity = 0.4, ultra thick, minimum size=7pt, inner sep=2pt] (c) at (2,1){};    
    \node[circle, draw = violet, ultra thick, minimum size=10pt, inner sep=2pt] (c) at (3,2){};   
    \end{tikzpicture}
    }
    \\
    S_1^{\text{triv}} & S_2^{\text{triv}} & S_2^{\text{nontriv}} \\
    \wt=z_1z_2^2 & \wt=z_1^2z_2 & \wt=\beta z_1^2z_2^2
\end{array} 
\]
\end{example}

\subsection{Crystal background}
We briefly recall the type $A$ crystal conventions used throughout this paper.
For additional details about crystals, we refer the reader to~\cite{BumpSchilling.2017}.

Let $\mathfrak{sl}_n$ be the Lie algebra of type $A_{n-1}$, and let $I=\{1,2, \ldots, n-1\}$. We denote the simple roots by $\{\alpha_i \mid i \in I\},$ the simple coroots by $\{h_i \mid i \in I\},$ and the fundamental weights by $\{\Lambda_i \mid i\in I\}.$ Let $\Lambda$ denote the integral span of the fundamental weights.

\begin{definition}
    A \defn{Kashiwara $U_q(\mathfrak{sl}_n)$-crystal} consists of a nonempty set $B$ together with maps 
    \begin{align*}
        e_i, f_i & \colon B \to B\sqcup\{0\},
        \\
        \wt & \colon B \to \Lambda, 
    \end{align*}
    for each $i \in I$. For $x \in B$, define the string lengths $\varepsilon_i, \varphi_i \colon B \to \ZZ_{\geqslant 0}$ by
         \begin{align*}
        \varepsilon_i(x) & = \max \{k \mid e_i^k(x) \neq 0\},
        \\ \varphi_i(x) & = \max \{k \mid f_i^k(x) \neq 0\}.
        \end{align*}
    We require that the following conditions hold for all $i \in I$: 
        \begin{enumerate}
        \item $f_i(x) = y$ if and only if $e_i(y)=x$ for all $x,y \in B$;
        \item $\wt(f_i(x)) = \wt(x) - \alpha_i$ for all $x \in B$ such that $f_i(x) \neq 0$;
        \item $\varphi_i(x) = \varepsilon_i(x) + \inner{h_i}{\wt(x)}$ for all $x\in B$.
        \end{enumerate}
\end{definition}

Crystals associated to a $U_q(\mathfrak{sl}_n)$-representation need to satisfy further conditions, see for example~\cite{BumpSchilling.2017}.

\begin{definition}
    Let $B_1$ and $B_2$ be two Kashiwara $U_q(\mathfrak{sl}_n)$-crystals.  A \defn{crystal morphism} $\Psi \colon B_1 \to B_2$ is a map 
    $\Psi\colon B_1\sqcup\{0\} \to B_2 \sqcup \{0\}$ with $\Psi(0) = 0$, such that whenever $x\in B_1$ and $\Psi(x)\in B_2$, the map preserves weights and string lengths:
    \[ \wt(\Psi(x))=\wt(x),\qquad \varepsilon_i(\Psi(x))=\varepsilon_i(x),\qquad \varphi_i(\Psi(x))=\varphi_i(x). \]

    Moreover, $\Psi$ commutes with the crystal operators whenever both sides are nonzero:
    \[ \Psi(e_i(x))=e_i(\Psi(x)), \qquad \Psi(f_i(x))=f_i(\Psi(x)).\]
\end{definition}
A bijective crystal morphism is called a \defn{crystal isomorphism}. If such a map exists between $B_1$ and $B_2$, then we write $B_1 \cong B_2$.

\section{RSK in the lattice model setting}
\label{section.RSK}

Since uncrowding uses the same row insertion as the Robinson--Schensted--Knuth correspondence for semistandard tableaux, a natural question 
is to identify the analogue of RSK within the framework of lattice models.

Given a semistandard tableau $T$ of shape $\lambda$ with entries in $[n]$, let $S^\text{dec} \in \decoratedstate^0_\lambda$ be the trivial decorated state 
corresponding to $T$. To allow for the additional box created by insertion, we append one empty column to the right and view $S^\text{dec}$ on an $n \times (m+1)$ grid, where $m = \lambda_1+n$. We wish to describe the analogue of the row 
insertion of $u\in [n]$ into $T$ on $S^\text{dec}$ in the lattice model. 
Throughout, tableau rows are indexed from top to bottom, while lattice rows are indexed from bottom to top. Under this convention, the natural path 
$L_i$ corresponds to row $i$ in the tableau.

\begin{lemma}
\label{lem:paths-to-tab}
Let $T\in \SSYT^n(\lambda)$ and let $S^\text{dec} \in \decoratedstate^0_\lambda$ be the corresponding trivial decorated state under $\psi^{-1}$ 
in~\eqref{eq:bij-decoratedstate-SVT}. For each $1\leqslant i \leqslant n$, the natural path $L_i$ corresponds to row $i$ of $T$. 
\end{lemma}

\begin{proof} 
Fix $1\leqslant i\leqslant n$. By definition, $L_i$ is the natural path entering the lattice from the left boundary edge in row $i$ to the $i$-th exit arrow.
Equivalently, in each row of the lattice, $L_i$ follows the $i$-th vertical arrow when the arrows are read from right to left. Hence $L_i$ is determined 
by the $i$-th 1 from the right in each corresponding $\{0,1\}$-sequence.

Under $\mathfrak{P}$, the vertical edges leaving lattice row $j$ encode the partition $\lambda^{(j)}$. Therefore as $L_i$ passes through the lattice, it 
records the entries 
\[
	\lambda_{i}^{(i)}, \lambda_{i}^{{(i+1)}}, \dots, \lambda_{i}^{(n)}.
\]
These are precisely the entries in the $i$-th column of the left-aligned MGT pattern. Under the bijection $\phi$, this column determines row $i$ of $T$. 
\end{proof}

The following algorithm produces the trivial decorated state that corresponds to $T \leftarrow u.$ We call this the \defn{lattice insertion algorithm}.

\begin{enumerate}[label=\textbf{Step \arabic*:}]

    \item Along natural path $L_1$, locate the unique $\tt c_1$ vertex of Figure~\ref{fig:uncolored_wts} lying in lattice row $u,$ and denote it by $V_1.$ Initialize $k=1$ and $B_{0}$ to be empty. 

    \item  Starting with the north edge leaving $V_{k},$ follow the natural path $L_{k}$ until either an exit arrow is reached or the path first reaches a $\tt b_2$ vertex. Let $\gamma_{k}$ denote this path segment. 
    Shift the path segment $\gamma_{k}$ one unit to the right. Then turn $V_{k}$ into a $\tt b_2$ vertex by inserting a right arrow on the east edge of $V_{k}$, so that this new right arrow connects to the shifted path. The final arrow of the shifted path segment may overlap with an existing right arrow.
    
    If $B_{k-1}$ is nonempty, then $B_{k-1}$ is a non-trivial $\tt a_2$ vertex.  After shifting $\gamma_{k}$, either this vertex has become a $\tt b_1$ vertex, or it remains a bump. In the latter case, declare $B_{k-1}$ to have trivial weight.
 
    \item If $\gamma_{k}$ terminates at an exit arrow, stop.

    \item If $\gamma_{k}$ terminates at a $\tt b_2$ vertex, then after shifting $\gamma_{k}$, the local configuration at this terminal vertex changes from a $\tt b_2$ vertex to a non-trivial $\tt a_2$ vertex. Denote this new non-trivial bump by $B_{k}$.
    Let $V_{k+1}$ be the rightmost vertex of type $\tt c_1$ to the left of $B_k$ in the same lattice row. This places the algorithm on the next natural path $L_{k+1}$.
    
     \item Increase $k$ by 1 and repeat Steps 2--4. 
\end{enumerate}

\begin{example}
\label{ex:lattice-insertion}
Consider inserting $1$ into the tableau $T$ below. 

We denote by $V_1$ the vertex identified in Step 1. In iteration $k=1$, the algorithm shifts the dotted path segment $\gamma_1$, creates the non-trivial bump $B_1$, and identifies the next vertex $V_2$. In iteration $k=2$, the algorithm shifts the dotted path segment $\gamma_2$, declares $B_1$ to be trivial, creates the non-trivial bump $B_2$, and identifies the next vertex $V_3$. In iteration $k=3$, the algorithm shifts the dotted path segment $\gamma_3$, turns $B_2$ into a $\tt b_1$ vertex, and then terminates. 

In each lattice, the dotted path segment is the segment $\gamma_{k}$ shifted during that iteration. The algorithm stops in the third iteration because $\gamma_3$ reaches an exit arrow and does not create a new non-trivial $\tt a_2$ vertex.

\[\begin{minipage}{0.6\textwidth}
{\begin{tikzpicture}

    \draw (0,-0.75)--(0,4);
    \draw (1,-0.75)--(1,4);
    \draw (2,-0.75)--(2,4);
    \draw (3,-0.75)--(3,4);
    \draw (4, -0.75) -- (4,4);
    \draw (5, -0.75) -- (5,4);
    \draw (6, -0.75) -- (6,4);

    \draw (-1,0)--(6.5,0);
    \draw (-1,1)--(6.5,1);
    \draw (-1, 2) -- (6.5, 2);
    \draw (-1, 3) -- (6.5, 3);

    \draw[-{Stealth[length=3mm,width=2mm]}, Orange, line width=1pt] (-1,3)--(0,3);
    \draw[-{Stealth[length=3mm,width=2mm]}, Orange, line width=1pt] (0,3)--(0,4);

    \draw[-{Stealth[length=3mm,width=2mm]}, Green, line width=1pt] (-1,2)--(0,2);
    \draw[-{Stealth[length=3mm,width=2mm]}, Green, line width=1pt] (0,2)--(1,2);
    \draw[-{Stealth[length=3mm,width=2mm]}, Green, line width=1pt] (1,2)--(1,3);
    \draw[-{Stealth[length=3mm,width=2mm]}, Green, line width=1pt] (1,3)--(2,3);
    \draw[-{Stealth[length=3mm,width=2mm]}, Green, line width=1pt] (2,3)--(2,4);

    \draw[-{Stealth[length=3mm,width=2mm]}, red, line width=1pt] (-1,1)--(0,1);
    \draw[-{Stealth[length=3mm,width=2mm]}, red, line width=1pt] (0,1)--(1,1);
    \draw[-{Stealth[length=3mm,width=2mm]}, red, line width=1pt] (1,1)--(1,2);
    \draw[-{Stealth[length=3mm,width=2mm]}, red, line width=1pt] (1,2)--(2,2);
    \draw[-{Stealth[length=3mm,width=2mm]}, red, line width=1pt] (2,2)--(3,2);
    \draw[-{Stealth[length=3mm,width=2mm]}, red, line width=1pt] (3,2)--(3,3);
    \draw[-{Stealth[length=3mm,width=2mm]}, red, line width=1pt] (3,3)--(4,3);
    \draw[-{Stealth[length=3mm,width=2mm]}, red, line width=1pt] (4,3)--(5,3);
    \draw[-{Stealth[length=3mm,width=2mm]}, red, line width=1pt] (5,3)--(5,4);

    \draw[-{Stealth[length=3mm,width=2mm]}, blue, line width=1pt] (-1,0)--(0,0);
    \draw[-{Stealth[length=3mm,width=2mm]}, blue, line width=1pt] (0,0)--(1,0);
    \draw[-{Stealth[length=3mm,width=2mm]}, blue, line width=1pt] (1,0)--(2,0);
    \draw[dotted, -{Stealth[length=3mm,width=2mm]}, blue, line width=2.5pt] (2,0)--(2,1);
    \draw[dotted, -{Stealth[length=3mm,width=2mm]}, blue, line width=2.5pt] (2,1)--(3,1);
    \draw[-{Stealth[length=3mm,width=2mm]}, blue, line width=1pt] (3,1)--(4,1);
    \draw[-{Stealth[length=3mm,width=2mm]}, blue, line width=1pt] (4,1)--(4,2);
    \draw[-{Stealth[length=3mm,width=2mm]}, blue, line width=1pt] (4,2)--(5,2);
    \draw[-{Stealth[length=3mm,width=2mm]}, blue, line width=1pt] (5,2)--(5,3);
    \draw[-{Stealth[length=3mm,width=2mm]}, blue, line width=1pt] (5,3)--(6,3);
    \draw[-{Stealth[length=3mm,width=2mm]}, blue, line width=1pt] (6,3)--(6,4);

    \node at (-1.25,0) {$ z_1$};
    \node at (-1.25,1) {$ z_2$};
    \node at (-1.25, 2) {$z_3$};
    \node at (-1.25, 3) {$z_4$};

     \node[circle, draw=violet, thick, minimum size=10pt] at (2,0){};
     \node at (2.25,-0.4) {$V_1$};
    \node at (2.45,1.35) {$\gamma_1$};

     \node[circle, draw=violet, fill=violet, opacity=0.4, ultra thick, minimum size=7pt, inner sep=2pt] at (0,0){};
    \node[circle, draw=violet, fill=violet, opacity=0.4, ultra thick, minimum size=7pt, inner sep=2pt] at (1,0){};
    \node[circle, draw=violet, fill=violet, opacity=0.4, ultra thick, minimum size=7pt, inner sep=2pt] at (3,1){};
    \node[circle, draw=violet, fill=violet, opacity=0.4, ultra thick, minimum size=7pt, inner sep=2pt] at (0,1){};
    \node[circle, draw=violet, fill=violet, opacity=0.4, ultra thick, minimum size=7pt, inner sep=2pt] at (2,2){};
    \node[circle, draw=violet, fill=violet, opacity=0.4, ultra thick, minimum size=7pt, inner sep=2pt] at (4,3){};
    \node[circle, draw=violet, fill=violet, opacity=0.4, ultra thick, minimum size=7pt, inner sep=2pt] at (0,2){};

    \node[cross out, draw=violet, ultra thick, minimum size=10pt, inner sep=2pt] at (2,1){};
    \node[cross out, draw=violet, ultra thick, minimum size=10pt, inner sep=2pt] at (4,2){};
    \node[cross out, draw=violet, ultra thick, minimum size=10pt, inner sep=2pt] at (3,3){};
    \node[cross out, draw=violet, ultra thick, minimum size=10pt, inner sep=2pt] at (1,3){};

    \draw [-Stealth](7.5,1.5) -- (8.5,1.5);   
    \node at (8,1.75) {$\psi$};
    
  \end{tikzpicture}}
  \end{minipage}
\begin{minipage}{0.3\textwidth}
      $\ytableausetup{boxsize=2.2em}
      \qquad \; \; \ytableaushort{1 1 2 , 2 3 4, 3 } \; = \; T_0$

      \end{minipage}
\]
    \commentAlt{The initial decorated state before applying the lattice insertion algorithm, together with the corresponding tableau $T_0$.}
    \commentLongAlt{The lattice state shows the path segment $\gamma_1$ beginning at the marked vertex $V_1$. The dotted blue segment indicates the portion shifted in the first iteration. Non-trivial weights are circled, and trivial bumps are crossed out.}
\vspace{0.5cm}
\[
    \begin{minipage}{0.6\textwidth}
{\begin{tikzpicture}

    \draw (0,-0.75)--(0,4);
    \draw (1,-0.75)--(1,4);
    \draw (2,-0.75)--(2,4);
    \draw (3,-0.75)--(3,4);
    \draw (4, -0.75) -- (4,4);
    \draw (5, -0.75) -- (5,4);
    \draw (6, -0.75) -- (6,4);

    \draw (-1,0)--(6.5,0);
    \draw (-1,1)--(6.5,1);
    \draw (-1, 2) -- (6.5, 2);
    \draw (-1, 3) -- (6.5, 3);
    
    \draw[-{Stealth[length=3mm,width=2mm]}, Orange, line width=1pt] (-1,3)--(0,3);
    \draw[-{Stealth[length=3mm,width=2mm]}, Orange, line width=1pt] (0,3)--(0,4);

    \draw[-{Stealth[length=3mm,width=2mm]}, Green, line width=1pt] (-1,2)--(0,2);
    \draw[-{Stealth[length=3mm,width=2mm]}, Green, line width=1pt] (0,2)--(1,2);
    \draw[-{Stealth[length=3mm,width=2mm]}, Green, line width=1pt] (1,2)--(1,3);
    \draw[-{Stealth[length=3mm,width=2mm]}, Green, line width=1pt] (1,3)--(2,3);
    \draw[-{Stealth[length=3mm,width=2mm]}, Green, line width=1pt] (2,3)--(2,4);

    \draw[-{Stealth[length=3mm,width=2mm]}, red, line width=1pt] (-1,1)--(0,1);
    \draw[-{Stealth[length=3mm,width=2mm]}, red, line width=1pt] (0,1)--(1,1);
    \draw[dotted, -{Stealth[length=3mm,width=2mm]}, red, line width=2.5pt] (1,1)--(1,2);
    \draw[dotted, -{Stealth[length=3mm,width=2mm]}, red, line width=2.5pt] (1,2)--(2,2);
    \draw[-{Stealth[length=3mm,width=2mm]}, red, line width=1pt] (2,2)--(3,2);
    \draw[-{Stealth[length=3mm,width=2mm]}, red, line width=1pt] (3,2)--(3,3);
    \draw[-{Stealth[length=3mm,width=2mm]}, red, line width=1pt] (3,3)--(4,3);
    \draw[-{Stealth[length=3mm,width=2mm]}, red, line width=1pt] (4,3)--(5,3);
    \draw[-{Stealth[length=3mm,width=2mm]}, red, line width=1pt] (5,3)--(5,4);

    \draw[-{Stealth[length=3mm,width=2mm]}, blue, line width=1pt] (-1,0)--(0,0);
    \draw[-{Stealth[length=3mm,width=2mm]}, blue, line width=1pt] (0,0)--(1,0);
    \draw[-{Stealth[length=3mm,width=2mm]}, blue, line width=1pt] (1,0)--(2,0);
    \draw[-{Stealth[length=3mm,width=2mm]}, blue, line width=1pt] (2,0)--(3,0);
    \draw[-{Stealth[length=3mm,width=2mm]}, blue, line width=1pt] (3,0)--(3,1);
    \draw[-{Stealth[length=3mm,width=2mm]}, blue, line width=1pt] (3,1)--(4,1);
    \draw[-{Stealth[length=3mm,width=2mm]}, blue, line width=1pt] (4,1)--(4,2);
    \draw[-{Stealth[length=3mm,width=2mm]}, blue, line width=1pt] (4,2)--(5,2);
    \draw[-{Stealth[length=3mm,width=2mm]}, blue, line width=1pt] (5,2)--(5,3);
    \draw[-{Stealth[length=3mm,width=2mm]}, blue, line width=1pt] (5,3)--(6,3);
    \draw[-{Stealth[length=3mm,width=2mm]}, blue, line width=1pt] (6,3)--(6,4);

    \node at (-1.25,0) {$ z_1$};
    \node at (-1.25,1) {$ z_2$};
    \node at (-1.25, 2) {$z_3$};
    \node at (-1.25, 3) {$z_4$};

    \node[circle, draw=violet, fill=violet, opacity=0.4, ultra thick, minimum size=7pt, inner sep=2pt] at (0,0){};
    \node[circle, draw=violet, fill=violet, opacity=0.4, ultra thick, minimum size=7pt, inner sep=2pt] at (1,0){};
    \node[circle, draw=violet, fill=violet, opacity=0.4, ultra thick, minimum size=7pt, inner sep=2pt] at (2,0){};
    \node[circle, draw=violet, fill=violet, opacity=0.4, ultra thick, minimum size=7pt, inner sep=2pt] at (0,1){};
    \node[circle, draw=violet, fill=violet, opacity=0.4, ultra thick, minimum size=7pt, inner sep=2pt] at (2,2){};
    \node[circle, draw=violet, fill=violet, opacity=0.4, ultra thick, minimum size=7pt, inner sep=2pt] at (4,3){};
    \node[circle, draw=violet, fill=violet, opacity=0.4, ultra thick, minimum size=7pt, inner sep=2pt] at (0,2){}; 

    \node[cross out, draw=violet, ultra thick, minimum size=10pt, inner sep=2pt] at (3,3){};
    \node[cross out, draw=violet, ultra thick, minimum size=10pt, inner sep=2pt] at (4,2){};
    \node[cross out, draw=violet, ultra thick, minimum size=10pt, inner sep=2pt] at (1,3){};

    \node[circle, draw=violet, thick, minimum size=10pt] at (1,1){};
    \node at (1.25,0.6) {$V_2$};

    \node[circle, draw = violet, ultra thick, minimum size=10pt, inner sep=2pt] at (3,1){};
    \node at (3.3,0.6) {$B_1$};

    \node at (1.45,2.35) {$\gamma_2$};
    
    \draw [-Stealth](7.5,1.5) -- (8.5,1.5);   
    \node at (8,1.75) {$\psi$};
    
  \end{tikzpicture}}
    \end{minipage}
    \begin{minipage}{0.3\textwidth}
{$\qquad \; \; \ytableausetup{boxsize=2.2em}
\ytableaushort{1 1 {1,2}, 2 3 4, 3 }  \; = \; T_1 $}
    
    \end{minipage}
\]
    \commentAlt{The decorated state after the first iteration of the lattice insertion algorithm, together with the corresponding tableau $T_1$.}
    \commentLongAlt{The lattice state shows the first intermediate state after shifting $\gamma_1$. The next marked vertex is $V_2$, the non-trivial bump is $B_1$, and the dotted red segment $\gamma_2$ indicates the portion shifted in the next iteration.}
    \vspace{0.5cm}
\[
    \begin{minipage}{0.6\textwidth}
{\begin{tikzpicture}

    \draw (0,-0.75)--(0,4);
    \draw (1,-0.75)--(1,4);
    \draw (2,-0.75)--(2,4);
    \draw (3,-0.75)--(3,4);
    \draw (4, -0.75) -- (4,4);
    \draw (5, -0.75) -- (5,4);
    \draw (6, -0.75) -- (6,4);

    \draw (-1,0)--(6.5,0);
    \draw (-1,1)--(6.5,1);
    \draw (-1, 2) -- (6.5, 2);
    \draw (-1, 3) -- (6.5, 3);
    
    \draw[-{Stealth[length=3mm,width=2mm]}, Orange, line width=1pt] (-1,3)--(0,3);
    \draw[-{Stealth[length=3mm,width=2mm]}, Orange, line width=1pt] (0,3)--(0,4);

    \draw[-{Stealth[length=3mm,width=2mm]}, Green, line width=1pt] (-1,2)--(0,2);
    \draw[-{Stealth[length=3mm,width=2mm]}, Green, line width=1pt] (0,2)--(1,2);
    \draw[dotted, -{Stealth[length=3mm,width=2mm]}, Green, line width=2.5pt] (1,2)--(1,3);
    \draw[dotted, -{Stealth[length=3mm,width=2mm]}, Green, line width=2.5pt] (1,3)--(2,3);
    \draw[dotted, -{Stealth[length=3mm,width=2mm]}, Green, line width=2.5pt] (2,3)--(2,4);

   \draw[-{Stealth[length=3mm,width=2mm]}, red, line width=1pt] (-1,1)--(0,1);
    \draw[-{Stealth[length=3mm,width=2mm]}, red, line width=1pt] (0,1)--(1,1);
    \draw[-{Stealth[length=3mm,width=2mm]}, red, line width=1pt] (1,1)--(2,1);
    \draw[-{Stealth[length=3mm,width=2mm]}, red, line width=1pt] (2,1)--(2,2);
    \draw[-{Stealth[length=3mm,width=2mm]}, red, line width=1pt] (2,2)--(3,2);
    \draw[-{Stealth[length=3mm,width=2mm]}, red, line width=1pt] (3,2)--(3,3);
    \draw[-{Stealth[length=3mm,width=2mm]}, red, line width=1pt] (3,3)--(4,3);
    \draw[-{Stealth[length=3mm,width=2mm]}, red, line width=1pt] (4,3)--(5,3);
    \draw[-{Stealth[length=3mm,width=2mm]}, red, line width=1pt] (5,3)--(5,4);

    \draw[-{Stealth[length=3mm,width=2mm]}, blue, line width=1pt] (-1,0)--(0,0);
    \draw[-{Stealth[length=3mm,width=2mm]}, blue, line width=1pt] (0,0)--(1,0);
    \draw[-{Stealth[length=3mm,width=2mm]}, blue, line width=1pt] (1,0)--(2,0);
    \draw[-{Stealth[length=3mm,width=2mm]}, blue, line width=1pt] (2,0)--(3,0);
    \draw[-{Stealth[length=3mm,width=2mm]}, blue, line width=1pt] (3,0)--(3,1);
    \draw[-{Stealth[length=3mm,width=2mm]}, blue, line width=1pt] (3,1)--(4,1);
    \draw[-{Stealth[length=3mm,width=2mm]}, blue, line width=1pt] (4,1)--(4,2);
    \draw[-{Stealth[length=3mm,width=2mm]}, blue, line width=1pt] (4,2)--(5,2);
    \draw[-{Stealth[length=3mm,width=2mm]}, blue, line width=1pt] (5,2)--(5,3);
    \draw[-{Stealth[length=3mm,width=2mm]}, blue, line width=1pt] (5,3)--(6,3);
    \draw[-{Stealth[length=3mm,width=2mm]}, blue, line width=1pt] (6,3)--(6,4);

    \node at (-1.25,0) {$ z_1$};
    \node at (-1.25,1) {$ z_2$};
    \node at (-1.25, 2) {$z_3$};
    \node at (-1.25, 3) {$z_4$};

    \node[circle, draw=violet, fill=violet, opacity=0.4, ultra thick, minimum size=7pt, inner sep=2pt] at (0,0){};
    \node[circle, draw=violet, fill=violet, opacity=0.4, ultra thick, minimum size=7pt, inner sep=2pt] at (1,0){};
    \node[circle, draw=violet, fill=violet, opacity=0.4, ultra thick, minimum size=7pt, inner sep=2pt] at (2,0){};
    \node[circle, draw=violet, fill=violet, opacity=0.4, ultra thick, minimum size=7pt, inner sep=2pt] at (0,1){};
    \node[circle, draw=violet, fill=violet, opacity=0.4, ultra thick, minimum size=7pt, inner sep=2pt] at (4,3){};
    \node[circle, draw=violet, fill=violet, opacity=0.4, ultra thick, minimum size=7pt, inner sep=2pt] at (0,2){}; 
    \node[circle, draw=violet, fill=violet, opacity=0.4, ultra thick, minimum size=7pt, inner sep=2pt] at (1,1){};

 
    \node[cross out, draw=violet, ultra thick, minimum size=10pt, inner sep=2pt] at (3,3){};
    \node[cross out, draw=violet, ultra thick, minimum size=10pt, inner sep=2pt] at (4,2){};
    \node[cross out, draw=violet, ultra thick, minimum size=10pt, inner sep=2pt] at (1,3){};
    \node[cross out, draw=violet, ultra thick, minimum size=10pt, inner sep=2pt] at (3,1){};

   \node[circle, draw=violet, thick, minimum size=10pt] at (1,2){};
    \node at (1.25,1.6) {$V_3$};
    
    \node[circle, draw = violet, ultra thick, minimum size=10pt, inner sep=2pt] at (2,2){};
    \node at (2.3,1.6) {$B_2$};
    
    \node at (2.4,3.45) {$\gamma_3$};

    \draw [-Stealth](7.5,1.5) -- (8.5,1.5);   
    \node at (8,1.75) {$\psi$};

  \end{tikzpicture}}
    \end{minipage}
    \begin{minipage}{0.3\textwidth}
$\qquad \; \; \ytableausetup{boxsize=2.2em}
       \ytableaushort{1 1 1, 2 {2,3} 4,3}  \; = \; T_2 $
    \end{minipage}
\]
    \commentAlt{The decorated state after the second iteration of the lattice insertion algorithm, together with the corresponding tableau $T_2$.}
    \commentLongAlt{The lattice state shows the second intermediate state. The next marked vertex is $V_3$, the non-trivial bump is $B_2$, and the dotted green segment $\gamma_3$ indicates the portion shifted in the final iteration.}
\vspace{0.5cm}
\[
    \begin{minipage}{0.6\textwidth}
{\begin{tikzpicture}

    \draw (0,-0.75)--(0,4);
    \draw (1,-0.75)--(1,4);
    \draw (2,-0.75)--(2,4);
    \draw (3,-0.75)--(3,4);
    \draw (4, -0.75) -- (4,4);
    \draw (5, -0.75) -- (5,4);
    \draw (6, -0.75) -- (6,4);

    \draw (-1,0)--(6.5,0);
    \draw (-1,1)--(6.5,1);
    \draw (-1, 2) -- (6.5, 2);
    \draw (-1, 3) -- (6.5, 3);
    
    \draw[-{Stealth[length=3mm,width=2mm]}, Orange, line width=1pt] (-1,3)--(0,3);
    \draw[-{Stealth[length=3mm,width=2mm]}, Orange, line width=1pt] (0,3)--(0,4);

    \draw[-{Stealth[length=3mm,width=2mm]}, Green, line width=1pt] (-1,2)--(0,2);
    \draw[-{Stealth[length=3mm,width=2mm]}, Green, line width=1pt] (0,2)--(1,2);
    \draw[-{Stealth[length=3mm,width=2mm]}, Green, line width=1pt] (1,2)--(2,2);
    \draw[-{Stealth[length=3mm,width=2mm]}, Green, line width=1pt] (2,2)--(2,3);
    \draw[-{Stealth[length=3mm,width=2mm]}, Green, line width=1pt] (2,3)--(3,3);
    \draw[-{Stealth[length=3mm,width=2mm]}, Green, line width=1pt] (3,3)--(3,4);

   \draw[-{Stealth[length=3mm,width=2mm]}, red, line width=1pt] (-1,1)--(0,1);
    \draw[-{Stealth[length=3mm,width=2mm]}, red, line width=1pt] (0,1)--(1,1);
    \draw[-{Stealth[length=3mm,width=2mm]}, red, line width=1pt] (1,1)--(2,1);
    \draw[-{Stealth[length=3mm,width=2mm]}, red, line width=1pt] (2,1)--(2,2);
    \draw[-{Stealth[length=3mm,width=2mm]}, red, line width=1pt] (2,2)--(3,2);
    \draw[-{Stealth[length=3mm,width=2mm]}, red, line width=1pt] (3,2)--(3,3);
    \draw[-{Stealth[length=3mm,width=2mm]}, red, line width=1pt] (3,3)--(4,3);
    \draw[-{Stealth[length=3mm,width=2mm]}, red, line width=1pt] (4,3)--(5,3);
    \draw[-{Stealth[length=3mm,width=2mm]}, red, line width=1pt] (5,3)--(5,4);

    \draw[-{Stealth[length=3mm,width=2mm]}, blue, line width=1pt] (-1,0)--(0,0);
    \draw[-{Stealth[length=3mm,width=2mm]}, blue, line width=1pt] (0,0)--(1,0);
    \draw[-{Stealth[length=3mm,width=2mm]}, blue, line width=1pt] (1,0)--(2,0);
    \draw[-{Stealth[length=3mm,width=2mm]}, blue, line width=1pt] (2,0)--(3,0);
    \draw[-{Stealth[length=3mm,width=2mm]}, blue, line width=1pt] (3,0)--(3,1);
    \draw[-{Stealth[length=3mm,width=2mm]}, blue, line width=1pt] (3,1)--(4,1);
    \draw[-{Stealth[length=3mm,width=2mm]}, blue, line width=1pt] (4,1)--(4,2);
    \draw[-{Stealth[length=3mm,width=2mm]}, blue, line width=1pt] (4,2)--(5,2);
    \draw[-{Stealth[length=3mm,width=2mm]}, blue, line width=1pt] (5,2)--(5,3);
    \draw[-{Stealth[length=3mm,width=2mm]}, blue, line width=1pt] (5,3)--(6,3);
    \draw[-{Stealth[length=3mm,width=2mm]}, blue, line width=1pt] (6,3)--(6,4);

    \node at (-1.25,0) {$ z_1$};
    \node at (-1.25,1) {$ z_2$};
    \node at (-1.25, 2) {$z_3$};
    \node at (-1.25, 3) {$z_4$};

    \node[circle, draw=violet, fill=violet, opacity=0.4, ultra thick, minimum size=7pt, inner sep=2pt] at (0,0){};
    \node[circle, draw=violet, fill=violet, opacity=0.4, ultra thick, minimum size=7pt, inner sep=2pt] at (1,0){};
    \node[circle, draw=violet, fill=violet, opacity=0.4, ultra thick, minimum size=7pt, inner sep=2pt] at (2,0){};
    \node[circle, draw=violet, fill=violet, opacity=0.4, ultra thick, minimum size=7pt, inner sep=2pt] at (0,1){};
    \node[circle, draw=violet, fill=violet, opacity=0.4, ultra thick, minimum size=7pt, inner sep=2pt] at (1,1){};
    \node[circle, draw=violet, fill=violet, opacity=0.4, ultra thick, minimum size=7pt, inner sep=2pt] at (4,3){};
    \node[circle, draw=violet, fill=violet, opacity=0.4, ultra thick, minimum size=7pt, inner sep=2pt] at (0,2){};
    \node[circle, draw=violet, fill=violet, opacity=0.4, ultra thick, minimum size=7pt, inner sep=2pt] at (1,2){};
    
    \node[cross out, draw=violet, ultra thick, minimum size=10pt, inner sep=2pt] at (4,2){};
    \node[cross out, draw=violet, ultra thick, minimum size=10pt, inner sep=2pt] at (3,1){};
    \node[cross out, draw=violet, ultra thick, minimum size=10pt, inner sep=2pt] at (2,3){};

    \draw [-Stealth](7.5,1.5) -- (8.5,1.5);   
    \node at (8,1.75) {$\psi$};
  \end{tikzpicture}}
    \end{minipage}
    \begin{minipage}{0.3\textwidth}
      $\qquad \; \; \ytableausetup{boxsize=2.2em}
      \ytableaushort{1 1 1, 2 2 4, 3 3}  \; = \; T_3 $
    \end{minipage}
\]
    \commentAlt{The final decorated state produced by the lattice insertion algorithm, together with the corresponding tableau $T_3$.}
    \commentLongAlt{The lattice state shows the final output of the lattice insertion algorithm. No non-trivial bump remains, and the corresponding tableau is $T_3$.}
\end{example}

\begin{proposition}
\label{prop:lattice_INSERTION_algorithm}
    Let $T$ be a semistandard tableau of shape $\lambda$ with entries in $[n]$, and let $S^\text{dec} \in \decoratedstate^0_\lambda$ be the corresponding 
    trivial decorated state. For $u\in [n]$, the lattice insertion algorithm produces the trivial decorated state corresponding to $T \leftarrow u$.
\end{proposition}

\begin{proof}
By Lemma~\ref{lem:paths-to-tab}, the natural path $L_1$ corresponds to the first row of $T$. The vertices of type $\tt b_2$ along $L_1$ in lattice 
row $i$ correspond to the entries equal to $i$ in the first row of $T$. By identifying $V_1$ in Step 1 as the vertex of type $\tt c_1$ in row $u,$ we have 
determined the vertex that precedes any $\tt b_2$ vertex in row $r$ where $r > u.$ 
This is equivalent to identifying the smallest entry $x$ in row $1$ of $T$ such that $x > u$, if such an entry exists.   
    
Identify $\gamma_1$ as in Step 2. 
If $\gamma_1$ terminates at an exit arrow, then there is no $\tt b_2$ vertex after $V_1$ along $L_1.$ 
Equivalently, by the bijection $\psi$, there is no entry greater than $u$ in the first row of $T.$ By shifting $\gamma_1,$ we create a $\tt b_2$ vertex in 
row $u$ and shift the exit arrow one column to the right. This corresponds to adding a box with entry $u$ to the end of row $1$ of $T.$ 

Now suppose $\gamma_1$ terminates at a $\tt b_2$ vertex. Then there exists an entry $x$ in the first row of $T$ such that $x > u.$ Since $\gamma_1$ terminates at the first encountered $\tt b_2$ vertex, this entry $x$ is the smallest such entry in the first row of $T$ greater than $u$.
Shifting $\gamma_1$ creates a $\tt b_2$ vertex in row $u$, which corresponds to inserting $u$ into row $1$ of $T$. The terminal $\tt b_2$ vertex is changed into a non-trivial $\tt a_2$ vertex, $B_1$, which records the entry $x$ that is temporarily an excess entry. Thus, after the first iteration, the decorated state corresponds to the intermediate set-valued tableau where $u$ and $x$ lie in the same multicell.

The algorithm then identifies $V_2$ on the next natural path $L_2$, which corresponds to inserting the bumped entry $x$ into the second row of $T$. In the next iteration, shifting $\gamma_2$ creates a $\tt b_2$ vertex on $L_2$ in the lattice row corresponding to the entry $x$, so $x$ is inserted into row $2$ of $T$.
After shifting $\gamma_2$, the previously created bump $B_1$ is resolved. It either becomes a $\tt b_1$ vertex or remains a bump, in which case it is declared to have trivial weight. In either case, $B_1$ no longer represents an excess entry, matching the fact that $x$ has been moved into the next row. 

The same argument holds inductively. At iteration $k>1$, the algorithm inserts the entry bumped from row $k-1$ into row $k$. If the shifted path $\gamma_k$ terminates at a $\tt b_2$ vertex, then the corresponding entry is bumped into the next row and recorded by the new non-trivial bump $B_k$. If $\gamma_k$ terminates at an exit arrow, then no entry is bumped, and the shifted exit arrow corresponds to adding a new box to row $k$ of $T$. Thus, each iteration of the lattice insertion algorithm coincides with one step of row insertion, and the final lattice state is the trivial decorated state corresponding to $T \leftarrow u$.
\end{proof}

\begin{proposition}
    Let $w= w_1 w_2 \ldots w_\ell$ be a word and let $\RSK(w)=(P,Q)$. Then $(P,Q)$ can be reconstructed from the lattice model 
    after successively lattice inserting the letters of $w$ into the admissible state for $\emptyset.$
\end{proposition}

\begin{proof}
    Let $(P_i, Q_i)$ denote the image of $\RSK(w_1 w_2 \ldots w_i).$ Let $S_0^\text{dec}$ be the trivial decorated state corresponding to the empty tableau and $S_1^\text{dec}, S_2^\text{dec}, \ldots, S_\ell^\text{dec}$ be the sequence of decorated states obtained from successively lattice inserting the letters of $w.$ That is, $S_i^{\text{dec}}$ is the 
    decorated state corresponding to $\emptyset \leftarrow w_1\leftarrow\cdots\leftarrow w_i.$ By Proposition~\ref{prop:lattice_INSERTION_algorithm}, 
    $S_i^{\text{dec}}$ corresponds to the tableau $P_i.$ Each step from $S_i^{\text{dec}}$ to $S_{i+1}^{\text{dec}}$ is one application of the lattice insertion algorithm which moves exactly 
    one exit arrow. Therefore, we may construct $Q_{i+1}$ recursively by adding one box to $Q_{i}$ in the location of $P_{i+1}/P_i$ and filling it with $i+1.$
\end{proof}

\begin{example}
    Consider the word $w=312.$ For each letter insertion step, we give the corresponding admissible state and the image of the RSK correspondence 
    for semistandard Young tableaux below:

 \[
    \begin{minipage}{0.25\linewidth}
      $\emptyset: $
  \end{minipage}  
\begin{minipage}{0.4\linewidth}
{\begin{tikzpicture}

    \draw (0,-0.75)--(0,3);
    \draw (1,-0.75)--(1,3);
    \draw (2,-0.75)--(2,3);
    \draw (3,-0.75)--(3,3);
    \draw (4, -0.75) -- (4,3);

    \draw (-1,0)--(4.5,0);
    \draw (-1,1)--(4.5,1);
    \draw (-1, 2) -- (4.5, 2);
   
   
    \draw[-{Stealth[length=3mm,width=2mm]}, Green, line width=1pt] (-1,2)--(0,2);
    \draw[-{Stealth[length=3mm,width=2mm]}, Green, line width=1pt] (0,2)--(0,3);
    
     \draw[-{Stealth[length=3mm,width=2mm]},red, line width=1.pt] (-1,1)--(0,1);
     \draw[-{Stealth[length=3mm,width=2mm]},red, line width=1pt] (0,1)--(0,2);
     \draw[-{Stealth[length=3mm,width=2mm]},red, line width = 1pt] (0,2)--(1,2);
    \draw[-{Stealth[length=3mm,width=2mm]},red, line width=1pt] (1,2)--(1,3);
 
    \draw[-{Stealth[length=3mm,width=2mm]},blue,line width=1pt] (-1,0)--(0,0);
     \draw[-{Stealth[length=3mm,width=2mm]},blue, line width=1pt] (0,0)--(0,1);
    \draw[-{Stealth[length=3mm,width=2mm]},blue, line width=1pt] (0,1)--(1,1);
     \draw[-{Stealth[length=3mm,width=2mm]},blue, line width=1pt] (1,1)--(1,2);
     \draw[-{Stealth[length=3mm,width=2mm]},blue, line width=1pt] (1,2)--(2,2);
      \draw[-{Stealth[length=3mm,width=2mm]},blue, line width=1pt] (2,2)--(2,3);

    \node at (-1.25,0) {$ z_1$};
    \node at (-1.25,1) {$ z_2$};
    \node at (-1.25, 2) {$z_3$};
    
  \end{tikzpicture}}
  \end{minipage}\qquad
  \begin{minipage}{0.25\linewidth}
      $(\emptyset, \emptyset)$
  \end{minipage}\]
  
 \[\begin{minipage}{0.25\linewidth}
      $\emptyset \leftarrow 3: $
  \end{minipage}  
\begin{minipage}{0.4\linewidth}
{\begin{tikzpicture}

    \draw (0,-0.75)--(0,3);
    \draw (1,-0.75)--(1,3);
    \draw (2,-0.75)--(2,3);
    \draw (3,-0.75)--(3,3);
    \draw (4, -0.75) -- (4,3);

    \draw (-1,0)--(4.5,0);
    \draw (-1,1)--(4.5,1);
    \draw (-1, 2) -- (4.5, 2);


    \draw[-{Stealth[length=3mm,width=2mm]}, Green, line width=1pt] (-1,2)--(0,2);
    \draw[-{Stealth[length=3mm,width=2mm]}, Green, line width=1pt] (0,2)--(0,3);

     \draw[-{Stealth[length=3mm,width=2mm]},red, line width=1.pt] (-1,1)--(0,1);
     \draw[-{Stealth[length=3mm,width=2mm]},red, line width=1pt] (0,1)--(0,2);
     \draw[-{Stealth[length=3mm,width=2mm]},red, line width = 1pt] (0,2)--(1,2);
    \draw[-{Stealth[length=3mm,width=2mm]},red, line width=1pt] (1,2)--(1,3);

    \draw[-{Stealth[length=3mm,width=2mm]},blue,line width=1pt] (-1,0)--(0,0);
     \draw[-{Stealth[length=3mm,width=2mm]},blue, line width=1pt] (0,0)--(0,1);
    \draw[-{Stealth[length=3mm,width=2mm]},blue, line width=1pt] (0,1)--(1,1);
     \draw[-{Stealth[length=3mm,width=2mm]},blue, line width=1pt] (1,1)--(1,2);
     \draw[-{Stealth[length=3mm,width=2mm]},blue, line width=1pt] (1,2)--(2,2);
      \draw[-{Stealth[length=3mm,width=2mm]},blue, line width=1pt] (2,2)--(3,2);
       \draw[-{Stealth[length=3mm,width=2mm]},blue, line width=1pt] (3,2)--(3,3);

    \node at (-1.25,0) {$ z_1$};
    \node at (-1.25,1) {$ z_2$};
    \node at (-1.25, 2) {$z_3$};

     \node[circle, draw = violet,fill=violet,opacity = 0.4, ultra thick, minimum size=7pt, inner sep=2pt] (c) at (2,2){};    

  \end{tikzpicture}}
  \end{minipage}\qquad
  \begin{minipage}{0.25\linewidth}
      $\left(~\ytableaushort{3} ~,~ \ytableaushort{1}~\right)$
  \end{minipage}\]
\[\begin{minipage}{0.25\linewidth}
      $(\emptyset \leftarrow 3)\leftarrow 1: $
  \end{minipage}  
\begin{minipage}{0.4\linewidth}
{\begin{tikzpicture}

    \draw (0,-0.75)--(0,3);
    \draw (1,-0.75)--(1,3);
    \draw (2,-0.75)--(2,3);
    \draw (3,-0.75)--(3,3);
    \draw (4, -0.75) -- (4,3);

    \draw (-1,0)--(4.5,0);
    \draw (-1,1)--(4.5,1);
    \draw (-1, 2) -- (4.5, 2);


    \draw[-{Stealth[length=3mm,width=2mm]}, Green, line width=1pt] (-1,2)--(0,2);
    \draw[-{Stealth[length=3mm,width=2mm]}, Green, line width=1pt] (0,2)--(0,3);
   
     \draw[-{Stealth[length=3mm,width=2mm]},red, line width=1.pt] (-1,1)--(0,1);
     \draw[-{Stealth[length=3mm,width=2mm]},red, line width=1pt] (0,1)--(0,2);
     \draw[-{Stealth[length=3mm,width=2mm]},red, line width = 1pt] (0,2)--(1,2);
    \draw[-{Stealth[length=3mm,width=2mm]},red, line width=1pt] (1,2)--(2,2);
     \draw[-{Stealth[length=3mm,width=2mm]},red, line width=1pt] (2,2)--(2,3);

    \draw[-{Stealth[length=3mm,width=2mm]},blue,line width=1pt] (-1,0)--(0,0);
     \draw[-{Stealth[length=3mm,width=2mm]},blue, line width=1pt] (0,0)--(1,0);
    \draw[-{Stealth[length=3mm,width=2mm]},blue, line width=1pt] (1,0)--(1,1);
     \draw[-{Stealth[length=3mm,width=2mm]},blue, line width=1pt] (1,1)--(2,1);
     \draw[-{Stealth[length=3mm,width=2mm]},blue, line width=1pt] (2,1)--(2,2);
      \draw[-{Stealth[length=3mm,width=2mm]},blue, line width=1pt] (2,2)--(3,2);
       \draw[-{Stealth[length=3mm,width=2mm]},blue, line width=1pt] (3,2)--(3,3);

    \node at (-1.25,0) {$ z_1$};
    \node at (-1.25,1) {$ z_2$};
    \node at (-1.25, 2) {$z_3$};

     \node[circle, draw = violet,fill=violet,opacity = 0.4, ultra thick, minimum size=7pt, inner sep=2pt] (c) at (1,2){};   
    \node[circle, draw = violet,fill=violet,opacity = 0.4, ultra thick, minimum size=7pt, inner sep=2pt] (c) at (0,0){};    
    \node[cross out, draw = violet, ultra thick, minimum size=10pt, inner sep=2pt] (c) at (1,1){};

  \end{tikzpicture}}
  \end{minipage}\qquad
    \begin{minipage}{0.25\linewidth}
      $\left(~\ytableaushort{1,3} ~,~ \ytableaushort{1,2}~\right)$
  \end{minipage}\]
\[\begin{minipage}{0.25\linewidth}
      $((\emptyset \leftarrow 3)\leftarrow 1)\leftarrow 2: $
  \end{minipage}  
 \begin{minipage}{0.4\linewidth}
{\begin{tikzpicture}

    \draw (0,-0.75)--(0,3);
    \draw (1,-0.75)--(1,3);
    \draw (2,-0.75)--(2,3);
    \draw (3,-0.75)--(3,3);
    \draw (4, -0.75) -- (4,3);

    \draw (-1,0)--(4.5,0);
    \draw (-1,1)--(4.5,1);
    \draw (-1, 2) -- (4.5, 2);


    \draw[-{Stealth[length=3mm,width=2mm]}, Green, line width=1pt] (-1,2)--(0,2);
    \draw[-{Stealth[length=3mm,width=2mm]}, Green, line width=1pt] (0,2)--(0,3);
   
     \draw[-{Stealth[length=3mm,width=2mm]},red, line width=1.pt] (-1,1)--(0,1);
     \draw[-{Stealth[length=3mm,width=2mm]},red, line width=1pt] (0,1)--(0,2);
     \draw[-{Stealth[length=3mm,width=2mm]},red, line width = 1pt] (0,2)--(1,2);
    \draw[-{Stealth[length=3mm,width=2mm]},red, line width=1pt] (1,2)--(2,2);
     \draw[-{Stealth[length=3mm,width=2mm]},red, line width=1pt] (2,2)--(2,3);

    \draw[-{Stealth[length=3mm,width=2mm]},blue,line width=1pt] (-1,0)--(0,0);
     \draw[-{Stealth[length=3mm,width=2mm]},blue, line width=1pt] (0,0)--(1,0);
    \draw[-{Stealth[length=3mm,width=2mm]},blue, line width=1pt] (1,0)--(1,1);
     \draw[-{Stealth[length=3mm,width=2mm]},blue, line width=1pt] (1,1)--(2,1);
     \draw[-{Stealth[length=3mm,width=2mm]},blue, line width=1pt] (2,1)--(3,1);
      \draw[-{Stealth[length=3mm,width=2mm]},blue, line width=1pt] (3,1)--(3,2);
       \draw[-{Stealth[length=3mm,width=2mm]},blue, line width=1pt] (3,2)--(4,2);
        \draw[-{Stealth[length=3mm,width=2mm]},blue, line width=1pt] (4,2)--(4,3);

    \node at (-1.25,0) {$ z_1$};
    \node at (-1.25,1) {$ z_2$};
    \node at (-1.25, 2) {$z_3$};

    \node[circle, draw = violet,fill=violet,opacity = 0.4, ultra thick, minimum size=7pt, inner sep=2pt] (c) at (1,2){};   
    \node[circle, draw = violet,fill=violet,opacity = 0.4, ultra thick, minimum size=7pt, inner sep=2pt] (c) at (0,0){};    
    \node[circle, draw = violet,fill=violet,opacity = 0.4, ultra thick, minimum size=7pt, inner sep=2pt] (c) at (2,1){};  
    \node[cross out, draw = violet, ultra thick, minimum size=10pt, inner sep=2pt] (c) at (1,1){};
    \node[cross out, draw = violet, ultra thick, minimum size=10pt, inner sep=2pt] (c) at (3,2){};
    
  \end{tikzpicture}}
  \end{minipage}\quad
      \begin{minipage}{0.25\linewidth}
      $\left(~\ytableaushort{1 2,3} ~,~ \ytableaushort{1 3,2}~\right)$
  \end{minipage}\]
 
\end{example}

\section{Uncrowding in the lattice model setting}
\label{sec:uncrowd-lattice-algo}

In this section we define a local procedure on decorated states that realizes Buch's uncrowding map on set-valued tableaux. 
Given $T\in \SVT^n(\lambda)$, let $S^\text{dec} \in \decoratedstate_\lambda$ be the corresponding decorated state under \eqref{eq:bij-decoratedstate-SVT}.

Recall that a bump vertex $\tt a_2$ in Figure~\ref{fig:uncolored_wts} is non-trivial when it contributes the $\beta z_i$ term in its weight from $1+\beta z_i$. 
Under the bijection with set-valued tableaux, 
non-trivial bumps encode the crowded entries of $T$, which we make precise in Lemma~\ref{lem:paths-to-SVTtab}. The lattice uncrowding 
algorithm successively removes these non-trivial bumps through local moves, in the same order that Buch's uncrowding map removes crowded entries. 
The resulting decorated state corresponds to the semistandard tableau $P_\SVT(T)$ of Definition~\ref{def:uncrowding-SVT}.

\begin{lemma}\label{lem:paths-to-SVTtab}
Let $T\in \SVT^n(\lambda)$ and let $S^\text{dec} \in \decoratedstate_\lambda$ be the corresponding decorated state. 
For each $1\leqslant i \leqslant n$, the natural path $L_i$ corresponds to row $i$ of $T$. Moreover, the non-trivial bumps along $L_i$, read from top to bottom, 
are in bijection with the crowded entries of row $i$ of $T$, read from right to left. 
\end{lemma}

\begin{proof} 
The first claim follows from Lemma~\ref{lem:paths-to-tab}.

For the second claim, a non-trivial bump along $L_i$ in lattice row $j$ occurs precisely when the entry $\lambda^{(j)}_{i}$ is marked in the MGT pattern. 
Thus, the marked entries in the $i$-th column of the MGT pattern, read from top to bottom, correspond under $\phi$ to the crowded entries in row $i$ of 
$T$, read from right to left.
\end{proof}


We now define the \defn{lattice uncrowding algorithm} on $S^\text{dec} \in \decoratedstate_\lambda$, a lattice version of Buch's uncrowding procedure that 
produces the trivial decorated state corresponding to $P_\SVT(T)$.

The algorithm iterates through the natural paths $L_n, \dots, L_1$ in decreasing order. For each path $L_r$, it removes the non-trivial bumps on $L_r$ 
from top to bottom through local moves. Each such move is carried out by applying the lattice insertion algorithm from the same lattice row as the chosen 
bump, starting on the adjacent path $L_{r+1}$. If this process creates additional non-trivial bumps on adjacent paths, those bumps are handled within the 
same insertion process before the algorithm returns to $L_r$.

During the algorithm, we keep track of the non-trivial bumps on each path using slide sequences. Once a bump has been processed, it is removed 
from its slide sequence. The temporary bumps created during a lattice insertion step are not added to these slide sequences.

The lattice uncrowding algorithm is as follows. 
\medskip

\begin{enumerate}[label=\textbf{Step \arabic*}, leftmargin=*]

    \item Let $r=n$. For each natural path $L_i$ starting in lattice row $1\leqslant i\leqslant n$, define a slide sequence $\Omega_i$ as the list of all 
    non-trivial $\tt a_2$ vertices traversed by $L_i$, ordered from top to bottom, equivalently by decreasing lattice rows. Thus, the first element in 
    $\Omega_i$ is the topmost bump along $L_i$. Note that $\Omega_i$ may be empty. 

    \item If $\Omega_r$ is empty, skip to Step 6.

    \item Let $B$ denote the first bump in $\Omega_r$, and suppose $B$ lies in lattice row $j$. In row $j$, locate the rightmost vertex of type $\tt c_1$ 
    on $L_{r+1}$ strictly to the left of $B$, and denote it by $V$.

     \item Starting with $k=r+1$, perform Steps 2--5 of the lattice insertion algorithm using the vertex $V_{k} = V$ from Step 3 and setting $B_{k-1} \coloneq B.$ 
     
    \item Remove the original bump $B$ chosen in Step 3 from $\Omega_r.$
    \item If $\Omega_r$ is nonempty, return to Step 3. 
    If $\Omega_r$ is empty and $r>1$, decrease $r$ by 1 and return to Step 2.
    If $\Omega_r$ is empty and $r=1$, stop. 
\end{enumerate}

Each application of Step 4 shifts one exit arrow and therefore changes the shape by adding one box. Thus, after all non-trivial bumps have been 
removed, the final state has some shape $\mu\supseteq\lambda$.

The output of the lattice uncrowding algorithm defines the \defn{lattice uncrowding map}
\[ \U_{\LAT}\colon \decoratedstate_\lambda \longrightarrow \bigsqcup_{\mu \supseteq \lambda} \decoratedstate^0_\mu,\]
where $\U_\LAT(S^\text{dec})$ is the final decorated state produced by the algorithm.

\begin{remark}
    Since $L_n$ is the northwestern-most natural path, it contains no $\tt a_2$ vertices. Hence $\Omega_n=\emptyset$, so Step 3 is never invoked when $r=n$.
\end{remark}

\begin{example} 
    Let $S^\text{dec}$ be a decorated state with one non-trivial bump (left), and $\hat{S}^{\text{dec}}$ be a trivial decorated state (right) after Step 4 of the lattice uncrowding algorithm.

\[
\begin{minipage}{0.49\linewidth}
\begin{tikzpicture} 

    \draw (0,-0.75)--(0,4);
    \draw (1,-0.75)--(1,4);
    \draw (2,-0.75)--(2,4);
    \draw (3,-0.75)--(3,4);
    \draw (4, -0.75) -- (4,4);
    \draw (5, -0.75) -- (5,4);
    \draw (6, -0.75) -- (6,4);

    \draw (-1,0)--(6.5,0);
    \draw (-1,1)--(6.5,1);
    \draw (-1, 2) -- (6.5, 2);
    \draw (-1, 3) -- (6.5, 3);
    
    \draw[-{Stealth[length=3mm,width=2mm]}, Orange, line width=1pt] (-1,3)--(0,3);
    \draw[-{Stealth[length=3mm,width=2mm]}, Orange, line width=1pt] (0,3)--(0,4);

    \draw[-{Stealth[length=3mm,width=2mm]}, Green, line width=1pt] (-1,2)--(0,2);
    \draw[-{Stealth[length=3mm,width=2mm]}, Green, line width=1pt] (0,2)--(1,2);
    \draw[-{Stealth[length=3mm,width=2mm]}, Green, line width=1pt] (1,2)--(1,3);
    \draw[-{Stealth[length=3mm,width=2mm]}, Green, line width=1pt] (1,3)--(2,3);
    \draw[-{Stealth[length=3mm,width=2mm]}, Green, line width=1pt] (2,3)--(2,4);

    \draw[-{Stealth[length=3mm,width=2mm]}, red, line width=1pt] (-1,1)--(0,1);
    \draw[-{Stealth[length=3mm,width=2mm]}, red, line width=1pt] (0,1)--(1,1);
    \draw[dotted, -{Stealth[length=3mm,width=2mm]}, red, line width=2.5pt] (1,1)--(1,2);
    \draw[dotted, -{Stealth[length=3mm,width=2mm]}, red, line width=2.5pt] (1,2)--(2,2);
    \draw[-{Stealth[length=3mm,width=2mm]}, red, line width=1pt] (2,2)--(3,2);
    \draw[-{Stealth[length=3mm,width=2mm]}, red, line width=1pt] (3,2)--(3,3);
    \draw[-{Stealth[length=3mm,width=2mm]}, red, line width=1pt] (3,3)--(4,3);
    \draw[-{Stealth[length=3mm,width=2mm]}, red, line width=1pt] (4,3)--(5,3);
    \draw[-{Stealth[length=3mm,width=2mm]}, red, line width=1pt] (5,3)--(5,4);

    \draw[-{Stealth[length=3mm,width=2mm]}, blue, line width=1pt] (-1,0)--(0,0);
    \draw[-{Stealth[length=3mm,width=2mm]}, blue, line width=1pt] (0,0)--(1,0);
    \draw[-{Stealth[length=3mm,width=2mm]}, blue, line width=1pt] (1,0)--(2,0);
    \draw[-{Stealth[length=3mm,width=2mm]}, blue, line width=1pt] (2,0)--(3,0);
    \draw[-{Stealth[length=3mm,width=2mm]}, blue, line width=1pt] (3,0)--(3,1);
    \draw[-{Stealth[length=3mm,width=2mm]}, blue, line width=1pt] (3,1)--(4,1);
    \draw[-{Stealth[length=3mm,width=2mm]}, blue, line width=1pt] (4,1)--(4,2);
    \draw[-{Stealth[length=3mm,width=2mm]}, blue, line width=1pt] (4,2)--(5,2);
    \draw[-{Stealth[length=3mm,width=2mm]}, blue, line width=1pt] (5,2)--(5,3);
    \draw[-{Stealth[length=3mm,width=2mm]}, blue, line width=1pt] (5,3)--(6,3);
    \draw[-{Stealth[length=3mm,width=2mm]}, blue, line width=1pt] (6,3)--(6,4);

    \node at (-1.25,0) {$ z_1$};
    \node at (-1.25,1) {$ z_2$};
    \node at (-1.25, 2) {$z_3$};
    \node at (-1.25, 3) {$z_4$};

    \node[circle, draw=violet, fill=violet, opacity=0.4, ultra thick, minimum size=7pt, inner sep=2pt] at (0,0){};
    \node[circle, draw=violet, fill=violet, opacity=0.4, ultra thick, minimum size=7pt, inner sep=2pt] at (1,0){};
    \node[circle, draw=violet, fill=violet, opacity=0.4, ultra thick, minimum size=7pt, inner sep=2pt] at (2,0){};
    \node[circle, draw=violet, fill=violet, opacity=0.4, ultra thick, minimum size=7pt, inner sep=2pt] at (0,1){};
    \node[circle, draw=violet, fill=violet, opacity=0.4, ultra thick, minimum size=7pt, inner sep=2pt] at (2,2){};
    \node[circle, draw=violet, fill=violet, opacity=0.4, ultra thick, minimum size=7pt, inner sep=2pt] at (4,3){};
    \node[circle, draw=violet, fill=violet, opacity=0.4, ultra thick, minimum size=7pt, inner sep=2pt] at (0,2){}; 

    \node[cross out, draw=violet, ultra thick, minimum size=10pt, inner sep=2pt] at (3,3){};
    \node[cross out, draw=violet, ultra thick, minimum size=10pt, inner sep=2pt] at (4,2){};
    \node[cross out, draw=violet, ultra thick, minimum size=10pt, inner sep=2pt] at (1,3){};

    \node[circle, draw=violet, thick, minimum size=10pt] at (1,1){};
    \node at (1.85,0.6) {$V= V_2$};

    \node[circle, draw = violet, ultra thick, minimum size=10pt, inner sep=2pt] at (3,1){};
    \node at (3.9,0.6) {$B=B_1$};

    \node at (1.45,2.35) {$\gamma_2$};
  \end{tikzpicture}
  \end{minipage} \qquad
  \begin{minipage}{0.49\linewidth}
  \begin{tikzpicture}

    \draw (0,-0.75)--(0,4);
    \draw (1,-0.75)--(1,4);
    \draw (2,-0.75)--(2,4);
    \draw (3,-0.75)--(3,4);
    \draw (4, -0.75) -- (4,4);
    \draw (5, -0.75) -- (5,4);
    \draw (6, -0.75) -- (6,4);

    \draw (-1,0)--(6.5,0);
    \draw (-1,1)--(6.5,1);
    \draw (-1, 2) -- (6.5, 2);
    \draw (-1, 3) -- (6.5, 3);
    
    \draw[-{Stealth[length=3mm,width=2mm]}, Orange, line width=1pt] (-1,3)--(0,3);
    \draw[-{Stealth[length=3mm,width=2mm]}, Orange, line width=1pt] (0,3)--(0,4);

    \draw[-{Stealth[length=3mm,width=2mm]}, Green, line width=1pt] (-1,2)--(0,2);
    \draw[-{Stealth[length=3mm,width=2mm]}, Green, line width=1pt] (0,2)--(1,2);
    \draw[-{Stealth[length=3mm,width=2mm]}, Green, line width=1pt] (1,2)--(2,2);
    \draw[-{Stealth[length=3mm,width=2mm]}, Green, line width=1pt] (2,2)--(2,3);
    \draw[-{Stealth[length=3mm,width=2mm]}, Green, line width=1pt] (2,3)--(3,3);
    \draw[-{Stealth[length=3mm,width=2mm]}, Green, line width=1pt] (3,3)--(3,4);

   \draw[-{Stealth[length=3mm,width=2mm]}, red, line width=1pt] (-1,1)--(0,1);
    \draw[-{Stealth[length=3mm,width=2mm]}, red, line width=1pt] (0,1)--(1,1);
    \draw[-{Stealth[length=3mm,width=2mm]}, red, line width=1pt] (1,1)--(2,1);
    \draw[-{Stealth[length=3mm,width=2mm]}, red, line width=1pt] (2,1)--(2,2);
    \draw[-{Stealth[length=3mm,width=2mm]}, red, line width=1pt] (2,2)--(3,2);
    \draw[-{Stealth[length=3mm,width=2mm]}, red, line width=1pt] (3,2)--(3,3);
    \draw[-{Stealth[length=3mm,width=2mm]}, red, line width=1pt] (3,3)--(4,3);
    \draw[-{Stealth[length=3mm,width=2mm]}, red, line width=1pt] (4,3)--(5,3);
    \draw[-{Stealth[length=3mm,width=2mm]}, red, line width=1pt] (5,3)--(5,4);

    \draw[-{Stealth[length=3mm,width=2mm]}, blue, line width=1pt] (-1,0)--(0,0);
    \draw[-{Stealth[length=3mm,width=2mm]}, blue, line width=1pt] (0,0)--(1,0);
    \draw[-{Stealth[length=3mm,width=2mm]}, blue, line width=1pt] (1,0)--(2,0);
    \draw[-{Stealth[length=3mm,width=2mm]}, blue, line width=1pt] (2,0)--(3,0);
    \draw[-{Stealth[length=3mm,width=2mm]}, blue, line width=1pt] (3,0)--(3,1);
    \draw[-{Stealth[length=3mm,width=2mm]}, blue, line width=1pt] (3,1)--(4,1);
    \draw[-{Stealth[length=3mm,width=2mm]}, blue, line width=1pt] (4,1)--(4,2);
    \draw[-{Stealth[length=3mm,width=2mm]}, blue, line width=1pt] (4,2)--(5,2);
    \draw[-{Stealth[length=3mm,width=2mm]}, blue, line width=1pt] (5,2)--(5,3);
    \draw[-{Stealth[length=3mm,width=2mm]}, blue, line width=1pt] (5,3)--(6,3);
    \draw[-{Stealth[length=3mm,width=2mm]}, blue, line width=1pt] (6,3)--(6,4);

    \node at (-1.25,0) {$ z_1$};
    \node at (-1.25,1) {$ z_2$};
    \node at (-1.25, 2) {$z_3$};
    \node at (-1.25, 3) {$z_4$};

    \node[circle, draw=violet, fill=violet, opacity=0.4, ultra thick, minimum size=7pt, inner sep=2pt] at (0,0){};
    \node[circle, draw=violet, fill=violet, opacity=0.4, ultra thick, minimum size=7pt, inner sep=2pt] at (1,0){};
    \node[circle, draw=violet, fill=violet, opacity=0.4, ultra thick, minimum size=7pt, inner sep=2pt] at (2,0){};
    \node[circle, draw=violet, fill=violet, opacity=0.4, ultra thick, minimum size=7pt, inner sep=2pt] at (0,1){};
    \node[circle, draw=violet, fill=violet, opacity=0.4, ultra thick, minimum size=7pt, inner sep=2pt] at (1,1){};
    \node[circle, draw=violet, fill=violet, opacity=0.4, ultra thick, minimum size=7pt, inner sep=2pt] at (4,3){};
    \node[circle, draw=violet, fill=violet, opacity=0.4, ultra thick, minimum size=7pt, inner sep=2pt] at (0,2){};
    \node[circle, draw=violet, fill=violet, opacity=0.4, ultra thick, minimum size=7pt, inner sep=2pt] at (1,2){};
    
    \node[cross out, draw=violet, ultra thick, minimum size=10pt, inner sep=2pt] at (4,2){};
    \node[cross out, draw=violet, ultra thick, minimum size=10pt, inner sep=2pt] at (3,1){};
    \node[cross out, draw=violet, ultra thick, minimum size=10pt, inner sep=2pt] at (2,3){};

  \end{tikzpicture}
  \end{minipage}\]

 Let $\psi(S^\text{dec})=T$ be the set-valued tableau corresponding to the decorated state on the left, and let $\psi(\hat{S}^{\text{dec}})=\hat{T}$ be the semistandard tableau obtained after applying the uncrowding operation on $T$. 

    \[T= \ytableausetup{boxsize=2.2em}
      \ytableaushort{1 1 {1,2}, 2 3 4, 3 } \qquad \qquad \hat{T}=\ytableausetup{boxsize=2.2em}
      \ytableaushort{1 1 1, 2 2 4, 3 3}\]

Note that $T$ and $\hat{T}$ are equal to $T_1$ and $T_3$ from Example~\ref{ex:lattice-insertion}.
\end{example}

\begin{lemma}
    Step $4$ of the lattice uncrowding algorithm is well-defined, that is, there always exists a path $\gamma$ starting from the north edge of $V$ 
    and ending at either a vertex of type $\tt b_2$ or an exit arrow $E$. In the latter case, no exit arrow lies immediately to the right of $E$.
\end{lemma}

\begin{proof}
    First, recall that we enumerate columns in increasing order from left to right while rows are ordered from bottom to top. Let the exit arrow of $L_r$ lie in column $t$. Since $V$ is defined to be the closest vertex to the left of $B,$ the exit arrow $E$ reached by the path $\gamma$ must lie on the natural path $L_{r+1}.$ Therefore if there is an exit arrow directly to the right of $E$, it must be the exit arrow corresponding to $L_r$. We claim that if $\gamma$ never reaches a vertex of type $\tt b_2,$ then $E$ and the exit arrow for $L_r$ in column $t$ 
    cannot lie in consecutive columns.
   
   Suppose $L_r$ has a bump $B$ in row $j$ column $q$, and let $V$ be the vertex located in row $j$ column $p$. By construction, 
   $q-p \geqslant 1$. Note that $L_r$ must take $(n-j+1)$ vertical steps to reach its exit arrow (including the exit arrow itself), and the admissibility conditions 
   require that each vertical step be preceded by at least one horizontal step to the right. Then, the number of horizontal steps along $L_r$ between $B$ and the exit arrow of $L_r$ must be at least $(n-j+1)$ (including the east edge of $B$).  Therefore, the exit arrow of $L_r$ must lie in a column that is weakly right of column $q+n-j+1,$ that is, $t \geqslant q+n-j+1$.
   
   Now assume that $\gamma$ never encounters a vertex of type $\tt b_2.$ Then, $\gamma$ must have exactly $(n-j)$ right arrows before reaching its exit arrow. Hence, $E$ 
   must lie in column $p +(n-j).$ Examining the difference between column $t$ and $p+n-j,$ we see that  
    $$t - (p+n-j) \geqslant (q+n-j+1) -(p+n-j) = q-p + 1 > 1.$$ That is, $E$ and the exit arrow for $L_r$ cannot lie in consecutive columns.
\end{proof}

\begin{theorem}
\label{thm:lattice_uncrowding_algorithm} 
Let $S^\text{dec} \in \decoratedstate_\lambda$ and let $T=\psi(S^\text{dec})\in \SVT^n(\lambda)$ be the corresponding set-valued tableau. Then the lattice uncrowding 
algorithm on $S^\text{dec}$ realizes Buch's uncrowding map on $T$ under the bijection $\psi.$ In particular, the following diagram
\[
\begin{tikzcd}[column sep=large,row sep=large]
\decoratedstate_\lambda
\arrow[r,"\U_\LAT"]
\arrow[d,"\psi"']
&
\displaystyle{\bigsqcup_{\mu\supseteq\lambda} \decoratedstate^0_\mu}
\arrow[d,"\psi"]
\\
\SVT^n(\lambda)
\arrow[r,"P_\SVT"']
&
\displaystyle{\bigsqcup_{\mu\supseteq\lambda} \SSYT^n(\mu)}
\end{tikzcd}
\]
commutes. Equivalently, \[\psi(\U_\LAT(S^\text{dec}))=P_\SVT (\psi(S^\text{dec})).\]
\end{theorem}

\begin{proof}
    Let $S^\text{dec} \in \decoratedstate_\lambda$ and let $T=\psi(S^\text{dec})$ be the corresponding set-valued tableau. Define the slide sequences as in Step 1 of 
    the lattice uncrowding algorithm. We will prove the statement by showing that the lattice uncrowding algorithm is equivalent to the uncrowding 
    algorithm on set-valued tableaux. For $1\leqslant i \leqslant n$, $\Omega_i$ is indexed by the natural path $L_i$, which by Lemma~\ref{lem:paths-to-SVTtab} 
    corresponds to row $i$ of $T$. Recall that uncrowding on set-valued tableaux finds the bottommost row containing a multicell and uncrowds 
    the largest crowded entry in that row. Since the bumps in each slide sequence are ordered from top to bottom, the slide sequences are processed 
    in the same order as crowded entries under the uncrowding operation on set-valued tableaux.
    Therefore, it suffices to show that one iteration of Steps 3--5 realizes one step of Buch's uncrowding map. This is because after one non-trivial 
    bump is removed, the algorithm repeats for any newly created bump, so the same local argument applies recursively to the remaining crowded entries.
    
    Suppose that $T$ has exactly one excess entry, say $j$, lying in row $r$ of $T$. Then the natural path $L_r$ will contain exactly one $\tt a_2$ 
    vertex in lattice row $j$ corresponding to the non-trivial bump $B$ in the slide sequence $\Omega_{r}$. Let $V$ be the rightmost vertex of type $\tt c_1$ 
    on $L_{r+1}$ as defined in Step $3.$ Now recall that uncrowding $T$ is equivalent to removing $j$ from its original cell and row inserting $j$ into 
    row $r+1.$ The next $\tt b_2$ vertex after $V$ along $L_{r+1}$ (if it exists) corresponds to the smallest entry greater than $j$ in row $r+1$ of $T.$  
    Then by Proposition~\ref{prop:lattice_INSERTION_algorithm}, performing the lattice insertion algorithm during Step 4 for $V_{r+1} \coloneq V$ is equivalent to row inserting $j$ into row $r+1.$
 
     During Step 2 of the lattice insertion algorithm, if the bump $B_{k-1}\coloneq B$ is not eliminated, then it is declared trivial. So, the non-trivial contribution from original excess entry $j$ is eliminated. This results in a decorated state corresponding exactly to the tableau obtained after one application of the uncrowding algorithm.  
    
    Repeating this procedure over all slide sequences, and continuing until no new bumps are created produces exactly the same sequence of insertions as the uncrowding map $\U_\SVT(T)$. Hence, under the bijection \eqref{eq:bij-decoratedstate-SVT}, the state $\U_\LAT(S^\text{dec})$ corresponds to $P_\SVT(T)$. Therefore, $\psi(\U_\LAT(S^\text{dec}))=P_\SVT(T).$ 
\end{proof}

Theorem~\ref{thm:lattice_uncrowding_algorithm} shows that the lattice uncrowding algorithm recovers the semistandard tableau $P_\SVT(T)$. Next we describe how the same process records the flagged increasing tableau $F_\SVT(T).$

\begin{proposition} \label{prop:lattice_uncrowding_recording}
    Let $T\in \SVT^n(\lambda)$. Then the flagged increasing tableau $F_\SVT(T)$ can be reconstructed from the lattice uncrowding algorithm using the sequence of exit-arrow moves together with the natural path on which each uncrowding step begins.
\end{proposition}

\begin{proof}
    Let $F_i$ be defined as in \eqref{eq:uncrowding_SVT_to_pair}. Let $S_0^{\text{dec}},S_1^{\text{dec}}, S_2^{\text{dec}}, \ldots , S_{\mathsf{ex}(T)}^{\text{dec}}$ be the sequence of decorated states 
    obtained during the uncrowding procedure where $S_0^{\text{dec}}$ is the decorated state corresponding to $T$, and $S_i^{\text{dec}}$ is the decorated state after the $i$-th exit arrow has 
    been moved. For $0\leqslant i \leqslant \mathsf{ex}(T)$, let $U_i$ denote the tableau corresponding to $S_i^{\text{dec}}$. By Theorem~\ref{thm:lattice_uncrowding_algorithm}, each step from $S_{i}^{\text{dec}}$ to $S_{i+1}^{\text{dec}}$ is exactly one application of the uncrowding operation, so $U_{i+1}$ is obtained from $U_{i}$ by adding a single box $C$ in the location $U_{i+1}/U_{i}$.

    Consider the step from $S_{i}^{\text{dec}}$ to $S_{i+1}^{\text{dec}}.$ Suppose the initial non-trivial bump considered during this step lies on the natural path $L_r$, and suppose the shifted exit arrow corresponds to the natural path $L_s$. Then the corresponding crowded entry in $U_i$ starts in row $r$ and after applying the uncrowding operation creates a new box $C$ in row $s.$ Hence, the entry added to $C$ in $F_{i+1}$ is $s-r$. Therefore, each uncrowding step determines one entry of $F_\SVT(T)$ from the pair $(r,s)$ which consists of the initial bump path $L_r$ and the shifted exit arrow for $L_s$.
\end{proof}

Theorem~\ref{thm:lattice_uncrowding_algorithm} and Proposition~\ref{prop:lattice_uncrowding_recording} together show that the lattice uncrowding algorithm recovers the full image of Buch's uncrowding map at the level of decorated states in the lattice model.

In order to obtain the Schur expansion using lattice models, one would need to uncrowd every decorated state in $\decoratedstate_\lambda.$ Theorem~\ref{thm:yamanouchi-expand} allows us to shorten this process by examining only the decorated states that encode Yamanouchi set-valued tableaux. 

\begin{definition}\label{def:lattice-reading-word}
    Let $S^\text{dec} \in \decoratedstate_\lambda$ be a decorated state with $n$ rows. Let $L_i$ be the $i$-th natural path of $S^\text{dec}.$ Define the reading word of $L_i, \rd(L_i)$, as follows. First group consecutive contributing bumps with the first $\tt b_2$ vertex to the left of these bumps. All other $\tt b_2$ vertices each form individual groups. Reading these groups left to right, list the weights of each group in decreasing order. The \defn{vertex reading word} of $S^\text{dec}$ is then defined as 
    \[\rd(S^\text{dec}) = \rd(L_n)\rd(L_{n-1}) \dotsm \rd(L_1).\]
    In correspondence with the tableau, we may only read the subscripts of the weights. 
\end{definition}

\begin{lemma}
    For $S^\text{dec} \in \decoratedstate_\lambda$, the vertex reading word of $S^\text{dec}$ is precisely the reading word of the corresponding set-valued tableau $\psi(S^\text{dec})$.
\end{lemma}
This follows immediately from Lemma~\ref{lem:paths-to-SVTtab}. Thus, a decorated state is Yamanouchi exactly when its vertex reading word is Yamanouchi.

\begin{example}
Consider the following decorated state with boundary condition $\lambda = (2,2,1,0)$ and every bump is non-trivial.
  \[
    \begin{tikzpicture}

    \draw (0,-0.75)--(0,4);
    \draw (1,-0.75)--(1,4);
    \draw (2,-0.75)--(2,4);
    \draw (3,-0.75)--(3,4);
    \draw (4, -0.75) -- (4,4);
    \draw (5, -0.75) -- (5,4);

    \draw (-1,0)--(5.5,0);
    \draw (-1,1)--(5.5,1);
    \draw (-1, 2) -- (5.5, 2);
    \draw (-1, 3) -- (5.5, 3);

    \draw[-{Stealth[length=3mm,width=2mm]}, Orange, line width=1pt] (-1,3)--(0,3);
    \draw[-{Stealth[length=3mm,width=2mm]}, Orange, line width=1pt] (0,3)--(0,4);

    \draw[-{Stealth[length=3mm,width=2mm]}, Green, line width=1pt] (-1,2)--(0,2);
    \draw[-{Stealth[length=3mm,width=2mm]}, Green, line width=1pt] (0,2)--(0,3);
    \draw[-{Stealth[length=3mm,width=2mm]}, Green, line width=1pt] (0,3)--(1,3);
    \draw[-{Stealth[length=3mm,width=2mm]}, Green, line width=1pt] (1,3)--(2,3);
     \draw[-{Stealth[length=3mm,width=2mm]}, Green, line width=1pt] (2,3)--(2,4);
   
     \draw[-{Stealth[length=3mm,width=2mm]},red, line width=1.pt] (-1,1)--(0,1);
     \draw[-{Stealth[length=3mm,width=2mm]},red, line width=1pt] (0,1)--(1,1);
     \draw[-{Stealth[length=3mm,width=2mm]},red, line width = 1pt] (1,1)--(1,2);
    \draw[-{Stealth[length=3mm,width=2mm]},red, line width=1pt] (1,2)--(2,2);
    \draw[-{Stealth[length=3mm,width=2mm]},red, line width=1pt] (2,2)--(2,3);
    \draw[-{Stealth[length=3mm,width=2mm]},red, line width=1pt] (2,3)--(3,3);
      \draw[-{Stealth[length=3mm,width=2mm]},red, line width=1pt] (3,3)--(4,3);
      \draw[-{Stealth[length=3mm,width=2mm]},red, line width=1pt] (4,3)--(4,4);

    \draw[-{Stealth[length=3mm,width=2mm]},blue,line width=1pt] (-1,0)--(0,0);
     \draw[-{Stealth[length=3mm,width=2mm]},blue, line width=1pt] (0,0)--(1,0);
    \draw[-{Stealth[length=3mm,width=2mm]},blue, line width=1pt] (1,0)--(2,0);
     \draw[-{Stealth[length=3mm,width=2mm]},blue, line width=1pt] (2,0)--(2,1);
     \draw[-{Stealth[length=3mm,width=2mm]},blue, line width=1pt] (2,1)--(3,1);
      \draw[-{Stealth[length=3mm,width=2mm]},blue, line width=1pt] (3,1)--(3,2);
     \draw[-{Stealth[length=3mm,width=2mm]},blue, line width=1pt] (3,2)--(4,2);
     \draw[-{Stealth[length=3mm,width=2mm]},blue, line width=1pt] (4,2)--(4,3);
     \draw[-{Stealth[length=3mm,width=2mm]},blue, line width=1pt] (4,3)--(5,3);
     \draw[-{Stealth[length=3mm,width=2mm]},blue, line width=1pt] (5,3)--(5,4);

    \node at (-1.25,0) {$ z_1$};
    \node at (-1.25,1) {$ z_2$};
    \node at (-1.25, 2) {$z_3$};
    \node at (-1.25, 3) {$z_4$};

    \node[circle, draw = violet, ultra thick, minimum size=10pt, inner sep=2pt] (c) at (2,1){};
    \node[circle, draw = violet, ultra thick, minimum size=10pt, inner sep=2pt] (c) at (1,2){};
    \node[circle, draw = violet, ultra thick, minimum size=10pt, inner sep=2pt] (c) at (3,2){};
    \node[circle, draw = violet, fill=violet,opacity = 0.4, ultra thick, minimum size=7pt, inner sep=2pt] (c) at (0,0){};    
    \node[circle, draw = violet, fill=violet, opacity = 0.4,ultra thick, minimum size=7pt, inner sep=2pt] (c) at (1,0){};
    \node[circle, draw = violet, fill=violet, opacity = 0.4,ultra thick, minimum size=7pt, inner sep=2pt] (c) at (0,1){};
    \node[circle, draw = violet, fill=violet, opacity = 0.4,ultra thick, minimum size=7pt, inner sep=2pt] (c) at (1,3){};
    \node[circle, draw = violet, fill=violet, opacity = 0.4,ultra thick, minimum size=7pt, inner sep=2pt] (c) at (3,3){};
    
    \end{tikzpicture}\]

    The reading word for this state is $\textcolor{Green}{4}\underline{\textcolor{red}{32}}\textcolor{red}{4}\textcolor{blue}{1}\underline{\textcolor{blue}{321}}$ where groupings with contributing bumps are underlined.
\end{example}

By uncrowding the decorated states that yield a Yamanouchi vertex reading word, we can obtain the Schur expansion for symmetric Grothendieck polynomials using only the lattice model.

\begin{example}
    Consider the partition $\lambda = (2,2,0)$ with $n=3$. There are $6$ admissible states in $\states_\lambda^3$ yielding 13 decorated states in $\decoratedstate_\lambda$. By restricting to decorated states with a Yamanouchi vertex reading word, we only need to uncrowd the following four decorated states. 
    \[
    \scalebox{0.9}{\begin{minipage}{0.5\textwidth}
{\begin{tikzpicture}

    \draw (0,-0.75)--(0,3);
    \draw (1,-0.75)--(1,3);
    \draw (2,-0.75)--(2,3);
    \draw (3,-0.75)--(3,3);
    \draw (4, -0.75) -- (4,3);

    \draw (-1,0)--(4.5,0);
    \draw (-1,1)--(4.5,1);
    \draw (-1, 2) -- (4.5, 2);

    \draw[-{Stealth[length=3mm,width=2mm]}, Green, line width=1pt] (-1,2)--(0,2);
    \draw[-{Stealth[length=3mm,width=2mm]}, Green, line width=1pt] (0,2)--(0,3);

     \draw[-{Stealth[length=3mm,width=2mm]},red, line width=1.pt] (-1,1)--(0,1);
     \draw[-{Stealth[length=3mm,width=2mm]},red, line width=1pt] (0,1)--(1,1);
     \draw[-{Stealth[length=3mm,width=2mm]},red, line width = 1pt] (1,1)--(2,1);
    \draw[-{Stealth[length=3mm,width=2mm]},red, line width=1pt] (2,1)--(2,2);
    \draw[-{Stealth[length=3mm,width=2mm]},red, line width=1pt] (2,2)--(3,2);
    \draw[-{Stealth[length=3mm,width=2mm]},red, line width=1pt] (3,2)--(3,3);

    \draw[-{Stealth[length=3mm,width=2mm]},blue,line width=1pt] (-1,0)--(0,0);
     \draw[-{Stealth[length=3mm,width=2mm]},blue, line width=1pt] (0,0)--(1,0);
    \draw[-{Stealth[length=3mm,width=2mm]},blue, line width=1pt] (1,0)--(2,0);
     \draw[-{Stealth[length=3mm,width=2mm]},blue, line width=1pt] (2,0)--(2,1);
     \draw[-{Stealth[length=3mm,width=2mm]},blue, line width=1pt] (2,1)--(3,1);
      \draw[-{Stealth[length=3mm,width=2mm]},blue, line width=1pt] (3,1)--(3,2);
     \draw[-{Stealth[length=3mm,width=2mm]},blue, line width=1pt] (3,2)--(4,2);
     \draw[-{Stealth[length=3mm,width=2mm]},blue, line width=1pt] (4,2)--(4,3);

    \node at (-1.25,0) {$ z_1$};
    \node at (-1.25,1) {$ z_2$};
    \node at (-1.25, 2) {$z_3$};

    \draw [-Stealth](5.5,1) -- (7,1); 
    \node ($U$) at (6.3,1.25) {$\psi$};
    
    \node[circle, draw = violet, fill=violet, opacity = 0.4,ultra thick, minimum size=7pt, inner sep=2pt] (c) at (0,0){};
    \node[circle, draw = violet, fill=violet, opacity = 0.4,ultra thick, minimum size=7pt, inner sep=2pt] (c) at (1,0){};
    \node[circle, draw = violet, fill=violet, opacity = 0.4,ultra thick, minimum size=7pt, inner sep=2pt] (c) at (0,1){};
    \node[circle, draw = violet, fill=violet,opacity = 0.4, ultra thick, minimum size=7pt, inner sep=2pt] (c) at (1,1){};
        \node[cross out, draw = violet,  ultra thick, minimum size=10pt, inner sep=2pt] (c) at (2,2){};
  \end{tikzpicture}}
  \end{minipage} \quad \quad 
  \begin{minipage}{0.2\textwidth}
      $\ytableaushort{1 {1},2 2}$
    \end{minipage}}
    \]

\[
\scalebox{0.9}{\begin{minipage}{0.45\textwidth}
{\begin{tikzpicture}

    \draw (0,-0.75)--(0,3);
    \draw (1,-0.75)--(1,3);
    \draw (2,-0.75)--(2,3);
    \draw (3,-0.75)--(3,3);
    \draw (4, -0.75) -- (4,3);

    \draw (-1,0)--(4.5,0);
    \draw (-1,1)--(4.5,1);
    \draw (-1, 2) -- (4.5, 2);

    \draw[-{Stealth[length=3mm,width=2mm]}, Green, line width=1pt] (-1,2)--(0,2);
    \draw[-{Stealth[length=3mm,width=2mm]}, Green, line width=1pt] (0,2)--(0,3);

     \draw[-{Stealth[length=3mm,width=2mm]},red, line width=1.pt] (-1,1)--(0,1);
     \draw[-{Stealth[length=3mm,width=2mm]},red, line width=1pt] (0,1)--(1,1);
     \draw[-{Stealth[length=3mm,width=2mm]},red, line width = 1pt] (1,1)--(2,1);
    \draw[-{Stealth[length=3mm,width=2mm]},red, line width=1pt] (2,1)--(2,2);
    \draw[-{Stealth[length=3mm,width=2mm]},red, line width=1pt] (2,2)--(3,2);
    \draw[-{Stealth[length=3mm,width=2mm]},red, line width=1pt] (3,2)--(3,3);

    \draw[-{Stealth[length=3mm,width=2mm]},blue,line width=1pt] (-1,0)--(0,0);
     \draw[-{Stealth[length=3mm,width=2mm]},blue, line width=1pt] (0,0)--(1,0);
    \draw[-{Stealth[length=3mm,width=2mm]},blue, line width=1pt] (1,0)--(2,0);
     \draw[-{Stealth[length=3mm,width=2mm]},blue, line width=1pt] (2,0)--(2,1);
     \draw[-{Stealth[length=3mm,width=2mm]},blue, line width=1pt] (2,1)--(3,1);
      \draw[-{Stealth[length=3mm,width=2mm]},blue, line width=1pt] (3,1)--(3,2);
     \draw[-{Stealth[length=3mm,width=2mm]},blue, line width=1pt] (3,2)--(4,2);
     \draw[-{Stealth[length=3mm,width=2mm]},blue, line width=1pt] (4,2)--(4,3);

    \node at (-1.25,0) {$ z_1$};
    \node at (-1.25,1) {$ z_2$};
    \node at (-1.25, 2) {$z_3$};   

    \node[circle, draw = violet, ultra thick, minimum size=10pt, inner sep=2pt] (c) at (2,2){};
    \node[circle, draw = violet, fill=violet, opacity = 0.4,ultra thick, minimum size=7pt, inner sep=2pt] (c) at (0,0){};
    \node[circle, draw = violet, fill=violet,opacity = 0.4, ultra thick, minimum size=7pt, inner sep=2pt] (c) at (1,0){};
    \node[circle, draw = violet, fill=violet,opacity = 0.4, ultra thick, minimum size=7pt, inner sep=2pt] (c) at (0,1){};
    \node[circle, draw = violet, fill=violet, opacity = 0.4,ultra thick, minimum size=7pt, inner sep=2pt] (c) at (1,1){};

    \draw[-Stealth](5,1) -- (6,1);
    \node ($U$) at (5.5,1.25) {$\U_\LAT$};

  \end{tikzpicture}}
  \end{minipage} \;
    \begin{minipage}{0.4\textwidth}
        {\begin{tikzpicture}

    \draw (0,-0.75)--(0,3);
    \draw (1,-0.75)--(1,3);
    \draw (2,-0.75)--(2,3);
    \draw (3,-0.75)--(3,3);
    \draw (4, -0.75) -- (4,3);

    \draw (-1,0)--(4.5,0);
    \draw (-1,1)--(4.5,1);
    \draw (-1, 2) -- (4.5, 2);

  
    \draw[-{Stealth[length=3mm,width=2mm]}, Green, line width=1pt] (-1,2)--(0,2);
    \draw[-{Stealth[length=3mm,width=2mm]}, Green, line width=1pt] (0,2)--(1,2);
    \draw[-{Stealth[length=3mm,width=2mm]}, Green, line width=1pt] (1,2)--(1,3);
    
     \draw[-{Stealth[length=3mm,width=2mm]},red, line width=1.pt] (-1,1)--(0,1);
     \draw[-{Stealth[length=3mm,width=2mm]},red, line width=1pt] (0,1)--(1,1);
     \draw[-{Stealth[length=3mm,width=2mm]},red, line width = 1pt] (1,1)--(2,1);
    \draw[-{Stealth[length=3mm,width=2mm]},red, line width=1pt] (2,1)--(2,2);
    \draw[-{Stealth[length=3mm,width=2mm]},red, line width=1pt] (2,2)--(3,2);
    \draw[-{Stealth[length=3mm,width=2mm]},red, line width=1pt] (3,2)--(3,3);

    \draw[-{Stealth[length=3mm,width=2mm]},blue,line width=1pt] (-1,0)--(0,0);
     \draw[-{Stealth[length=3mm,width=2mm]},blue, line width=1pt] (0,0)--(1,0);
    \draw[-{Stealth[length=3mm,width=2mm]},blue, line width=1pt] (1,0)--(2,0);
     \draw[-{Stealth[length=3mm,width=2mm]},blue, line width=1pt] (2,0)--(2,1);
     \draw[-{Stealth[length=3mm,width=2mm]},blue, line width=1pt] (2,1)--(3,1);
      \draw[-{Stealth[length=3mm,width=2mm]},blue, line width=1pt] (3,1)--(3,2);
     \draw[-{Stealth[length=3mm,width=2mm]},blue, line width=1pt] (3,2)--(4,2);
     \draw[-{Stealth[length=3mm,width=2mm]},blue, line width=1pt] (4,2)--(4,3);

    \node at (-1.25,0) {$ z_1$};
    \node at (-1.25,1) {$ z_2$};
    \node at (-1.25, 2) {$z_3$};

    \node at (5.5,1.25) {$\psi$};
    \draw [-Stealth](5,1) -- (6,1);   
    \node[circle, draw = violet, fill=violet, opacity = 0.4,ultra thick, minimum size=7pt, inner sep=2pt] (c) at (0,0){};
    \node[circle, draw = violet, fill=violet,opacity = 0.4, ultra thick, minimum size=7pt, inner sep=2pt] (c) at (1,0){};
    \node[circle, draw = violet, fill=violet, opacity = 0.4,ultra thick, minimum size=7pt, inner sep=2pt] (c) at (0,1){};
    \node[circle, draw = violet, fill=violet, opacity = 0.4,ultra thick, minimum size=7pt, inner sep=2pt] (c) at (1,1){};
        \node[circle, draw = violet, fill=violet,opacity = 0.4, ultra thick, minimum size=7pt, inner sep=2pt] (c) at (0,2){};
    \node[cross out, draw = violet, ultra thick, minimum size=10pt, inner sep=2pt] (c) at (2,2){};
  \end{tikzpicture}}
    \end{minipage} \quad \quad \quad \quad
  \begin{minipage}{0.25\textwidth}
      $\ytableaushort{1 {1},2 2, 3}$
  \end{minipage}} \]
  
\[
\scalebox{0.9}{\begin{minipage}{0.45\textwidth}
{\begin{tikzpicture}

    \draw (0,-0.75)--(0,3);
    \draw (1,-0.75)--(1,3);
    \draw (2,-0.75)--(2,3);
    \draw (3,-0.75)--(3,3);
    \draw (4, -0.75) -- (4,3);
 
    \draw (-1,0)--(4.5,0);
    \draw (-1,1)--(4.5,1);
    \draw (-1, 2) -- (4.5, 2);


    \draw[-{Stealth[length=3mm,width=2mm]}, Green, line width=1pt] (-1,2)--(0,2);
    \draw[-{Stealth[length=3mm,width=2mm]}, Green, line width=1pt] (0,2)--(0,3);
  
     \draw[-{Stealth[length=3mm,width=2mm]},red, line width=1.pt] (-1,1)--(0,1);
     \draw[-{Stealth[length=3mm,width=2mm]},red, line width=1pt] (0,1)--(1,1);
     \draw[-{Stealth[length=3mm,width=2mm]},red, line width = 1pt] (1,1)--(1,2);
    \draw[-{Stealth[length=3mm,width=2mm]},red, line width=1pt] (1,2)--(2,2);
    \draw[-{Stealth[length=3mm,width=2mm]},red, line width=1pt] (2,2)--(3,2);
    \draw[-{Stealth[length=3mm,width=2mm]},red, line width=1pt] (3,2)--(3,3);

    \draw[-{Stealth[length=3mm,width=2mm]},blue,line width=1pt] (-1,0)--(0,0);
     \draw[-{Stealth[length=3mm,width=2mm]},blue, line width=1pt] (0,0)--(1,0);
    \draw[-{Stealth[length=3mm,width=2mm]},blue, line width=1pt] (1,0)--(2,0);
     \draw[-{Stealth[length=3mm,width=2mm]},blue, line width=1pt] (2,0)--(2,1);
     \draw[-{Stealth[length=3mm,width=2mm]},blue, line width=1pt] (2,1)--(3,1);
      \draw[-{Stealth[length=3mm,width=2mm]},blue, line width=1pt] (3,1)--(3,2);
     \draw[-{Stealth[length=3mm,width=2mm]},blue, line width=1pt] (3,2)--(4,2);
     \draw[-{Stealth[length=3mm,width=2mm]},blue, line width=1pt] (4,2)--(4,3);

    \node at (-1.25,0) {$ z_1$};
    \node at (-1.25,1) {$ z_2$};
    \node at (-1.25, 2) {$z_3$};

    \node[circle, draw = violet, ultra thick, minimum size=10pt] (c) at (2,1){};

    \draw[-Stealth](5,1) -- (6,1);
    \node ($U$) at (5.5,1.25) {$\U_\LAT$}; 

    \node[circle, draw = violet, fill=violet, opacity = 0.4, ultra thick, minimum size=7pt, inner sep=2pt] (c) at (0,0){};
    \node[circle, draw = violet, fill=violet, opacity = 0.4, ultra thick, minimum size=7pt, inner sep=2pt] (c) at (1,0){};
    \node[circle, draw = violet, fill=violet, opacity = 0.4, ultra thick, minimum size=7pt, inner sep=2pt] (c) at (0,1){};
    \node[circle, draw = violet, fill=violet,opacity = 0.4, ultra thick, minimum size=7pt, inner sep=2pt] (c) at (2,2){};
    
     \node[cross out, draw = violet, ultra thick, minimum size=10pt, inner sep=2pt] (c) at (1,2){};
   
  \end{tikzpicture}}
  \end{minipage} \;
    \begin{minipage}{0.4\textwidth}
        {\begin{tikzpicture}

    \draw (0,-0.75)--(0,3);
    \draw (1,-0.75)--(1,3);
    \draw (2,-0.75)--(2,3);
    \draw (3,-0.75)--(3,3);
    \draw (4, -0.75) -- (4,3);

    \draw (-1,0)--(4.5,0);
    \draw (-1,1)--(4.5,1);
    \draw (-1, 2) -- (4.5, 2);
   
    \draw[-{Stealth[length=3mm,width=2mm]}, Green, line width=1pt] (-1,2)--(0,2);
    \draw[-{Stealth[length=3mm,width=2mm]}, Green, line width=1pt] (0,2)--(1,2);
    \draw[-{Stealth[length=3mm,width=2mm]}, Green, line width=1pt] (1,2)--(1,3);

     \draw[-{Stealth[length=3mm,width=2mm]},red, line width=1.pt] (-1,1)--(0,1);
     \draw[-{Stealth[length=3mm,width=2mm]},red, line width=1pt] (0,1)--(1,1);
     \draw[-{Stealth[length=3mm,width=2mm]},red, line width = 1pt] (1,1)--(2,1);
    \draw[-{Stealth[length=3mm,width=2mm]},red, line width=1pt] (2,1)--(2,2);
    \draw[-{Stealth[length=3mm,width=2mm]},red, line width=1pt] (2,2)--(3,2);
    \draw[-{Stealth[length=3mm,width=2mm]},red, line width=1pt] (3,2)--(3,3);
   
    \draw[-{Stealth[length=3mm,width=2mm]},blue,line width=1pt] (-1,0)--(0,0);
     \draw[-{Stealth[length=3mm,width=2mm]},blue, line width=1pt] (0,0)--(1,0);
    \draw[-{Stealth[length=3mm,width=2mm]},blue, line width=1pt] (1,0)--(2,0);
     \draw[-{Stealth[length=3mm,width=2mm]},blue, line width=1pt] (2,0)--(2,1);
     \draw[-{Stealth[length=3mm,width=2mm]},blue, line width=1pt] (2,1)--(3,1);
      \draw[-{Stealth[length=3mm,width=2mm]},blue, line width=1pt] (3,1)--(3,2);
     \draw[-{Stealth[length=3mm,width=2mm]},blue, line width=1pt] (3,2)--(4,2);
     \draw[-{Stealth[length=3mm,width=2mm]},blue, line width=1pt] (4,2)--(4,3);
   
    \node at (-1.25,0) {$ z_1$};
    \node at (-1.25,1) {$ z_2$};
    \node at (-1.25, 2) {$z_3$};

    \node[circle, draw = violet, fill=violet, opacity = 0.4,ultra thick, minimum size=7pt, inner sep=2pt] (c) at (0,0){};
    \node[circle, draw = violet, fill=violet, opacity = 0.4,ultra thick, minimum size=7pt, inner sep=2pt] (c) at (1,0){};
    \node[circle, draw = violet, fill=violet, opacity = 0.4,ultra thick, minimum size=7pt, inner sep=2pt] (c) at (0,1){};
    \node[circle, draw = violet, fill=violet, opacity = 0.4,ultra thick, minimum size=7pt, inner sep=2pt] (c) at (1,1){};
    \node[circle, draw = violet, fill=violet, opacity = 0.4,ultra thick, minimum size=7pt, inner sep=2pt] (c) at (0,2){};
        \node[cross out, draw = violet, ultra thick, minimum size=10pt, inner sep=2pt] (c) at (2,2){};
    
    \draw [-Stealth](5,1) -- (6,1);   
    \node at (5.5,1.25) {$\psi$};
  \end{tikzpicture}}
    \end{minipage} \quad \quad \quad \quad
  \begin{minipage}{0.25\textwidth}
      $\ytableaushort{1 {1},2 2, 3}$
  \end{minipage}} \]

\[
\scalebox{0.9}{\begin{minipage}{0.45\textwidth}
{\begin{tikzpicture}

    \draw (0,-0.75)--(0,3);
    \draw (1,-0.75)--(1,3);
    \draw (2,-0.75)--(2,3);
    \draw (3,-0.75)--(3,3);
    \draw (4, -0.75) -- (4,3);

    \draw (-1,0)--(4.5,0);
    \draw (-1,1)--(4.5,1);
    \draw (-1, 2) -- (4.5, 2);
   
    \draw[-{Stealth[length=3mm,width=2mm]}, Green, line width=1pt] (-1,2)--(0,2);
    \draw[-{Stealth[length=3mm,width=2mm]}, Green, line width=1pt] (0,2)--(0,3);

     \draw[-{Stealth[length=3mm,width=2mm]},red, line width=1.pt] (-1,1)--(0,1);
     \draw[-{Stealth[length=3mm,width=2mm]},red, line width=1pt] (0,1)--(1,1);
     \draw[-{Stealth[length=3mm,width=2mm]},red, line width = 1pt] (1,1)--(1,2);
    \draw[-{Stealth[length=3mm,width=2mm]},red, line width=1pt] (1,2)--(2,2);
    \draw[-{Stealth[length=3mm,width=2mm]},red, line width=1pt] (2,2)--(3,2);
    \draw[-{Stealth[length=3mm,width=2mm]},red, line width=1pt] (3,2)--(3,3);

    \draw[-{Stealth[length=3mm,width=2mm]},blue,line width=1pt] (-1,0)--(0,0);
     \draw[-{Stealth[length=3mm,width=2mm]},blue, line width=1pt] (0,0)--(1,0);
    \draw[-{Stealth[length=3mm,width=2mm]},blue, line width=1pt] (1,0)--(2,0);
     \draw[-{Stealth[length=3mm,width=2mm]},blue, line width=1pt] (2,0)--(2,1);
     \draw[-{Stealth[length=3mm,width=2mm]},blue, line width=1pt] (2,1)--(3,1);
      \draw[-{Stealth[length=3mm,width=2mm]},blue, line width=1pt] (3,1)--(3,2);
     \draw[-{Stealth[length=3mm,width=2mm]},blue, line width=1pt] (3,2)--(4,2);
     \draw[-{Stealth[length=3mm,width=2mm]},blue, line width=1pt] (4,2)--(4,3);

    \node at (-1.25,0) {$ z_1$};
    \node at (-1.25,1) {$ z_2$};
    \node at (-1.25, 2) {$z_3$};
  
    \node[circle, draw = violet, ultra thick, minimum size=10pt] (c) at (2,1){};
    \node[circle, draw = violet, ultra thick, minimum size=10pt] (c) at (1,2){};
    \draw[-Stealth](5,1) -- (6,1);
    \node ($U$) at (5.5,1.25) {$\U_\LAT$};
    \node[circle, draw = violet, fill=violet, opacity = 0.4,ultra thick, minimum size=7pt, inner sep=2pt] (c) at (0,0){};
    \node[circle, draw = violet, fill=violet, opacity = 0.4,ultra thick, minimum size=7pt, inner sep=2pt] (c) at (1,0){};
    \node[circle, draw = violet, fill=violet,opacity = 0.4, ultra thick, minimum size=7pt, inner sep=2pt] (c) at (0,1){};
    \node[circle, draw = violet, fill=violet, opacity = 0.4,ultra thick, minimum size=7pt, inner sep=2pt] (c) at (2,2){};
    
  \end{tikzpicture}}
  \end{minipage} \;
    \begin{minipage}{0.4\textwidth}
        {\begin{tikzpicture}

    \draw (0,-0.75)--(0,3);
    \draw (1,-0.75)--(1,3);
    \draw (2,-0.75)--(2,3);
    \draw (3,-0.75)--(3,3);
    \draw (4, -0.75) -- (4,3);

    \draw (-1,0)--(4.5,0);
    \draw (-1,1)--(4.5,1);
    \draw (-1, 2) -- (4.5, 2);

    \draw[-{Stealth[length=3mm,width=2mm]}, Green, line width=1pt] (-1,2)--(0,2);
    \draw[-{Stealth[length=3mm,width=2mm]}, Green, line width=1pt] (0,2)--(1,2);
    \draw[-{Stealth[length=3mm,width=2mm]}, Green, line width=1pt] (1,2)--(2,2);
    \draw[-{Stealth[length=3mm,width=2mm]}, Green, line width=1pt] (2,2)--(2,3);

     \draw[-{Stealth[length=3mm,width=2mm]},red, line width=1.pt] (-1,1)--(0,1);
     \draw[-{Stealth[length=3mm,width=2mm]},red, line width=1pt] (0,1)--(1,1);
     \draw[-{Stealth[length=3mm,width=2mm]},red, line width = 1pt] (1,1)--(2,1);
    \draw[-{Stealth[length=3mm,width=2mm]},red, line width=1pt] (2,1)--(2,2);
    \draw[-{Stealth[length=3mm,width=2mm]},red, line width=1pt] (2,2)--(3,2);
    \draw[-{Stealth[length=3mm,width=2mm]},red, line width=1pt] (3,2)--(3,3);

    \draw[-{Stealth[length=3mm,width=2mm]},blue,line width=1pt] (-1,0)--(0,0);
     \draw[-{Stealth[length=3mm,width=2mm]},blue, line width=1pt] (0,0)--(1,0);
    \draw[-{Stealth[length=3mm,width=2mm]},blue, line width=1pt] (1,0)--(2,0);
     \draw[-{Stealth[length=3mm,width=2mm]},blue, line width=1pt] (2,0)--(2,1);
     \draw[-{Stealth[length=3mm,width=2mm]},blue, line width=1pt] (2,1)--(3,1);
      \draw[-{Stealth[length=3mm,width=2mm]},blue, line width=1pt] (3,1)--(3,2);
     \draw[-{Stealth[length=3mm,width=2mm]},blue, line width=1pt] (3,2)--(4,2);
     \draw[-{Stealth[length=3mm,width=2mm]},blue, line width=1pt] (4,2)--(4,3);

    \node at (-1.25,0) {$ z_1$};
    \node at (-1.25,1) {$ z_2$};
    \node at (-1.25, 2) {$z_3$};
    \node[circle, draw = violet, fill=violet,opacity = 0.4, ultra thick, minimum size=7pt, inner sep=2pt] (c) at (0,0){};
    \node[circle, draw = violet, fill=violet,opacity = 0.4, ultra thick, minimum size=7pt, inner sep=2pt] (c) at (1,0){};
    \node[circle, draw = violet, fill=violet, opacity = 0.4,ultra thick, minimum size=7pt, inner sep=2pt] (c) at (0,1){};
    \node[circle, draw = violet, fill=violet, opacity = 0.4,ultra thick, minimum size=7pt, inner sep=2pt] (c) at (1,1){};
    \node[circle, draw = violet, fill=violet, opacity = 0.4,ultra thick, minimum size=7pt, inner sep=2pt] (c) at (0,2){};
    \node[circle, draw = violet, fill=violet,opacity = 0.4, ultra thick, minimum size=7pt, inner sep=2pt] (c) at (1,2){};

    \draw [-Stealth](5,1) -- (6,1);    
    \node at (5.5,1.25) {$\psi$};
  \end{tikzpicture}}
    \end{minipage} \quad \quad \quad \quad
  \begin{minipage}{0.25\textwidth}
      $\ytableaushort{1 1,2 2, 3 3}$
  \end{minipage}}
\]
  \commentAlt{Four rows of decorated states. The first row consists of a decorated state and the corresponding semistandard Young tableaux. Each of the remaining three rows contains a decorated state, the state after uncrowding, and the corresponding semistandard Young tableau.}

    We obtain one decorated state with no non-trivial bumps corresponding to the shape $(2,2),$ two decorated states with one non-trivial bump, each uncrowding to the 
    shape $(2,2,1),$ and one decorated state with two non-trivial bumps uncrowding to the shape $(2,2,2).$ This gives us
    $$\G_\lambda(z_1, z_2, z_3; \beta) = s_{(2,2,0)} + 2\beta s_{(2,2,1)} + \beta^2s_{(2,2,2)}.$$
\end{example}

\section{Crystal Operators on Lattice Models}
\label{section.crystal}

In this section, we define a $U_q(\mathfrak{sl}_n)$-crystal structure on $\decoratedstate_\lambda$ and show that it corresponds to
the crystal structure on set-valued tableaux defined in~\cite{Monical_2020} under the map $\psi$ in~\eqref{eq:bij-decoratedstate-SVT}.

\subsection{Definition of crystal operators}
For a fixed partition $\lambda$, let $n \geqslant \ell(\lambda)$ and $m=\lambda_1+n$. For $i \in I :=\{1,2,\ldots, n-1\}$, define 
\[
	e_i, f_i\colon \decoratedstate_\lambda \rightarrow \decoratedstate_\lambda \sqcup \{0\}
\]
as follows.

Fix $S^\text{dec} \in \decoratedstate_\lambda$ and $i \in I$. Recall that the contributing vertices of $S^\text{dec}$ are the $\tt b_2$ vertices and non-trivial bumps associated
to $\tt a_2$ vertices in Figure~\ref{fig:uncolored_wts}.
We first describe the signature rule for $S^\text{dec}$. Recall that lattice rows are indexed from bottom to top, so lattice row $i+1$ lies immediately above lattice row $i$.
\begin{enumerate}
    \item [(D1)] Decorate each contributing vertex in lattice row $i$ with $\rightbrack$. Decorate each contributing vertex in lattice row $i+1$ with $\leftbrack$.
    \item [(D2)] First, if $\leftbrack$ lies on an $\tt a_2$ vertex, pair it with $\rightbrack$ in the column immediately to the left of the $\leftbrack$, if it exists.
    \item [(D3)] Next, if $\leftbrack$ lies on a $\tt b_2$ vertex, pair it with the closest $\rightbrack$ that lies in a column weakly to the right of that $\leftbrack$.
\end{enumerate}

We illustrate the two bracketing steps in \Cref{fig:lattice_bracketing}, where the dashed orange and green arrows indicate adjacent path segments, 
which may be empty.

\begin{figure}[hbt!]
    \centering
\begin{tikzpicture}
     \draw (0,1)--(3,1);
     \draw (0,2)--(3,2);
     \draw (1,0)--(1,3);
     \draw (2,0)--(2,3);

    \draw[dotted][-{Stealth[length=3mm,width=2mm]},ForestGreen,line width=2.5pt](0,1)--(1,1);
    \draw[dotted][-{Stealth[length=3mm,width=2mm]},ForestGreen,line width=2.5pt](1,0)--(1,1);
    
    \draw[-{Stealth[length=3mm,width=2mm]},ForestGreen,line width=3pt] (1,1)--(2,1);
    \draw[-{Stealth[length=3mm,width=2mm]},ForestGreen,line width=3pt] (2,1)--(2,2);
    \draw[-{Stealth[length=3mm,width=2mm]},ForestGreen,line width=3pt] (2,2)--(3,2);

    \node at (2.1,2) {\Huge $\boldsymbol{\leftbrack}$};
    \node at (1.1,1) {\Huge $\boldsymbol{\rightbrack}$};
    
    \node at (-.5,1) {$ z_i$ };
    \node at (-.5,2) {$ z_{i+1}$ };

    \node at (0.5,1) {\Huge $($};
    \node at (2.5,2) {\Huge $)$};
  \end{tikzpicture} \qquad
  \begin{tikzpicture}
     \draw (0,0)--(3,0);
     \draw (0,1)--(3,1);
     \draw (1,-1)--(1,2);
     \draw (2,-1)--(2,2);

     \draw[dotted][-{Stealth[length=3mm,width=2mm]},ForestGreen,line width=2.5pt](0,0)--(1,0);
     \draw[dotted][-{Stealth[length=3mm,width=2mm]},ForestGreen,line width=2.5pt](1,-1)--(1,0);
     \draw[-{Stealth[length=3mm,width=2mm]},ForestGreen,line width=3pt] (1,0)--(2,0);
     \draw[-{Stealth[length=3mm,width=2mm]},ForestGreen,line width=3pt] (2,0)--(3,0);

     \draw[dotted][-{Stealth[length=3mm,width=2mm]},orange,line width=2.5pt](2,1)--(2,2);
     \draw[dotted][-{Stealth[length=3mm,width=2mm]},orange,line width=2.5pt](2,1)--(3,1);
     \draw[-{Stealth[length=3mm,width=2mm]},orange,line width=3pt] (0,1)--(1,1);
     \draw[-{Stealth[length=3mm,width=2mm]},orange,line width=3pt] (1,1)--(2,1);

    \node at (1.1,1) {\Huge $\boldsymbol{\leftbrack}$};
    \node at (2.1,0) {\Huge $\boldsymbol{\rightbrack}$};
    
    \node at (-.5,0) {$ z_i$ };
    \node at (-.5,1) {$ z_{i+1}$ };

    \node at (0.5,1) {\Huge $($};
    \node at (2.5,0) {\Huge $)$};
  \end{tikzpicture}
      \caption{The bracketing rules (D2) on the left and (D3) on the right.}
        \label{fig:lattice_bracketing}
        \commentAlt{Figure~\ref{fig:lattice_bracketing}. Two lattices on a 2 by 2 lattice exemplifying the possible local bracketing of a state. See long description.}
        \commentLongAlt{Figure~\ref{fig:lattice_bracketing}. The left image contains a left bracket on a bump vertex in the first row second column and a right bracket in the second row first column. The right image contains a left bracket in the first row first column and a right bracket in the second row second column.}
\end{figure}
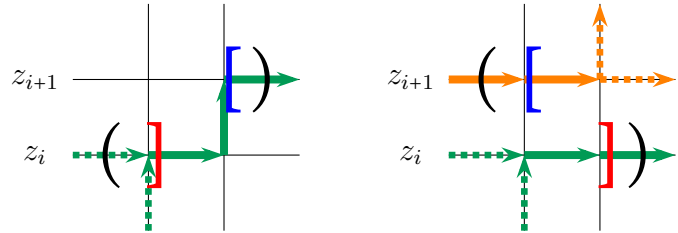

After bracketing, $f_i$ acts on the vertex corresponding to the rightmost unbracketed $\rightbrack$. The operator $e_i$ acts on the vertex corresponding to the 
leftmost unbracketed $\leftbrack$. In either case, if no such $\leftbrack$ or $\rightbrack$ exists, then the crystal operator returns 0. 

Now we describe the action of $f_i$. Assume the rightmost unbracketed $\rightbrack$ exists, and let $v$ be the corresponding vertex. By inspecting 
the admissible local configurations, exactly one of the following three cases occurs at $v$.

\smallskip

\noindent
\textbf{Case $1f$:} Suppose the vertices above and northeast of a $\tt b_2$ vertex $v$ are $u$ and $w$ respectively, where $w$ is a trivial bump. Then $f_i(S^\text{dec})$ 
shifts the path from $v$ to $w$ one unit to the left and $u$ becomes a trivial bump as shown in \Cref{fig:lattice-fi-ei-1}.

\begin{figure}[H]
    \centering
\begin{tikzpicture}
     \draw (1,0)--(1,3);
    \draw (2,0)--(2,3);
    \draw (3,0)--(3,3);
    \draw (0,1)--(4,1);
    \draw (0,2)--(4,2);

    \draw[-{Stealth[length=3mm,width=2mm]},ForestGreen,line width=3pt](0,1)--(1,1);
    \draw[-{Stealth[length=3mm,width=2mm]},ForestGreen,line width=3pt] (1,1)--(2,1);
    \draw[-{Stealth[length=3mm,width=2mm]},ForestGreen,line width=3pt] (2,1)--(2,2);
    \draw[-{Stealth[length=3mm,width=2mm]},ForestGreen,line width=3pt] (2,2)--(3,2);
    \draw[-{Stealth[length=3mm,width=2mm]},ForestGreen,line width=3pt] (3,2)--(3,3);

    \draw[dotted][-{Stealth[length=3mm,width=2mm]},orange,line width=2.5pt] (0,2)--(1,2);
    \draw[dotted][-{Stealth[length=3mm,width=2mm]},orange,line width=2.5pt] (1,2)--(1,3);

    \node at (1.1,1) {\Huge $\boldsymbol{\rightbrack}$};
    \node[cross out, draw = violet, ultra thick, minimum size=10pt] at (2,2){};
    
    \node at (-.5,1) {$ z_i$ };
    \node at (-.5,2) {$ z_{i+1}$ };
    \node at (0.8,1.2) {$v$};
    \node at (0.7,2.3) {$u$};
    \node at (1.7,2.3) {$w$};
    
    \draw[-Stealth](5,2) -- (6,2);
    \node at (5.5,2.25) {$f_i$};
    \draw[Stealth-](5,1) -- (6,1);
    \node at (5.5,.75) {$e_i$};
    
  \end{tikzpicture}\qquad
  \begin{tikzpicture}
     \draw (1,0)--(1,3);
    \draw (2,0)--(2,3);
    \draw (3,0)--(3,3);
    \draw (0,1)--(4,1);
    \draw (0,2)--(4,2);

    \draw[-{Stealth[length=3mm,width=2mm]},ForestGreen,line width=3pt](0,1)--(1,1);
    \draw[-{Stealth[length=3mm,width=2mm]},ForestGreen,line width=3pt] (1,1)--(1,2);
    \draw[-{Stealth[length=3mm,width=2mm]},ForestGreen,line width=3pt] (1,2)--(2,2);
    \draw[-{Stealth[length=3mm,width=2mm]},ForestGreen,line width=3pt] (2,2)--(3,2);
    \draw[-{Stealth[length=3mm,width=2mm]},ForestGreen,line width=3pt] (3,2)--(3,3);
    
    \draw[dotted][-{Stealth[length=3mm,width=2mm]},orange,line width=2.5pt] (0,2)--(1,2);
    \draw[dotted][-{Stealth[length=3mm,width=2mm]},orange,line width=2.5pt] (1,2)--(1,3);

    \node[cross out, draw = violet, ultra thick, minimum size=10pt, inner sep=2pt] at (1,2){};
    \node at (-.5,1) {$ z_i$ };
    \node at (-.5,2) {$ z_{i+1}$ };
    \node at (0.8,1.2) {$v$};
    \node at (0.8,2.3) {$u$};
    \node at (1.7,2.3) {$w$};
    \node at (2.1,2) {\Huge $\boldsymbol{\leftbrack}$};
    
  \end{tikzpicture}

      \caption{Crystal operator $f_i$ and $e_i$ for Case $1f$ and $1e$.}
        \label{fig:lattice-fi-ei-1}
        \commentAlt{Figure~\ref{fig:lattice-fi-ei-1}. The left lattice contains a trivial bump in the first row second column and a right bracket in the second row first column. The right lattice contains a left bracket in the first row, second column. See long description.}
        \commentLongAlt{Figure~\ref{fig:lattice-fi-ei-1}. In the left lattice, there is a green path starting on the boundary on the second row with a sequence of five arrows pointing right, right, up, right, then up. The path starting on the boundary on the second row in the right lattice instead points right, up, right, right, up.}
\end{figure}
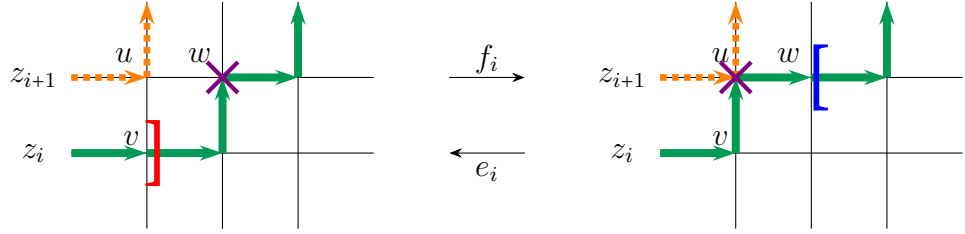

\noindent
\textbf{Case $2f$:} Suppose $v$ is a non-trivial bump and $w$ is a trivial bump northeast of $v$. Then $f_i(S^\text{dec})$ turns $v$ into a trivial bump and $w$ into a non-trivial 
bump as shown in \Cref{fig:lattice-fi-ei-2}.
\begin{figure}[H]
    \centering
\begin{tikzpicture}
     \draw (1,0)--(1,3);
    \draw (2,0)--(2,3);
    \draw (3,0)--(3,3);
    \draw (0,1)--(4,1);
    \draw (0,2)--(4,2);

    \draw[-{Stealth[length=3mm,width=2mm]},ForestGreen,line width=3pt] (1,0)--(1,1);
    \draw[-{Stealth[length=3mm,width=2mm]},ForestGreen,line width=3pt] (1,1)--(2,1);
    \draw[-{Stealth[length=3mm,width=2mm]},ForestGreen,line width=3pt] (2,1)--(2,2);
    \draw[-{Stealth[length=3mm,width=2mm]},ForestGreen,line width=3pt] (2,2)--(3,2);

    \draw[-{Stealth[length=3mm,width=2mm]},ForestGreen,line width=3pt] (3,2)--(3,3);

    \draw[dotted][-{Stealth[length=3mm,width=2mm]},orange,line width=2.5pt] (0,2)--(1,2);
    \draw[dotted][-{Stealth[length=3mm,width=2mm]},orange,line width=2.5pt] (1,2)--(1,3);

    \node[cross out, draw = violet, ultra thick, minimum size=10pt] at (2,2){};
    \node at (1.1,1) {\Huge $\boldsymbol{\rightbrack}$};
    \node at (0.8,1.2) {$v$};
    \node at (1.7,2.3) {$w$};
    \node at (-.5,1) {$ z_i$ };
    \node at (-.5,2) {$ z_{i+1}$ };
    
    \draw[-Stealth](5,2) -- (6,2);
    \node at (5.5,2.25) {$f_i$};
    \draw[Stealth-](5,1) -- (6,1);
    \node at (5.5,.75) {$e_i$};
    
  \end{tikzpicture} \qquad
\begin{tikzpicture}
     \draw (1,0)--(1,3);
    \draw (2,0)--(2,3);
    \draw (3,0)--(3,3);
    \draw (0,1)--(4,1);
    \draw (0,2)--(4,2);

    \draw[-{Stealth[length=3mm,width=2mm]},ForestGreen,line width=3pt] (1,0)--(1,1);
    \draw[-{Stealth[length=3mm,width=2mm]},ForestGreen,line width=3pt] (1,1)--(2,1);
    \draw[-{Stealth[length=3mm,width=2mm]},ForestGreen,line width=3pt] (2,1)--(2,2);
    \draw[-{Stealth[length=3mm,width=2mm]},ForestGreen,line width=3pt] (2,2)--(3,2);

    \draw[-{Stealth[length=3mm,width=2mm]},ForestGreen,line width=3pt] (3,2)--(3,3);

    \draw[dotted][-{Stealth[length=3mm,width=2mm]},orange,line width=2.5pt] (0,2)--(1,2);
    \draw[dotted][-{Stealth[length=3mm,width=2mm]},orange,line width=2.5pt] (1,2)--(1,3);

    \node[cross out, draw = violet, ultra thick, minimum size=10pt] at (1,1){};
    \node at (0.6,1.2) {$v$};
    \node at (1.7,2.3) {$w$};    
    \node at (-.5,1) {$ z_i$ };
    \node at (-.5,2) {$ z_{i+1}$ };
    \node at (2.1,2) {\Huge $\boldsymbol{\leftbrack}$};
  \end{tikzpicture}
  \caption{Crystal operator $f_i$ and $e_i$ for Case $2f$ and $2e$.}
  \label{fig:lattice-fi-ei-2}
  \commentAlt{Figure~\ref{fig:lattice-fi-ei-2}. The left lattice contains a trivial bump in the first row second column and a right bracket in the second row first column. The right lattice contains a left bracket in the first row second column and a non-trivial bump in the second row first column. See long description.}
  \commentLongAlt{Figure~\ref{fig:lattice-fi-ei-2}. In the left lattice, there is a green path starting in the first column and below the second row with a sequence of five arrows pointing up, right, up, right, then up. The path starting in the first column below the second row in the right lattice instead points up, right, up, right, up.}
\end{figure}
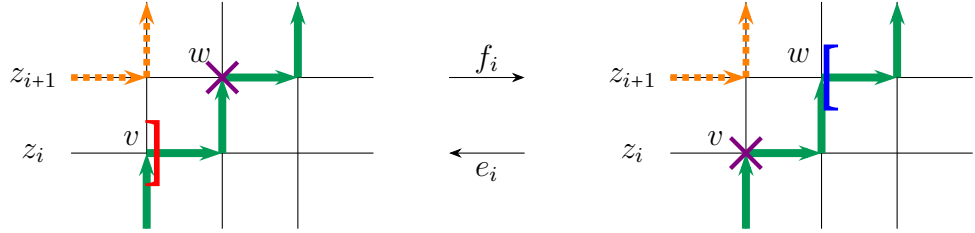

\smallskip

\noindent
\textbf{Case $3f$:} Suppose the local configuration contains vertices $w$ and $u$ that form a bracketed pair, and $v$ is either a non-trivial bump, or a $\tt b_2$ vertex. Then $f_i(S^\text{dec})$ shifts the path from $w$ to $u$ as shown in \Cref{fig:lattice-fi-ei-3}.

\begin{figure}[H]
    \centering
\begin{tikzpicture}
     \draw (1,0)--(1,3);
    \draw (2,0)--(2,3);
    \draw (3,0)--(3,3);
    \draw (0,1)--(4,1);
    \draw (0,2)--(4,2);

    \draw[-{Stealth[length=3mm,width=2mm]},ForestGreen,line width=3pt] (1,1)--(2,1);
    \draw[-{Stealth[length=3mm,width=2mm]},ForestGreen,line width=3pt] (2,1)--(3,1);
    \draw[-{Stealth[length=3mm,width=2mm]},ForestGreen,line width=3pt] (3,1)--(3,2);
    \draw[-{Stealth[length=3mm,width=2mm]},ForestGreen,line width=3pt] (3,2)--(4,2);

    \draw[dotted][-{Stealth[length=3mm,width=2mm]},orange,line width=2.5pt] (0,2)--(1,2);
    \draw[dotted][-{Stealth[length=3mm,width=2mm]},orange,line width=2.5pt] (1,2)--(1,3);
    \draw[dotted][-{Stealth[length=3mm,width=2mm]},ForestGreen,line width=2.5pt](0,1)--(1,1);
     \draw[dotted][-{Stealth[length=3mm,width=2mm]},ForestGreen,line width=2.5pt](1,0)--(1,1);

    \node at (1.1,1) {\Huge $\boldsymbol{\rightbrack}$};
    \node at (2.1,1) {\Huge $\boldsymbol{\rightbrack}$};
    \node at (3.1,2) {\Huge $\boldsymbol{\leftbrack}$};
    \node at (3.5,2) {\Huge $)$};
    \node at (1.6,1) {\Huge $($};
    
    \node at (-.5,1) {$ z_i$ };
    \node at (-.5,2) {$ z_{i+1}$ };
    \node at (0.8,1.2) {$v$};
    \node at (2.8,2.2) {$u$};
    \node at (1.8,1.2) {$w$};
    \draw[-Stealth](5,2) -- (6,2);
    \node at (5.5,2.25) {$f_i$};
    \draw[Stealth-](5,1) -- (6,1);
    \node at (5.5,.75) {$e_i$};

  \end{tikzpicture} \qquad
  \begin{tikzpicture}
     \draw (1,0)--(1,3);
    \draw (2,0)--(2,3);
    \draw (3,0)--(3,3);
    \draw (0,1)--(4,1);
    \draw (0,2)--(4,2);

     \draw[dotted][-{Stealth[length=3mm,width=2mm]},ForestGreen,line width=2.5pt](0,1)--(1,1);
     \draw[dotted][-{Stealth[length=3mm,width=2mm]},ForestGreen,line width=2.5pt](1,0)--(1,1);
    \draw[-{Stealth[length=3mm,width=2mm]},ForestGreen,line width=3pt] (1,1)--(2,1);
    \draw[-{Stealth[length=3mm,width=2mm]},ForestGreen,line width=3pt] (2,1)--(2,2);
    \draw[-{Stealth[length=3mm,width=2mm]},ForestGreen,line width=3pt] (2,2)--(3,2);
    \draw[-{Stealth[length=3mm,width=2mm]},ForestGreen,line width=3pt] (3,2)--(4,2);
    
    \draw[dotted][-{Stealth[length=3mm,width=2mm]},orange,line width=2.5pt] (0,2)--(1,2);
    \draw[dotted][-{Stealth[length=3mm,width=2mm]},orange,line width=2.5pt] (1,2)--(1,3);

    \node at (-.5,1) {$ z_i$ };
    \node at (-.5,2) {$ z_{i+1}$ };
    \node at (0.8,1.2) {$v$};
    \node at (2.8,2.2) {$u$};
    \node at (1.8,1.2) {$w$};
    \node at (1.8,2.2) {$t$};

    \node at (3.1,2) {\Huge $\boldsymbol{\leftbrack}$};
    \node at (1.1,1) {\Huge $\boldsymbol{\rightbrack}$};
    \node at (2.1,2) {\Huge $\boldsymbol{\leftbrack}$};
    \node at (2.4,2) {\Huge $)$};
    \node at (0.6,1) {\Huge $($};
  \end{tikzpicture}
  \caption{Crystal operator $f_i$ and $e_i$ for Case $3f$ and $3e$.}
      \label{fig:lattice-fi-ei-3}
      \commentAlt{Figure~\ref{fig:lattice-fi-ei-3}. The left lattice contains a left bracket in the first row third column and two right brackets in the second row, first and second columns. The right lattice contains two left brackets in the first row, second and third columns as well as a right bracket in the second row first column.}
\end{figure}
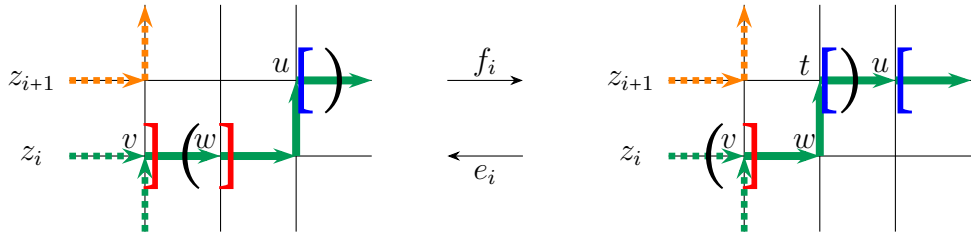

We now describe the action of the $e_i$ operator. Assume the leftmost unbracketed $\leftbrack$ exists, and let $w$ be the corresponding vertex. Let $v$ be the vertex one column left and one row below $w.$ 

\smallskip

\noindent
\textbf{Case $1e$:} Let $w$ be the leftmost unbracketed $\leftbrack$, and let $v$ be the vertex one column left and one row below $w.$ 
Suppose $w$ is a $\tt b_2$ vertex and the vertex immediately to the left is either a trivial bump or a $\tt b_1$ vertex. Then $e_i(S^\text{dec})$ 
shifts the path from $w$ to $v$ one unit to the right and $w$ becomes a trivial bump while $v$ becomes a $\tt b_2$ vertex as shown in \Cref{fig:lattice-fi-ei-1}.

\noindent
\textbf{Case $2e$:} Let $w$ be the leftmost unbracketed $\leftbrack$, and let $v$ be the vertex one column left and one row below $w.$ Suppose 
$w$ is a non-trivial bump and $v$ is a trivial bump. Then $e_i(S^\text{dec})$ turns $w$ into a trivial bump and $v$ into 
a non-trivial bump as shown in Figure~\ref{fig:lattice-fi-ei-2}.

\smallskip

\noindent
\textbf{Case $3e$:}  Let $u$ be the leftmost unbracketed $\leftbrack$, let $t$ be a non-trivial bump immediately to the left of $u$, and let $v$ be the 
vertex one column below and to the left of $t$, where $v$ and $t$ form a bracketed pair. Then, $e_i(S^\text{dec})$ shifts the path from $v$ to $t$ one unit to the 
right and $u$ becomes a non-trivial bump while $w$, the vertex to the right of $v$, becomes a $\tt b_2$ vertex as shown in Figure~\ref{fig:lattice-fi-ei-3}.

\begin{lemma}
    The lattice crystal operators $e_i$ and $f_i$ on $\decoratedstate_\lambda$ are well-defined. 
\end{lemma}

\begin{proof} We prove the statement for $f_i$, and the proof for $e_i$ is analogous.
The bracketing rule determines which symbols remain unbracketed, so the rightmost unbracketed $\rightbrack$ when it exists, is uniquely determined. 
If such a $\rightbrack$ exists, then the admissible local configurations around $v$ fall in one of Cases $1f$--$3f$. In each case, the local move changes 
an admissible decorated state into another, and the boundary conditions do not change. 
If such a $\rightbrack$ does not exist, then the corresponding operator is defined to be $0$.
\end{proof}

The next lemma follows by comparing the corresponding local moves in Cases $1f$ and $1e$, $2f$ and $2e$, and $3f$ and $3e$.

\begin{lemma}
    For each $i \in I$, the lattice crystal operators $e_i$ and $f_i$ are partial inverses. 
\end{lemma}

Let us now recall a crystal structure on set-valued tableaux defined in~\cite{Monical_2020}. It was proven in~\cite{Monical_2020} that the 
operators $\widetilde{f}_i$ and $\widetilde{e}_i$ below indeed define a $U_q(\mathfrak{sl}_n)$-crystal structure on $\SVT^n(\lambda)$.

We begin by defining the signature rule from~\cite{Monical_2020}. Fix $T\in \SVT^n(\lambda)$ and $i \in I$. 

  \begin{enumerate}
    \item [(T1)] Write $\rightbrack$ above each column of $T$ containing an $i$. Write $\leftbrack$ above each column of $T$ containing an $i+1$.
    \item [(T2)] First, cancel any $\rightbrack \; \leftbrack$ pairs if $i$ and $i+1$ are in the same column of $T$. 
    \item [(T3)] Next, successively cancel any remaining $\leftbrack \; \rightbrack$ pairs.
\end{enumerate}

    If every $\rightbrack$ cancels, then $\widetilde{f}_i(T)=0$. 
    Otherwise, let $A$ be the multicell containing the rightmost unbracketed $\rightbrack$, and let $B$ be the multicell immediately to the right of $A$. We have two cases.
    \begin{enumerate}
        \item If $i \notin B$ or $B$ does not exist, $\widetilde{f}_i(T)$ replaces the entry $i$ in $A$ with $i+1$.
        \item If $i \in B$, then $A$ (the green multicell) contains $i$, 
        and $B$ (the yellow multicell) contains both $i$ and $i+1$. In this case, $\widetilde{f}_i(T)$ removes $i$ from $B$ and adds $i+1$ to $A$. 
\begin{equation*} \text{$A,B$ locally in $T$: }
\ytableausetup{boxsize=3em}{
\begin{ytableau}
    *(green!25)i & *(yellow!35) i,\;i+1
\end{ytableau}} \qquad \qquad 
\text{$A,B$ locally in $\widetilde{f}_i(T)$: }
    \ytableausetup{boxsize=3em}{
\begin{ytableau}
    *(green!25)i, \;i+1 & *(yellow!35) i+1
\end{ytableau}}
\end{equation*}
\end{enumerate}
The action of $\widetilde{e}_i$ is analogous. If every $\leftbrack$ cancels, then $\widetilde{e}_i(T)=0$. 
Otherwise, let $B$ be the multicell containing the leftmost unbracketed $\leftbrack$, and let $A$ be the multicell immediately to the 
left of $B$. We have two cases.
    \begin{enumerate}
        \item If $i+1 \notin A$ or $A$ does not exist, $\widetilde{e}_i(T)$ replaces the entry $i+1$ in $B$ with $i$.
        \item If $i+1 \in A$, then $A$ contains both $i$ and $i+1$, and $B$ contains $i+1$. In this case, $\widetilde{e}_i(T)$ removes $i+1$ 
        from $A$ and adds $i$ to $B$. 
    \end{enumerate}

Compared to~\cite{Monical_2020}, we use the symbols by $+ \leftrightarrow \rightbrack$ and $- \leftrightarrow \leftbrack$ for the signature rule
in the description above.

\begin{theorem}\label{thm:crystal_commutes_bij}
The following diagrams commute:

    \[
\begin{tikzcd}[column sep=large,row sep=large]
	\decoratedstate_\lambda \arrow["\psi",d] \arrow[r,"f_i"] & \decoratedstate_\lambda\sqcup\{0\} \arrow[d,"\psi"] \\
	\SVT^n(\lambda) \arrow[r,"\widetilde{f}_i"]& \SVT^n(\lambda)\sqcup\{0\}
\end{tikzcd}
\qquad
\begin{tikzcd}[column sep=large,row sep=large]
	\decoratedstate_\lambda \arrow["\psi",d] \arrow[r,"e_i"] & \decoratedstate_\lambda\sqcup\{0\} \arrow[d,"\psi"] \\
	\SVT^n(\lambda) \arrow[r,"\widetilde{e}_i"]& \SVT^n(\lambda)\sqcup\{0\}
\end{tikzcd}
\]

\end{theorem}

\begin{proof}
Let $S^\text{dec} \in \decoratedstate_\lambda$. It follows from the bijection $\psi$ in~\eqref{eq:bij-decoratedstate-SVT} and Lemma~\ref{lem:paths-to-SVTtab} 
that the signature rule (T1)--(T3) on $\SVT^n(\lambda)$ is exactly the same as the signature rule (D1)--(D3) on $\decoratedstate_\lambda$.
Hence in particular, $f_i(S^\text{dec})=0$ if and only if $\widetilde{f}_i(\psi(S^\text{dec}))=0$. The analogous statement also holds for $e_i$ and $\widetilde{e}_i$.

Next assume that $S^\text{dec} \in \decoratedstate_\lambda$ with $f_i(S^\text{dec})\neq 0$. We will show that $\psi(f_i(S^\text{dec})) = \widetilde{f}_i(\psi(S^\text{dec}))$. The identity 
$\psi(e_i(S^\text{dec})) = \widetilde{e}_i(\psi(S^\text{dec}))$ then follows since $e_i$ and $f_i$ are partial inverses. Define $T = \psi(S^\text{dec})$ and let $A$ be the multicell 
containing the rightmost unbracketed $\rightbrack$ and $B$ be the multicell immediately to the right of $A$ if it exists. For simplicity, we say that a 
vertex of $S^\text{dec}$ of weight $z_i$ has weight $i.$ 

Note that in $S^{\text{dec}}$, bracketing a $\leftbrack$ on a non-trivial bump with a $\rightbrack$ immediately to the left, if it exists, is equivalent to bracketing $i$ and $i+1$ in the same multicell of $T.$ Furthermore, bracketing all other $\leftbrack$ with the closest $\rightbrack$ that lies in a column weakly to the right of that $\leftbrack$ is equivalent to bracketing $i$ with $i+1$ in $T$ if $i+1$ lies in a column weakly to the left of $i.$ Then the rightmost unbracketed $\rightbrack$ in $S^\text{dec}$ corresponds to the entry $i$ in $A$.

\smallskip

\noindent
\textbf{Case 1:} If $i \notin B,$ we know that $\widetilde{f}_i(T)$ will replace $i$ in $A$ with $i+1.$ If $A$ is not a multicell, that is, the only entry in $A$ is $i,$ then locally $S^\text{dec}$ will look like Case $1f.$ Then $f_i(S^\text{dec})$ will remove a $\tt b_2$ vertex with weight $i$ and add a $\tt b_2$ vertex with weight $i+1.$ See Figure~\ref{fig:case1f_commute}. If the $i$ in $A$ is an excess entry, then locally $S^\text{dec}$ will look like Case $2f.$ Then $f_i(S^\text{dec})$ will turn a bump of weight $i$ into a trivial bump and a trivial bump into a non-trivial $i+1$ bump. In either case, $f_i(S^\text{dec})$ removes a vertex of weight $i$ and adds a vertex of weight $i+1$ that corresponds to the same box $A$ in $T.$ 
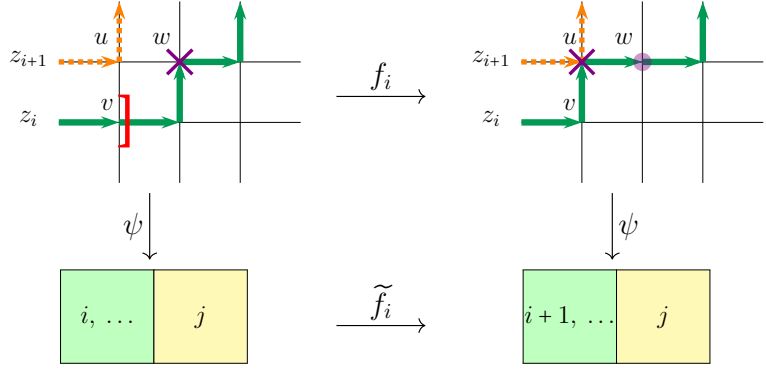
\begin{figure}[H]
    \newcommand{\latticebox}[1]{\makebox[5.8cm][c]{#1}}
\begin{center}\scalebox{0.8}{\begin{tikzcd}[
  ampersand replacement=\&,
  column sep=large,
  row sep=large,
  every label/.append style={font=\large}
]
\latticebox{
\begin{tikzpicture}
     \draw (1,0)--(1,3);
    \draw (2,0)--(2,3);
    \draw (3,0)--(3,3);
    \draw (0,1)--(4,1);
    \draw (0,2)--(4,2);

    \draw[-{Stealth[length=3mm,width=2mm]},ForestGreen,line width=3pt](0,1)--(1,1);
    \draw[-{Stealth[length=3mm,width=2mm]},ForestGreen,line width=3pt] (1,1)--(2,1);
    \draw[-{Stealth[length=3mm,width=2mm]},ForestGreen,line width=3pt] (2,1)--(2,2);
    \draw[-{Stealth[length=3mm,width=2mm]},ForestGreen,line width=3pt] (2,2)--(3,2);
    \draw[-{Stealth[length=3mm,width=2mm]},ForestGreen,line width=3pt] (3,2)--(3,3);

    \draw[dotted][-{Stealth[length=3mm,width=2mm]},orange,line width=2.5pt] (0,2)--(1,2);
    \draw[dotted][-{Stealth[length=3mm,width=2mm]},orange,line width=2.5pt] (1,2)--(1,3);

    \node at (1.1,0.8) {\Huge $\boldsymbol{\rightbrack}$};
    \node[cross out, draw = violet, ultra thick, minimum size=10pt] at (2,2){};
    
    \node at (-.5,1) {$ z_i$ };
    \node at (-.5,2) {$ z_{i+1}$ };
    \node at (0.8,1.2) {$v$};
    \node at (0.7,2.3) {$u$};
    \node at (1.7,2.3) {$w$};
\end{tikzpicture}
}
\arrow[d,"\psi"']
\arrow[r,shift left=13,"f_i"]
\&
\latticebox{
\begin{tikzpicture}
     \draw (1,0)--(1,3);
    \draw (2,0)--(2,3);
    \draw (3,0)--(3,3);
    \draw (0,1)--(4,1);
    \draw (0,2)--(4,2);

    \draw[-{Stealth[length=3mm,width=2mm]},ForestGreen,line width=3pt](0,1)--(1,1);
    \draw[-{Stealth[length=3mm,width=2mm]},ForestGreen,line width=3pt] (1,1)--(1,2);
    \draw[-{Stealth[length=3mm,width=2mm]},ForestGreen,line width=3pt] (1,2)--(2,2);
    \draw[-{Stealth[length=3mm,width=2mm]},ForestGreen,line width=3pt] (2,2)--(3,2);
    \draw[-{Stealth[length=3mm,width=2mm]},ForestGreen,line width=3pt] (3,2)--(3,3);
    
    \draw[dotted][-{Stealth[length=3mm,width=2mm]},orange,line width=2.5pt] (0,2)--(1,2);
    \draw[dotted][-{Stealth[length=3mm,width=2mm]},orange,line width=2.5pt] (1,2)--(1,3);

    \node[cross out, draw = violet, ultra thick, minimum size=10pt, inner sep=2pt] at (1,2){};
    \node[circle, draw = violet, fill=violet, opacity = 0.4,ultra thick, minimum size=7pt, inner sep=2pt] at (2,2){};
    \node at (-.5,1) {$ z_i$ };
    \node at (-.5,2) {$ z_{i+1}$ };
    \node at (0.8,1.2) {$v$};
    \node at (0.8,2.3) {$u$};
    \node at (1.7,2.3) {$w$};
\end{tikzpicture}
}
\arrow[d,"\psi"]
\\ 
\latticebox{
\ytableausetup{boxsize=3.7em}
\begin{ytableau}
    *(green!25)i, \;\dots & *(yellow!35) j
\end{ytableau}
}
\arrow[r,shift left=5,"\widetilde{f}_i"]
\&
\latticebox{
\ytableausetup{boxsize=3.7em}
\begin{ytableau}
    *(green!25)i+1,\;\dots & *(yellow!35) j
\end{ytableau}
}
\end{tikzcd}}\end{center}
    \caption{The local structure of the commutative diagram in Case 1f.}
    \label{fig:case1f_commute}
    \commentAlt{Figure~\ref{fig:case1f_commute}. Square commutative diagram where the first row consists of decorated states while the second row shows the two-box local section of the corresponding set-valued tableaux.}
\end{figure}

\noindent
\textbf{Case 2:} If $i\in B,$ then $\widetilde{f}_i$ will remove $i$ from $B$ and add $i+1$ to $A.$ Then locally, $S^\text{dec}$ will look like Case $3f$ where the paired bump and $\rightbrack$ immediately to the left of it correspond to the $i$ and $i+1$ in $B$ of $T.$ Then $f_i(S^\text{dec})$ will shift the paired bump and $\tt b_2$ vertex to the left and create a vertex of weight $i+1$ after this bump. By shifting the path, a vertex of weight $i$ gets removed. This is equivalent to removing $i$ from $B$ in $T$ and adding an $i+1$ to $A.$ See Figure~\ref{fig:case3f_commute} for example.

In both cases, we have that $\psi(f_i(S^\text{dec})) = \widetilde{f}_i(\psi(S^\text{dec})).$ Since $f_i$ and $e_i$ are partial inverses, we also conclude that $\psi(e_i(S^\text{dec})) = \widetilde{e}_i(\psi(S^\text{dec})).$
\end{proof}

\begin{figure}[hbpt!]
    \centering
    \newcommand{\latticebox}[1]{\makebox[5.8cm][c]{#1}}
\scalebox{0.8}{\begin{tikzcd}[
  ampersand replacement=\&,
  column sep=large,
  row sep=large,
  every label/.append style={font=\large}
]
\latticebox{
\begin{tikzpicture}
    \draw (1,0)--(1,3);
    \draw (2,0)--(2,3);
    \draw (3,0)--(3,3);
    \draw (0,1)--(4,1);
    \draw (0,2)--(4,2);

    \draw[-{Stealth[length=3mm,width=2mm]},ForestGreen,line width=3pt] (1,1)--(2,1);
    \draw[-{Stealth[length=3mm,width=2mm]},ForestGreen,line width=3pt] (2,1)--(3,1);
    \draw[-{Stealth[length=3mm,width=2mm]},ForestGreen,line width=3pt] (3,1)--(3,2);
    \draw[-{Stealth[length=3mm,width=2mm]},ForestGreen,line width=3pt] (3,2)--(4,2);

    \draw[dotted][-{Stealth[length=3mm,width=2mm]},orange,line width=2.5pt] (0,2)--(1,2);
    \draw[dotted][-{Stealth[length=3mm,width=2mm]},orange,line width=2.5pt] (1,2)--(1,3);
    \draw[dotted][-{Stealth[length=3mm,width=2mm]},ForestGreen,line width=2.5pt](0,1)--(1,1);
     \draw[dotted][-{Stealth[length=3mm,width=2mm]},ForestGreen,line width=2.5pt](1,0)--(1,1);

    \node at (1.1,0.8) {\Huge $\boldsymbol{\rightbrack}$};
    \node at (2.1,0.8) {\Huge $\boldsymbol{\rightbrack}$};
    \node at (3.1,1.8) {\Huge $\boldsymbol{\leftbrack}$};
    \node at (3.5,2) {\Huge $)$};
    \node at (1.6,1) {\Huge $($};
    
    \node at (-.5,1) {$ z_i$ };
    \node at (-.5,2) {$ z_{i+1}$ };
    \node at (0.8,1.2) {$v$};
    \node at (2.8,2.2) {$u$};
    \node at (1.8,1.2) {$w$};
\end{tikzpicture}
}
\arrow[d,"\psi"']
\arrow[r,shift left=13,"f_i"]
\&
\latticebox{
\begin{tikzpicture}
      \draw (1,0)--(1,3);
    \draw (2,0)--(2,3);
    \draw (3,0)--(3,3);
    \draw (0,1)--(4,1);
    \draw (0,2)--(4,2);

     \draw[dotted][-{Stealth[length=3mm,width=2mm]},ForestGreen,line width=2.5pt](0,1)--(1,1);
     \draw[dotted][-{Stealth[length=3mm,width=2mm]},ForestGreen,line width=2.5pt](1,0)--(1,1);
    \draw[-{Stealth[length=3mm,width=2mm]},ForestGreen,line width=3pt] (1,1)--(2,1);
    \draw[-{Stealth[length=3mm,width=2mm]},ForestGreen,line width=3pt] (2,1)--(2,2);
    \draw[-{Stealth[length=3mm,width=2mm]},ForestGreen,line width=3pt] (2,2)--(3,2);
    \draw[-{Stealth[length=3mm,width=2mm]},ForestGreen,line width=3pt] (3,2)--(4,2);
    
    \draw[dotted][-{Stealth[length=3mm,width=2mm]},orange,line width=2.5pt] (0,2)--(1,2);
    \draw[dotted][-{Stealth[length=3mm,width=2mm]},orange,line width=2.5pt] (1,2)--(1,3);

    \node[circle, draw = violet, ultra thick, minimum size=10pt] at (2,2){};
    \node at (-.5,1) {$ z_i$ };
    \node at (-.5,2) {$ z_{i+1}$ };
    \node at (0.8,1.2) {$v$};
    \node at (2.8,2.2) {$u$};
    \node at (1.8,1.2) {$w$};

    \node[circle, draw = violet, fill=violet, opacity = 0.4,ultra thick, minimum size=7pt, inner sep=2pt] at (3,2){};
    \node[circle, draw = violet, fill=violet, opacity = 0.4,ultra thick, minimum size=7pt, inner sep=2pt] at (1,1){};
\end{tikzpicture}
}
\arrow[d,"\psi"]
\\ 
\latticebox{
\ytableausetup{boxsize=4.7em}
\begin{ytableau}
    *(green!25)i & *(yellow!35) i, i+1, \;\dots
\end{ytableau}
}
\arrow[r,shift left=5,"\widetilde{f}_i"]
\&
\latticebox{
\ytableausetup{boxsize=4.7em}
\begin{ytableau}
    *(green!25)i,i+1 & *(yellow!35) i+1, \;\dots
\end{ytableau}
}
\end{tikzcd}}
    \caption{The local structure of the commutative diagram in Case 3f.}
    \label{fig:case3f_commute}
    \commentAlt{Figure~\ref{fig:case3f_commute}. Square commutative diagram where the first row consists of decorated states while the second row shows the two-box local section of the corresponding set-valued tableaux.}
\end{figure}
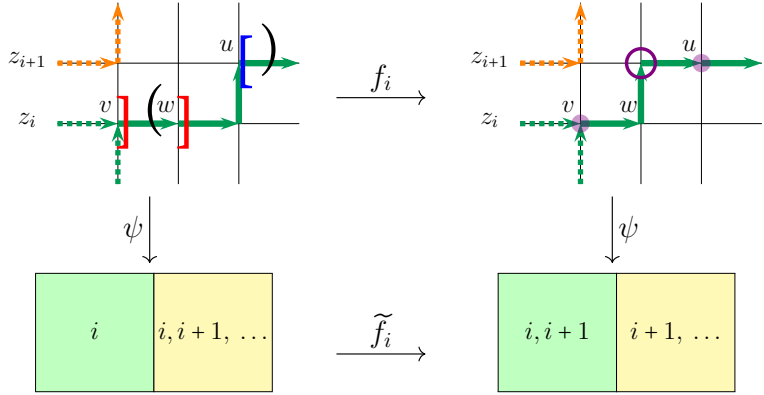

For $S^\text{dec} \in \decoratedstate_\lambda,$ we define $\wt(S^\text{dec}) = \sum_{i=1}^{n} c_i\mathbf{e}_i$ where $c_i$ is the number of contributing vertices with weight 
$z_i.$ It is clear that $\wt(f_i(S^\text{dec})) = \wt(S^\text{dec}) -\alpha_i$.

Note that $\psi$ is weight preserving, that is, $\wt(\psi(S^\text{dec})) = \wt(S^\text{dec}),$ and it is easy to see that the number of left and right brackets of $S^\text{dec}$ correspond to those of $\psi(S^\text{dec}).$ Therefore $\psi$ is a crystal isomorphism, giving us the following corollary.

\begin{corollary}
    The crystal on $\decoratedstate_\lambda$ is isomorphic to the crystal on $\SVT^n(\lambda).$
\end{corollary}

\begin{example}
    We display two isomorphic crystals, $\SVT^3(2,1,1)$ and $\decoratedstate_{(2,1,1)}$, in \Cref{fig:svt-crystal-example} and \Cref{fig:lattice-crystal-example}. 
\end{example}

\subsection{Crystal operators and uncrowding}

By~\cite[Theorem~3.12]{Monical_2020}, Buch's uncrowding map 
\[T \mapsto (P_\SVT(T), F_\SVT(T))\]
yields a crystal isomorphism from $\SVT^n(\lambda)$ onto a disjoint union of crystals on semistandard Young tableaux. In particular, the 
crystal operators intertwine with the uncrowding map, and hence
\[P_\SVT(\widetilde{f}_i (T))=\widetilde{f}_i(P_\SVT(T)),\]
with an analogous identity for $\widetilde{e}_i$.

\begin{theorem}
    Let $S^\text{dec} \in \decoratedstate_\lambda$ and fix $i \in I$. Then the lattice uncrowding map intertwines with the lattice crystal operators:
    \[f_i \circ \U_\LAT = \U_\LAT \circ f_i \qquad \text{and} \qquad e_i \circ \U_\LAT = \U_\LAT \circ e_i.\]
    Equivalently, 
    \begin{enumerate}
        \item If $f_i(S^\text{dec})=0$, then $f_i(\U_\LAT(S^\text{dec}))=0.$
        \item If $e_i(S^\text{dec})=0$, then $e_i(\U_\LAT(S^\text{dec}))=0.$
        \item If $f_i(S^\text{dec})\neq 0$, then $f_i(\U_\LAT(S^\text{dec}))=\U_\LAT(f_i(S^\text{dec})).$
        \item If $e_i(S^\text{dec})\neq 0$, then $e_i(\U_\LAT(S^\text{dec}))=\U_\LAT(e_i(S^\text{dec})).$
    \end{enumerate}
\end{theorem}

\begin{proof}
    By Theorem~\ref{thm:crystal_commutes_bij}, Theorem~\ref{thm:lattice_uncrowding_algorithm}, and~\cite[Theorem~3.12]{Monical_2020}, the 
    diagram below commutes, where each map is extended by sending 0 to 0.

    \[
\begin{tikzcd}[column sep=large,row sep=large]
\decoratedstate_\lambda
\arrow[r,"f_i"]
\arrow[d,"\psi"']
\arrow[ddd,bend right=60,"\U_\LAT"']
&
\decoratedstate_\lambda\sqcup\{0\}
\arrow[d,"\psi"]
\arrow[ddd,bend left=60,"\U_\LAT"]
\\
\SVT^n(\lambda)
\arrow[r,"\widetilde{f}_i"']
\arrow[d,"P_\SVT"']
&
\SVT^n(\lambda)\sqcup\{0\}
\arrow[d,"P_\SVT"]
\\
\displaystyle{\bigsqcup_{\mu\supseteq\lambda}\SSYT^n(\mu)}
\arrow[r,"\widetilde{f}_i"']
&
\displaystyle{\bigsqcup_{\mu\supseteq\lambda} \SSYT^n(\mu)\sqcup\{0\}}
\\
\displaystyle{\bigsqcup_{\mu\supseteq\lambda}\decoratedstate^0_\mu}
\arrow[u,"\psi"]
\arrow[r,"f_i"']
&
\displaystyle{\bigsqcup_{\mu\supseteq\lambda}\decoratedstate^0_\mu\sqcup\{0\}}
\arrow[u,"\psi"']
\end{tikzcd}
\]

Hence, $\psi(f_i(\U_\LAT(S^\text{dec})))=\psi(\U_\LAT(f_i(S^\text{dec})))$.
Since $\psi$ is injective, we obtain 
\[
	f_i(\U_\LAT(S^\text{dec}))=\U_\LAT(f_i(S^\text{dec})).
\]
If $f_i(S^\text{dec})\neq 0$, this is exactly the identity above. 
If $f_i(S^\text{dec})=0$, this becomes $f_i(\U_\LAT(S^\text{dec}))=\U_\LAT(0)=0$.

The proof for $e_i$ is analogous.
\end{proof}

\section{Generalizations and discussion}
\label{sec:future_work}

Let us conclude with a discussion of several potential generalizations of this work.

The symmetric Grothendieck polynomials have been generalized in multiple ways. 
A refined version of symmetric Grothendieck polynomials $G_\lambda(\mathbf{z};\boldsymbol{\beta})$, where
$\boldsymbol{\beta} = (\beta_1,\beta_2,\ldots)$, was introduced by Chan and Pflueger~\cite{Chan_2021} with applications to 
Brill--Noether varieties from algebraic geometry, inspired by work of Galashin, Grinberg, and Liu~\cite{GGL.2016} who refined
the parameters for dual Grothendieck polynomials. In a separate direction, Yeliussizov~\cite{Yeliussizov.2017} 
defined the canonical Grothendieck polynomials $G_\lambda(\mathbf{z}; \alpha,\beta)$ inspired by symmetries
under the $\omega$ involution on symmetric functions. Unifying both of these directions, Hwang, Jang, Kim, Song, and Song~\cite{HJKSS1,HJKSS2} 
introduced and studied refined canonical stable Grothendieck polynomials. A free fermionic model for the refined canonical Grothendieck
polynomials was given by Iwao, Motegi, and Scrimshaw~\cite{Iwao2024}, while Gunna and Zinn--Justin~\cite{MR4552711} gave a 
vertex model. The tableaux underlying the (refined) canonical Grothendieck polynomials are hook-valued tableaux, replacing the 
set-valued tableaux for the symmetric Grothendieck polynomials. Uncrowding operators for hook-valued tableaux analogous to those in
Definition~\ref{def:uncrowding-SVT} were introduced in~\cite{pan_uncrowding_2022,JKPPS.2026}. It would be interesting to generalize the analysis
of the current paper and give a lattice model interpretation of the uncrowding algorithm on hook-valued tableaux.

While the symmetric Grothendieck polynomials are $K$-theoretic analogues of the Schur polynomials indexed by partitions, Grothendieck polynomials
are $K$-theoretic generalizations of Schubert polynomials indexed by permutations. There are two prominent lattice model interpretations 
of Schubert polynomials: the classical pipe dream model~\cite{BergeronBilley.1993,KM.2005} and the more recent bumpless 
pipe dreams~\cite{LLS.2021}. The frozen pipe dream model of~\cite{BFHTW.2023} is a lattice model for Grothendieck polynomials generalizing the
classical pipe dream model for Schubert polynomials. Bumpless pipe dream models for Grothendieck polynomials are given by Weigandt~\cite{WeigandtASM} and Buciumas and Scrimshaw~\cite{BuciumasScrimshawGrothendieck}. Again, it would be interesting to extend the analysis of this paper to these 
settings. In particular, it would be desirable to define crystal operators directly on the lattice model. While this was done in the Schubert
case on classical pipe dreams~\cite{Lenart.2004,GMS.2024a} and equivalently on compatible sequences (or decreasing factorizations of reduced
words of the permutation)~\cite{MS.2016}, in the Grothendieck setting this is currently only known for 321-avoiding permutations~\cite{MPPS.2020}.
Uncrowding in this setting is related to the $\star$-insertion~\cite{MPPS.2020} and Hecke insertion~\cite{BKSTY.2008}.
The lattice models in this case are 6-vertex models and hence are beyond the scope of this paper.

\begin{figure}[H]
    \centering
    \includegraphics[scale=0.95]{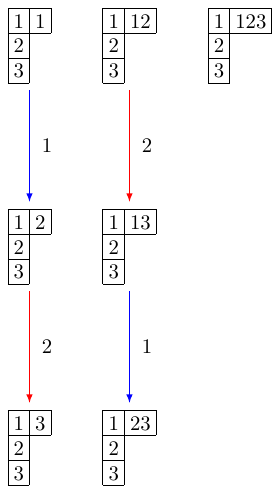}
    \caption{The crystal on $\SVT^3(2,1,1)$.}
    \label{fig:svt-crystal-example}
    \commentAlt{Figure~\ref{fig:svt-crystal-example}. The crystal contains three connected components. The first two components consist of a string of three set-valued tableaux while the last component has one singular tableau.}
\end{figure}
\begin{figure}[H]
    \centering
\includegraphics[scale=0.6]{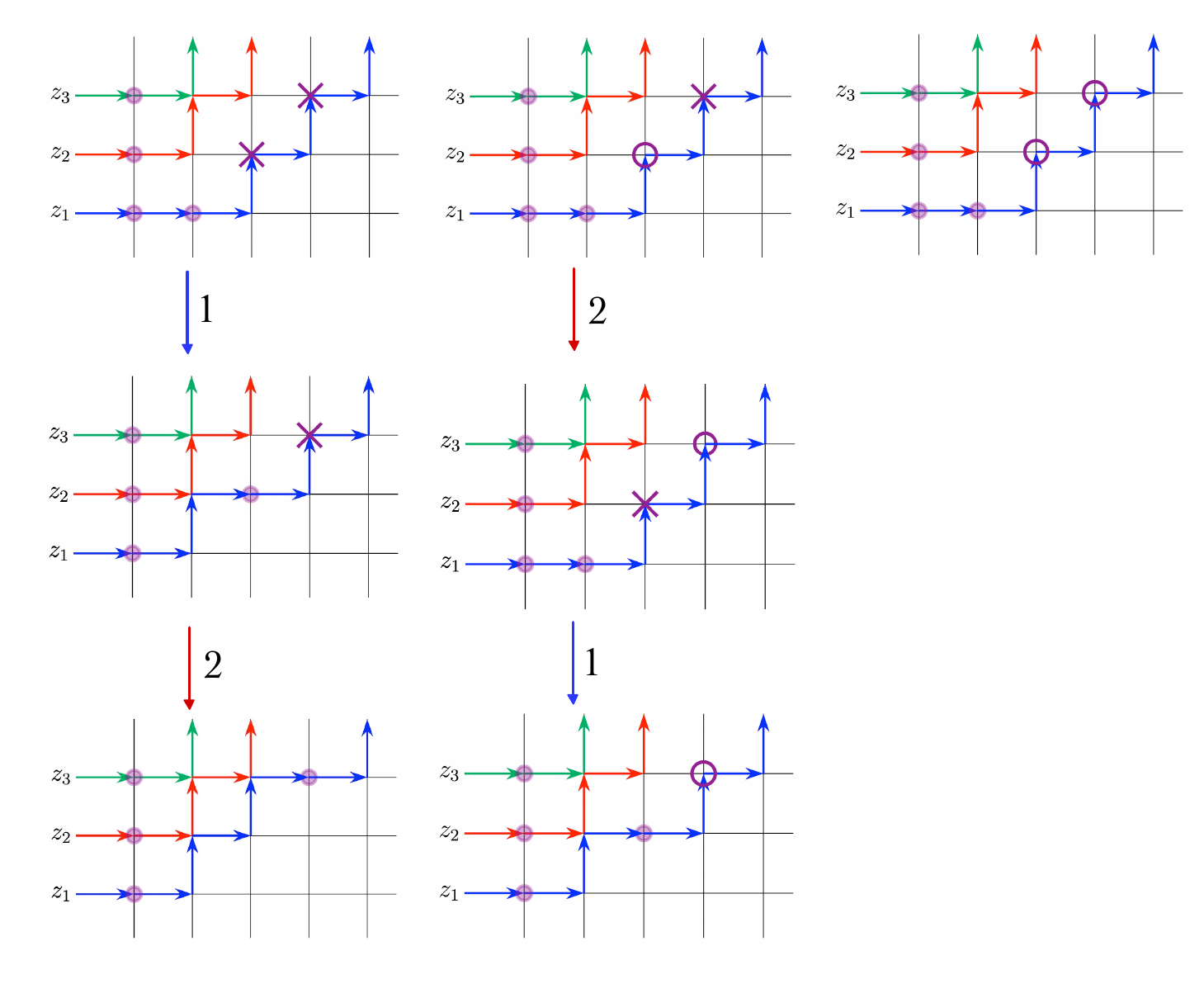}
    \caption{The crystal on $\decoratedstate_{(2,1,1)}$ for $n=3$.}
    \label{fig:lattice-crystal-example}
    \commentAlt{Figure~\ref{fig:lattice-crystal-example}. The crystal contains three connected components corresponding to the tableaux in Figure~\ref{fig:svt-crystal-example}.}
\end{figure}

\bibliography{refs}

@article{Buciumas_2020,
    AUTHOR = {Buciumas, Valentin and Scrimshaw, Travis and Weber, Katherine},
     TITLE = {Colored five-vertex models and {L}ascoux polynomials and
              atoms},
   JOURNAL = {J. Lond. Math. Soc. (2)},
  FJOURNAL = {Journal of the London Mathematical Society. Second Series},
    VOLUME = {102},
      YEAR = {2020},
    NUMBER = {3},
     PAGES = {1047--1066},
      ISSN = {0024-6107,1469-7750},
   MRCLASS = {05A19 (05E05 14M15 82B23)},
  MRNUMBER = {4186121},
MRREVIEWER = {Eric\ S.\ Egge},
       DOI = {10.1112/jlms.12347},
       URL = {https://doi.org/10.1112/jlms.12347},
}

@article{Monical_2020,
    AUTHOR = {Monical, Cara and Pechenik, Oliver and Scrimshaw, Travis},
     TITLE = {Crystal structures for symmetric {G}rothendieck polynomials},
   JOURNAL = {Transform. Groups},
  FJOURNAL = {Transformation Groups},
    VOLUME = {26},
      YEAR = {2021},
    NUMBER = {3},
     PAGES = {1025--1075},
      ISSN = {1083-4362,1531-586X},
   MRCLASS = {05E05 (14N15 19D99)},
  MRNUMBER = {4309558},
       DOI = {10.1007/s00031-020-09623-y},
       URL = {https://doi.org/10.1007/s00031-020-09623-y},
}

@article{Buch_2002,
    AUTHOR = {Buch, Anders Skovsted},
     TITLE = {A {L}ittlewood-{R}ichardson rule for the {$K$}-theory of
              {G}rassmannians},
   JOURNAL = {Acta Math.},
  FJOURNAL = {Acta Mathematica},
    VOLUME = {189},
      YEAR = {2002},
    NUMBER = {1},
     PAGES = {37--78},
      ISSN = {0001-5962,1871-2509},
   MRCLASS = {14M15 (05E05 05E15)},
  MRNUMBER = {1946917},
MRREVIEWER = {Frank\ Sottile},
       DOI = {10.1007/BF02392644},
       URL = {https://doi.org/10.1007/BF02392644},
}

@article{Motegi_2013,
    AUTHOR = {Motegi, Kohei and Sakai, Kazumitsu},
     TITLE = {Vertex models, {TASEP} and {G}rothendieck polynomials},
   JOURNAL = {J. Phys. A},
  FJOURNAL = {Journal of Physics. A. Mathematical and Theoretical},
    VOLUME = {46},
      YEAR = {2013},
    NUMBER = {35},
     PAGES = {355201, 26},
      ISSN = {1751-8113,1751-8121},
   MRCLASS = {81R12},
  MRNUMBER = {3100873},
       DOI = {10.1088/1751-8113/46/35/355201},
       URL = {https://doi.org/10.1088/1751-8113/46/35/355201},
}

@article{Len_00,
    AUTHOR = {Lenart, Cristian},
     TITLE = {Combinatorial aspects of the {$K$}-theory of {G}rassmannians},
   JOURNAL = {Ann. Comb.},
  FJOURNAL = {Annals of Combinatorics},
    VOLUME = {4},
      YEAR = {2000},
    NUMBER = {1},
     PAGES = {67--82},
      ISSN = {0218-0006,0219-3094},
   MRCLASS = {05E15 (14M15)},
  MRNUMBER = {1763950},
MRREVIEWER = {Ang\`ele\ M.\ Hamel},
       DOI = {10.1007/PL00001276},
       URL = {https://doi.org/10.1007/PL00001276},
}

@article{Bandlow_2012,
    AUTHOR = {Bandlow, Jason and Morse, Jennifer},
     TITLE = {Combinatorial expansions in {$K$}-theoretic bases},
   JOURNAL = {Electron. J. Combin.},
  FJOURNAL = {Electronic Journal of Combinatorics},
    VOLUME = {19},
      YEAR = {2012},
    NUMBER = {4},
     PAGES = {Paper 39, 27},
      ISSN = {1077-8926},
   MRCLASS = {05E05 (05A19 19E20)},
  MRNUMBER = {3007174},
       DOI = {10.37236/2320},
       URL = {https://doi.org/10.37236/2320},
}

@article{Reiner_2018,
title = {Poset edge densities, nearly reduced words, and barely set-valued tableaux},
journal = {Journal of Combinatorial Theory, Series A},
volume = {158},
pages = {66-125},
year = {2018},
issn = {0097-3165},
doi = {https://doi.org/10.1016/j.jcta.2018.03.010},
url = {https://www.sciencedirect.com/science/article/pii/S0097316518300372},
author = {Victor Reiner and Bridget Eileen Tenner and Alexander Yong}
}

@article{Patrias_2016,
    AUTHOR = {Patrias, Rebecca},
     TITLE = {Antipode formulas for some combinatorial {H}opf algebras},
   JOURNAL = {Electron. J. Combin.},
  FJOURNAL = {Electronic Journal of Combinatorics},
    VOLUME = {23},
      YEAR = {2016},
    NUMBER = {4},
     PAGES = {Paper 4.30, 32},
      ISSN = {1077-8926},
   MRCLASS = {05E05 (06A07 16T30)},
  MRNUMBER = {3604788},
MRREVIEWER = {Philip\ B.\ Zhang},
       DOI = {10.37236/5949},
       URL = {https://doi.org/10.37236/5949},
}

@article{Chan_2021,
    AUTHOR = {Chan, Melody and Pflueger, Nathan},
     TITLE = {Combinatorial relations on skew {S}chur and skew stable
              {G}rothendieck polynomials},
   JOURNAL = {Algebr. Comb.},
  FJOURNAL = {Algebraic Combinatorics},
    VOLUME = {4},
      YEAR = {2021},
    NUMBER = {1},
     PAGES = {175--188},
      ISSN = {2589-5486},
   MRCLASS = {05E05},
  MRNUMBER = {4226561},
MRREVIEWER = {Sho\ Matsumoto},
       DOI = {10.5802/alco.144},
       URL = {https://doi.org/10.5802/alco.144},
}

@article{pan_uncrowding_2022,
    AUTHOR = {Pan, Jianping and Pappe, Joseph and Poh, Wencin and Schilling,
              Anne},
     TITLE = {Uncrowding algorithm for hook-valued tableaux},
   JOURNAL = {Ann. Comb.},
  FJOURNAL = {Annals of Combinatorics},
    VOLUME = {26},
      YEAR = {2022},
    NUMBER = {1},
     PAGES = {261--301},
      ISSN = {0218-0006,0219-3094},
   MRCLASS = {05E05 (05E10 14N10 14N15 20G42)},
  MRNUMBER = {4407062},
       DOI = {10.1007/s00026-022-00567-6},
       URL = {https://doi.org/10.1007/s00026-022-00567-6},
}

@article {LS.1982,
    AUTHOR = {Lascoux, Alain and Sch\"{u}tzenberger, Marcel-Paul},
     TITLE = {Structure de {H}opf de l'anneau de cohomologie et de l'anneau
              de {G}rothendieck d'une vari\'{e}t\'{e} de drapeaux},
   JOURNAL = {C. R. Acad. Sci. Paris S\'{e}r. I Math.},
  FJOURNAL = {Comptes Rendus des S\'{e}ances de l'Acad\'{e}mie des Sciences. S\'{e}rie
              I. Math\'{e}matique},
    VOLUME = {295},
      YEAR = {1982},
    NUMBER = {11},
     PAGES = {629--633},
      ISSN = {0249-6291},
   MRCLASS = {14M17},
  MRNUMBER = {686357},
}

@incollection {FK.1994,
    AUTHOR = {Fomin, Sergey and Kirillov, Anatol N.},
     TITLE = {Grothendieck polynomials and the {Y}ang-{B}axter equation},
 BOOKTITLE = {Formal power series and algebraic combinatorics/{S}\'{e}ries
              formelles et combinatoire alg\'{e}brique},
     PAGES = {183--189},
 PUBLISHER = {DIMACS, Piscataway, NJ},
      YEAR = {1994},
   MRCLASS = {05E05 (14N15)},
  MRNUMBER = {2307216},
}

@article {Yeliussizov.2017,
    AUTHOR = {Yeliussizov, Damir},
     TITLE = {Duality and deformations of stable {G}rothendieck polynomials},
   JOURNAL = {J. Algebraic Combin.},
  FJOURNAL = {Journal of Algebraic Combinatorics. An International Journal},
    VOLUME = {45},
      YEAR = {2017},
    NUMBER = {1},
     PAGES = {295--344},
      ISSN = {0925-9899},
   MRCLASS = {05E05 (05A17)},
  MRNUMBER = {3591379},
MRREVIEWER = {Michael Xinxin Zhong},
       DOI = {10.1007/s10801-016-0708-4},
       URL = {https://doi.org/10.1007/s10801-016-0708-4},
}

@article {HJKSS1,
    AUTHOR = {Hwang, Byung-Hak and Jang, Jihyeug and Kim, Jang Soo and Song,
              Minho and Song, U-Keun},
     TITLE = {Refined canonical stable {G}rothendieck polynomials and their
              duals, {P}art 1},
   JOURNAL = {Adv. Math.},
  FJOURNAL = {Advances in Mathematics},
    VOLUME = {446},
      YEAR = {2024},
     PAGES = {Paper No. 109670, 42},
      ISSN = {0001-8708,1090-2082},
   MRCLASS = {05E05 (14M15)},
  MRNUMBER = {4735116},
       DOI = {10.1016/j.aim.2024.109670},
       URL = {https://doi.org/10.1016/j.aim.2024.109670},
}

@article {HJKSS2,
    AUTHOR = {Hwang, Byung-Hak and Jang, Jihyeug and Kim, Jang Soo and Song,
              Minho and Song, U-Keun},
     TITLE = {Refined canonical stable {G}rothendieck polynomials and their
              duals, {P}art 2},
   JOURNAL = {European J. Combin.},
  FJOURNAL = {European Journal of Combinatorics},
    VOLUME = {127},
      YEAR = {2025},
     PAGES = {Paper No. 104166, 34},
      ISSN = {0195-6698,1095-9971},
   MRCLASS = {05E05},
  MRNUMBER = {4899734},
       DOI = {10.1016/j.ejc.2025.104166},
       URL = {https://doi.org/10.1016/j.ejc.2025.104166},
}

@article {Iwao2024,
    AUTHOR = {Iwao, Shinsuke and Motegi, Kohei and Scrimshaw, Travis},
     TITLE = {{Free fermions and canonical {G}}rothendieck polynomials},
   JOURNAL = {Algebr. Comb.},
  FJOURNAL = {Algebraic Combinatorics},
    VOLUME = {7},
      YEAR = {2024},
    NUMBER = {1},
     PAGES = {245--274},
      ISSN = {2589-5486},
   MRCLASS = {05E05 (82B23)},
  MRNUMBER = {4715539},
       DOI = {10.5802/alco.332},
       URL = {https://doi.org/10.5802/alco.332},
}

@article {JKPPS.2026,
    AUTHOR = {Jang, Jihyeug and Kim, Jang Soo and Pan, Jianping and Pappe,
              Joseph and Schilling, Anne},
     TITLE = {Hook-valued tableau uncrowding and tableau switching},
   JOURNAL = {SIAM J. Discrete Math.},
  FJOURNAL = {SIAM Journal on Discrete Mathematics},
    VOLUME = {40},
      YEAR = {2026},
    NUMBER = {1},
     PAGES = {32--51},
      ISSN = {0895-4801,1095-7146},
   MRCLASS = {05E10 (14N15)},
  MRNUMBER = {5011488},
       DOI = {10.1137/24M1708346},
       URL = {https://doi.org/10.1137/24M1708346},
}

@article {BFHTW.2023,
    AUTHOR = {Brubaker, Ben and Frechette, Claire and Hardt, Andrew and
              Tibor, Emily and Weber, Katherine},
     TITLE = {Frozen pipes: lattice models for {G}rothendieck polynomials},
   JOURNAL = {Algebr. Comb.},
  FJOURNAL = {Algebraic Combinatorics},
    VOLUME = {6},
      YEAR = {2023},
    NUMBER = {3},
     PAGES = {789--833},
      ISSN = {2589-5486},
   MRCLASS = {05E05 (05E14 14C35)},
  MRNUMBER = {4614163},
MRREVIEWER = {John\ Machacek},
       DOI = {10.5802/alco.277},
       URL = {https://doi.org/10.5802/alco.277},
}

@article {BuciumasScrimshawGrothendieck,
    AUTHOR = {Buciumas, Valentin and Scrimshaw, Travis},
     TITLE = {Double {G}rothendieck polynomials and colored lattice models},
   JOURNAL = {Int. Math. Res. Not. IMRN},
  FJOURNAL = {International Mathematics Research Notices. IMRN},
      YEAR = {2022},
    NUMBER = {10},
     PAGES = {7231--7258},
      ISSN = {1073-7928,1687-0247},
   MRCLASS = {14M15 (16T25 19L47)},
  MRNUMBER = {4418706},
MRREVIEWER = {Eric\ S.\ Egge},
       DOI = {10.1093/imrn/rnaa327},
       URL = {https://doi-org.stanford.idm.oclc.org/10.1093/imrn/rnaa327},
}

@article {WeigandtASM,
    AUTHOR = {Weigandt, Anna},
     TITLE = {Bumpless pipe dreams and alternating sign matrices},
   JOURNAL = {J. Combin. Theory Ser. A},
  FJOURNAL = {Journal of Combinatorial Theory. Series A},
    VOLUME = {182},
      YEAR = {2021},
     PAGES = {Paper No. 105470, 52},
      ISSN = {0097-3165,1096-0899},
   MRCLASS = {05E14 (05E05 14N15 19D99)},
  MRNUMBER = {4258766},
MRREVIEWER = {Edward\ E.\ Allen},
       DOI = {10.1016/j.jcta.2021.105470}
}

@article {BergeronBilley.1993,
    AUTHOR = {Bergeron, Nantel and Billey, Sara},
     TITLE = {R{C}-graphs and {S}chubert polynomials},
   JOURNAL = {Experiment. Math.},
  FJOURNAL = {Experimental Mathematics},
    VOLUME = {2},
      YEAR = {1993},
    NUMBER = {4},
     PAGES = {257--269},
      ISSN = {1058-6458,1944-950X},
   MRCLASS = {05E99 (05E05 14M15 20C30)},
  MRNUMBER = {1281474},
MRREVIEWER = {Axel\ Kohnert},
       URL = {http://projecteuclid.org/euclid.em/1048516036},
}

@article {KM.2005,
    AUTHOR = {Knutson, Allen and Miller, Ezra},
     TITLE = {Gr\"obner geometry of {S}chubert polynomials},
   JOURNAL = {Ann. of Math. (2)},
  FJOURNAL = {Annals of Mathematics. Second Series},
    VOLUME = {161},
      YEAR = {2005},
    NUMBER = {3},
     PAGES = {1245--1318},
      ISSN = {0003-486X,1939-8980},
   MRCLASS = {05E15 (13C40 13F55 13P10 14M15 14N15)},
  MRNUMBER = {2180402},
MRREVIEWER = {Harry\ Tamvakis},
       DOI = {10.4007/annals.2005.161.1245},
       URL = {https://doi.org/10.4007/annals.2005.161.1245},
}

@article {LLS.2021,
    AUTHOR = {Lam, Thomas and Lee, Seung Jin and Shimozono, Mark},
     TITLE = {Back stable {S}chubert calculus},
   JOURNAL = {Compos. Math.},
  FJOURNAL = {Compositio Mathematica},
    VOLUME = {157},
      YEAR = {2021},
    NUMBER = {5},
     PAGES = {883--962},
      ISSN = {0010-437X,1570-5846},
   MRCLASS = {14M15 (05E05)},
  MRNUMBER = {4252201},
MRREVIEWER = {Praise\ Adeyemo},
       DOI = {10.1112/S0010437X21007028},
       URL = {https://doi.org/10.1112/S0010437X21007028},
}

@article {Lenart.2004,
    AUTHOR = {Lenart, Cristian},
     TITLE = {A unified approach to combinatorial formulas for {S}chubert
              polynomials},
   JOURNAL = {J. Algebraic Combin.},
  FJOURNAL = {Journal of Algebraic Combinatorics. An International Journal},
    VOLUME = {20},
      YEAR = {2004},
    NUMBER = {3},
     PAGES = {263--299},
      ISSN = {0925-9899,1572-9192},
   MRCLASS = {05E15 (05E10 14N15)},
  MRNUMBER = {2106961},
MRREVIEWER = {Gregory\ S.\ Warrington},
       DOI = {10.1023/B:JACO.0000048515.00922.47},
       URL = {https://doi.org/10.1023/B:JACO.0000048515.00922.47},
}

@Article{GMS.2024a,
author={Gold, Sarah
and Mili{\'{c}}evi{\'{c}}, Elizabeth
and Sun, Yuxuan},
title={Crystal Chute Moves on Pipe Dreams},
journal={Annals of Combinatorics},
year={2025},
month={Mar},
day={12},
issn={0219-3094},
doi={10.1007/s00026-025-00747-0},
url={https://doi.org/10.1007/s00026-025-00747-0}
}

@article {MS.2016,
    AUTHOR = {Morse, Jennifer and Schilling, Anne},
     TITLE = {Crystal approach to affine {S}chubert calculus},
   JOURNAL = {Int. Math. Res. Not. IMRN},
  FJOURNAL = {International Mathematics Research Notices. IMRN},
      YEAR = {2016},
    NUMBER = {8},
     PAGES = {2239--2294},
      ISSN = {1073-7928,1687-0247},
   MRCLASS = {14M15 (14N35)},
  MRNUMBER = {3519114},
MRREVIEWER = {Pierre-Emmanuel\ Chaput},
       DOI = {10.1093/imrn/rnv194},
       URL = {https://doi.org/10.1093/imrn/rnv194},
}

@article{MPPS.2020,
    title={A crystal on decreasing factorizations in the $0$-{H}ecke monoid},
    volume={27},
    url={https://www.combinatorics.org/ojs/index.php/eljc/article/view/v27i2p29},
    DOI={10.37236/9168},
    number={2},
    journal={The Electronic Journal of Combinatorics},
    author={Morse, Jennifer and Pan, Jianping and Poh, Wencin and Schilling, Anne},
    year={2020},
    month={May},
    pages={\#P2.29},
}

@article {MR4552711,
    AUTHOR = {Gunna, Ajeeth and Zinn-Justin, Paul},
     TITLE = {Vertex models for {C}anonical {G}rothendieck polynomials and
              their duals},
   JOURNAL = {Algebr. Comb.},
  FJOURNAL = {Algebraic Combinatorics},
    VOLUME = {6},
      YEAR = {2023},
    NUMBER = {1},
     PAGES = {109--162},
      ISSN = {2589-5486},
   MRCLASS = {05E05},
  MRNUMBER = {4552711},
       DOI = {10.5802/alco.235},
       URL = {https://doi.org/10.5802/alco.235},
}

@article {MP.2022,
    AUTHOR = {Mucciconi, Matteo and Petrov, Leonid},
     TITLE = {Spin {$q$}-{W}hittaker polynomials and deformed quantum
              {T}oda},
   JOURNAL = {Comm. Math. Phys.},
  FJOURNAL = {Communications in Mathematical Physics},
    VOLUME = {389},
      YEAR = {2022},
    NUMBER = {3},
     PAGES = {1331--1416},
      ISSN = {0010-3616,1432-0916},
   MRCLASS = {05E05 (33D52 37K10 82C22)},
  MRNUMBER = {4381175},
MRREVIEWER = {Kohei\ Motegi},
       DOI = {10.1007/s00220-021-04279-5},
       URL = {https://doi.org/10.1007/s00220-021-04279-5},
}

@article {GK.2017,
    AUTHOR = {Gorbounov, Vassily and Korff, Christian},
     TITLE = {Quantum integrability and generalised quantum {S}chubert
              calculus},
   JOURNAL = {Adv. Math.},
  FJOURNAL = {Advances in Mathematics},
    VOLUME = {313},
      YEAR = {2017},
     PAGES = {282--356},
      ISSN = {0001-8708,1090-2082},
   MRCLASS = {14M15 (05E05 14F43 19L47 55N20 55N22 82B23)},
  MRNUMBER = {3649227},
MRREVIEWER = {Xin\ Fang},
       DOI = {10.1016/j.aim.2017.03.030},
       URL = {https://doi.org/10.1016/j.aim.2017.03.030},
}

@article {MS.2014,
    AUTHOR = {Motegi, Kohei and Sakai, Kazumitsu},
     TITLE = {{$K$}-theoretic boson-fermion correspondence and melting
              crystals},
   JOURNAL = {J. Phys. A},
  FJOURNAL = {Journal of Physics. A. Mathematical and Theoretical},
    VOLUME = {47},
      YEAR = {2014},
    NUMBER = {44},
     PAGES = {445202, 30},
      ISSN = {1751-8113,1751-8121},
   MRCLASS = {81V35},
  MRNUMBER = {3270564},
MRREVIEWER = {Dmitry\ V.\ Artamonov},
       DOI = {10.1088/1751-8113/47/44/445202},
       URL = {https://doi.org/10.1088/1751-8113/47/44/445202},
}

@book {BumpSchilling.2017,
    AUTHOR = {Bump, Daniel and Schilling, Anne},
     TITLE = {Crystal bases},
      NOTE = {Representations and combinatorics},
 PUBLISHER = {World Scientific Publishing Co. Pte. Ltd., Hackensack, NJ},
      YEAR = {2017},
     PAGES = {xii+279},
      ISBN = {978-981-4733-44-1},
   MRCLASS = {05-01 (05E10 14T05 17B10)},
  MRNUMBER = {3642318},
       DOI = {10.1142/9876},
       URL = {https://doi.org/10.1142/9876},
}

@article {GGL.2016,
    AUTHOR = {Galashin, Pavel and Grinberg, Darij and Liu, Gaku},
     TITLE = {Refined dual stable {G}rothendieck polynomials and generalized
              {B}ender-{K}nuth involutions},
   JOURNAL = {Electron. J. Combin.},
  FJOURNAL = {Electronic Journal of Combinatorics},
    VOLUME = {23},
      YEAR = {2016},
    NUMBER = {3},
     PAGES = {Paper 3.14, 28},
      ISSN = {1077-8926},
   MRCLASS = {05E05 (05A17)},
  MRNUMBER = {3558051},
MRREVIEWER = {Eric\ S.\ Egge},
       DOI = {10.37236/5737},
       URL = {https://doi.org/10.37236/5737},
}

@article {BKSTY.2008,
    AUTHOR = {Buch, Anders Skovsted and Kresch, Andrew and Shimozono, Mark
              and Tamvakis, Harry and Yong, Alexander},
     TITLE = {Stable {G}rothendieck polynomials and {$K$}-theoretic factor
              sequences},
   JOURNAL = {Math. Ann.},
  FJOURNAL = {Mathematische Annalen},
    VOLUME = {340},
      YEAR = {2008},
    NUMBER = {2},
     PAGES = {359--382},
      ISSN = {0025-5831,1432-1807},
   MRCLASS = {05E15 (05E05 05E10 14M15)},
  MRNUMBER = {2368984},
MRREVIEWER = {Frank\ Sottile},
       DOI = {10.1007/s00208-007-0155-6},
       URL = {https://doi.org/10.1007/s00208-007-0155-6},
}

@article {BBBG.2021,
    AUTHOR = {Brubaker, Ben and Buciumas, Valentin and Bump, Daniel and
              Gustafsson, Henrik P. A.},
     TITLE = {Colored five-vertex models and {D}emazure atoms},
   JOURNAL = {J. Combin. Theory Ser. A},
  FJOURNAL = {Journal of Combinatorial Theory. Series A},
    VOLUME = {178},
      YEAR = {2021},
     PAGES = {Paper No. 105354, 48},
      ISSN = {0097-3165,1096-0899},
   MRCLASS = {05E05 (05E10 05E16)},
  MRNUMBER = {4165627},
MRREVIEWER = {Laura\ Colmenarejo},
       DOI = {10.1016/j.jcta.2020.105354},
       URL = {https://doi.org/10.1016/j.jcta.2020.105354},
}

@unpublished{Yang.2025,
    AUTHOR = {Yang, Yingzi},
    TITLE = {Closed colored models and {D}emazure crystals},
     NOTE = {preprint, arXiv:2512.05479}}
\bibliographystyle{alpha}

\end{document}